\documentclass[11pt]{article}
\usepackage[margin=1in]{geometry}
\usepackage{graphicx}
\usepackage[numbers,sort]{natbib}
\usepackage{amsmath,amsthm,amsfonts,mathrsfs,amssymb,float,color}
\usepackage{graphicx}
\usepackage{amsmath,amsthm,amsfonts,mathrsfs,amssymb,float,color}
\usepackage{graphicx}
\usepackage{verbatim}
\usepackage{ upgreek }
\usepackage[T1]{fontenc}
\usepackage[utf8]{inputenc}
\usepackage{hyperref}
\hypersetup{colorlinks=true, linkcolor=blue, citecolor=red}
\usepackage{epsfig,subfigure,fancybox,balance}
\usepackage{diagbox}
\usepackage{youngtab}

\newcommand{\beq}[1]{\begin{equation} \label{#1}}
\newcommand{\eeq}{\end{equation}}
\newcommand{\bea}{\bed\begin{array}{rl}}
\newcommand{\eea}{\end{array}\eed}
\newcommand{\bed}{\begin{displaymath}}
\newcommand{\eed}{\end{displaymath}}
\newcommand{\barray}{\begin{array}{ll}}
\newcommand{\earray}{\end{array}}
\newcommand{\disp}{\displaystyle}
\newcommand{\ad}{&\!\disp}
\newcommand{\aad}{&\disp}
\newcommand{\al}{\alpha}
\newcommand{\e}{\varepsilon}

\newcommand{\la}{\lambda}

\newcommand{\sg}{\sigma}

\newcommand{\ga}{\gamma}
\newcommand{\Ga}{\Gamma}
\newcommand{\dl}{\delta}
\newcommand{\Dl}{\Delta}
\newcommand{\cd}{(\cdot)}

\newcommand{\sqe}{\sqrt{\e}}

\def\phi{\varphi}
\def\indi{{\bf 1}}
\def\half{\frac{1}{2}}

\newcommand{\CA}{{\mathcal A}}
\newcommand{\CF}{{\mathcal F}}
\newcommand{\CM}{\mathcal{M}}
\newcommand{\CU}{\mathcal{U}}

\newcommand{\CO}{\mathcal{O}}

\newcommand{\CB}{\mathcal{B}}

\newcommand{\CL}{\mathcal{L}}
\newcommand{\CP}{\mathcal{P}}

\newcommand{\CK}{\mathcal{K}}
\newcommand{\CD}{\mathcal{D}}
\newcommand{\CR}{\mathcal{R}}
\newcommand{\CS}{\mathcal{S}}
\newcommand{\CV}{\mathcal{V}}

\newcommand{\CH}{\mathcal{H}}

\newcommand{\CG}{\mathcal{G}}

\newcommand{\EE}{{\mathbb E}}
\newcommand{\PP}{{\mathbb P}}

\newcommand{\NN}{{\mathbb N}}
\newcommand{\rr}{{\mathbb R}}

\newcommand{\TT}{{\mathbb T}}
\newcommand{\MM}{{\mathbb M}}
\newcommand{\LL}{{\mathbb L}}
\newcommand{\HH}{{\mathbb H}}

\newcommand{\bv}{\mathbf{v}}
\newcommand{\lbar}{\overline}
\newcommand{\wdt}{\widetilde}
\newcommand{\wdh}{\widehat}

\numberwithin{equation}{section}

\newtheorem{thm}{Theorem}[section]
\newtheorem{lem}[thm]{Lemma}
\newtheorem{defn}[thm]{Definition}
\newtheorem{cor}[thm]{Corollary}
\newtheorem{prop}[thm]{Proposition}
\newtheorem{rem}[thm]{Remark}

\newcommand{\thmref}[1]{Theorem~{\rm \ref{#1}}}
\newcommand{\lemref}[1]{Lemma~{\rm \ref{#1}}}

\newcommand{\corref}[1]{Corollary~{\rm \ref{#1}}}
\newcommand{\propref}[1]{Proposition~{\rm \ref{#1}}}

\newcommand{\remref}[1]{Remark~{\rm \ref{#1}}}

\newcommand{\secref}[1]{Section~{\rm \ref{#1}}}

\newcommand{\etae}{\eta^{\e,u^\e,\phi^\e}}
\newcommand{\Xe}{X^{\e,u^\e,\phi^\e}}
\newcommand{\Ye}{Y^{\e,u^\e,\phi^\e}}
\newcommand{\trace}{\text{Tr}}

\begin{document}
\title{Moderate Deviation Principles for Stochastic Differential Equations in Fast-Varying Markovian Environment
}
\date{}
\author{Hongjiang Qian \thanks{Department of Mathematics and Statistics, Auburn University, Auburn, AL 36849, United States. E-mail: hjqian.math@gmail.com}}
\maketitle
\begin{abstract}
In this paper, we proved moderate deviation principles for a fully coupled two-time-scale stochastic systems, where the slow process is given by stochastic differential equations with small noise, while the fast process is a rapidly changing purely jump process on finite state space. The system is fully coupled in that the drift and diffusion coefficients of the slow process, as well as the jump distribution of the fast process, depend on states of both processes. Moreover, the diffusion component in the slow process can be degenerate. Our approach is based on the combination of the weak convergence method from [A. Budhiraja, P. Dupuis, and A. Ganguly, Electron. J. Probab. 23 (2018), pp. 1-33; Ann. Probab. 44 (2016), pp. 1723-1775] with Poisson equation for the fast-varying purely jump process.

\medskip


\medskip


\end{abstract}

\section{Introduction}
In this paper, we study moderate deviation principles (MDPs) for a fully coupled two-time-scale stochastic system $(X^\e\cd, Y^\e\cd)$. The $X^\e$ is a $d$-dimensional stochastic process given as the solution of the following stochastic differential equations:
\beq{sde}
dX^\e(t) = b(X^\e(t), Y^\e(t))dt + \sqrt{\e} \sigma(X^\e(t), Y^\e(t))dW(t), \ X^\e(0) = x_0, Y^\e(0)=y_0,
\eeq
where $\e$ is the small parameter, $W\cd$ represents a $d$-dimensional Brownian motion. The functions $b: \rr^d \times \LL \rightarrow \rr^d$ and $\sigma: \mathbb{R}^d \times \LL \rightarrow \rr^{d \times d}$ are continuous functions satisfying suitable conditions, which will be specified later.

The $Y^\e$ is a fast-varying jump process defined over a finite state space $\LL = \{1, \dots, |\LL|\}$, and its transition probabilities are determined by $q: \LL\times \LL \times \rr^d \to [0,1]$, which depends on both $X^\e$ and $Y^\e$. Specifically, the transition rates of $Y^\e$ are defined as:
\bea
\PP(Y^\e(t+\Dl)=j |Y^\e(t)=i, X^\e(t)=x)=\left\{\barray
\e^{-1} q_{ij}(x)\Dl+o(\Dl),\ad \text{ if } j \neq i, \\
1+ \e^{-1} q_{ii}(x) \Dl + o(\Dl),\ad \text{ otherwise}.
\earray\right.
\eea
Here $q_{ij}(x)$ is the shorthand of $q(i,j,x)$ for $x\in \rr^d$, and $Q(x):=(q_{ij}(x))_{i,j\in \LL}$ is the generator of a Markov chain. The process $Y^\e$ is also called state-dependent regime-switching process and is widely used to model the discrete random events in systems such as manufacturing, product planing, queuing networks, Monte Carlo simulations, and wireless communications, and so on. We refer to the work \cite{YZ10} for its comprehensive study on switching process.

A key feature of system \eqref{sde} is its combination of continuous dynamics (via $X^\e$) and discrete events (via $Y^\e$). The process $X^\e$ evolves on a slower timescale compared to $Y^\e$, which oscillates rapidly due to the timescale factor $\e^{-1}$. Under suitable conditions, the fast process $Y^\e$, for fixed $x\in \rr^d$, admits a unique invariant measure $\mu^x=\{\mu_i^x\}_{i\in \LL}$, leading to the well-known averaging principle. That is, $X^\e$ converges to the solution of an averaged equation:
\beq{bar-X}
d\bar X(t) = \bar{b}( 
\bar X(t))dt, \quad  \bar X(t) = x_0,
\eeq
in various sense, where $\bar b(x):=\sum_{i\in \LL} b(x,i) \mu_i^x$; see  recent work  \cite{MS24} on averaging principle for two time-scale regime-switching diffusions. For $x\in \rr^d$ and $i,j \in \LL$, we sometimes write $b(x,i),\sg(x,i), q(i,j,x)$ as $b_i(x), \sg_i(x), q_{ij}(x)$ for notation convenience. For each $(x,i,j)\in \rr^d\times \LL\times \LL$, we let $
E_{ij}(x)$ be the interval $[0,q_{ij}(x)]$. 
Define
\bea\ad \zeta := \sup_{(x,i,j)\in \rr^d \times \LL \times \LL} |q_{ij}(x)|+1<\infty.
\eea
We denote by $\la = \la_\zeta$ the Lebesgue measure on the space $([0,\zeta], \CB([0,\zeta]))$.

In this paper, we focus on deviations of $X^\e(t)$ from the averaged solution $\bar X(t)$ as $\e\to 0$, specifically on analyzing the asymptotic behavior of trajectory:
\beq{eta}
\eta^\e(t) = \frac{X^\e(t) - \bar X(t)}{\sqrt{\e} h(\e)}, \quad t\in [0,T],
\eeq
as $\e \to 0$, where $h(\e)$ is the deviation scale satisfying 
\beq{dev-h}
h(\e)\to +\infty, \quad \sqrt{\e} h(\e)\to 0,\text{ as } \e \to 0.
\eeq

The result concerning to decay rate of probabilities for $\eta^\e$ is called moderate deviation principles. It bridges the gap between central limit theorems (CLTs) and large deviation principles (LDPs). If $h(\e)=1$, the limiting behavior of $\eta^\e(t)$ falls in the domain of CLTs; see \cite{SX25}. For $h(\e)=\e^{-1/2}$, the limiting behavior of $\eta^\e(t)$ leads to large deviation principles in \cite{BDG18}, where the authors established the large deviation principle for ${X^\e}$ in $C([0,T];\rr^d)$ as $\e \rightarrow 0$, using the weak convergence method \cite{DE97}. Define the function $\ell:[0,\infty)\to [0,\infty)$ by $\ell(x)=x\log x-x +1$. Let
\bea
\TT &\!\!\! :=\{(i,j) \in \LL\times \LL: q_{ij}(x)>0 \text{ for some } x \in \rr^d\}, \\
\CR &\!\!\! :=\{\phi=(\phi_{ij})_{(i,j)\in \TT}: \phi_{ij}: [0,T]\times [0,\zeta] \to \rr_+ \text{ is a measurable map}, (i,j)\in \TT \}.
\eea
They demonstrated ${X^\e}$ adheres to LDPs with speed $\e^{-1}$ and rate function $J(\cdot)$ given by
\beq{I-LDP}\barray\ad 
\!\!\!\!\!\!\!\! J(\xi)=\!\!\!\! \inf_{(u,\phi,\pi)\in \CV(\xi)} \bigg\{\sum_{i\in \LL} \half\int_0^T \|u_i(s)\|^2 \pi_i(s) ds +\sum_{(i,j)\in \TT}\int_{[0,T]\times [0,\zeta]} \ell(\phi_{ij}(s,z))\pi_i(s)\la_\zeta(dz)ds\bigg\},
\earray\eeq
where $\CV(\xi)$ is the collection of all $
\{u=(u_i), \phi=(\phi_{ij}), \pi=(\pi_i)\}\in (\MM([0,T];\rr^d))^{|\LL|} \times \CR \times \MM([0,T];\CP(\LL))$, such that $\int_0^T \|u_i(s)\|^2 \pi_i(s) ds <\infty$, and for each $i\in \LL$, 
\bea\ad 
\xi(t)=x_0+\sum_{j\in \LL}\int_0^t b_j(\xi(s))\pi_j(s)ds +  \sum_{j\in \LL} \int_0^t \sg_j(\xi(s)) u_j(s)\pi_j(s)ds, \; t\in [0,T],
\eea
and 
\beq{ldp-Ga}
\sum_{i\in \LL} \pi_i(s)\Ga_{ij}^{\phi_{i,\cdot}(s,\cdot)}(\xi(s))=0, \quad \text{ for a.e. }  s\in [0,T] \text{ and } j\in \LL,
\eeq
where, for $\psi=\{\psi_j\}_{j\in \LL}$ with each $\psi_j:[0,\zeta] \to \rr_+$ measurable, the $\Ga_{ij}^\psi$ is defined as 
\beq{Ga}\barray
\Ga_{ij}^\psi(x)=\left\{\barray
\int_{E_{ij}(x)} \psi_j(z)\la_\zeta(dz), \ad \quad i \neq j,\\
-\sum_{y:y\neq j} \Ga_{jy}^\psi(x), \ad \quad i=j.
\earray\right.
\earray\eeq
Here $\MM([0,T];\rr^d), \MM([0,T];\CP(\LL))$ denote the space of measurable maps from $[0,T]$ to $\rr^d$ and from $[0,T]$ to $\CP(\LL)$, respectively.

The primary objective of this paper is to investigate moderate deviation principles (MDPs) for the process $X^\e$ (or say large deviation principle for $\eta^\e$), focusing on the intermediate regime between CLT and LDP.
To the best of our knowledge, no prior results have been established for the MDP of $X^\e$ in a fast-varying purely-jump random environment on a finite state space. We address this gap in current paper. For a general theory of large deviations, we refer to the works of \cite{DZ09,DE97,FK06}.

The study of large and moderate deviation principles for two-time-scale (or multi-time-scale) stochastic dynamics has attracted considerable interest in the research community recent years. Early work on moderate deviations for diffusion processes can be traced back to \cite{BF77,Fre78}, which considered cases where the slow motion has  no diffusion coefficients.  In \cite{Gui03}, Guillin explored the MDPs for the diffusion process $X^\e$ of the form \eqref{sde}, where the fast-varying environment $Y^\e$ is modeled as an exponentially ergodic Markov process, independent of the Brownian motion $W\cd$. Such work builds on \cite{Gui01}, which addresses MDPs for inhomogeneous functionals of exponentially ergodic Markov processes. When the fast-varying process is modeled by a finite-state homogeneous or non-homogeneous Markov chain, MDPs were developed in \cite{HY14} for diffusion processes and in \cite{QY22} for Langevin dynamics. The literature  mentioned above generally assume that the fast-varying process is independent of the slow process.

In contrast, coupled systems have also been investigated a lot. In \cite{MS17}, Morse and Spiliopulos examined MDPs for fully coupled fast-slow diffusions using Poisson equations and weak convergence method. The Poisson equation approach can also be extended to study infinite dimensional fast-slow dynamics, where an elliptic Kolmogorov equation in Hilbert space replaces the role of the Poisson equation. Interested reader are referred to the work of \cite{GSS23} which studied MDPs for fast-slow stochastic reaction-diffusion equations, where the slow process has diffusion coefficients independent of the fast dynamics. The MDPs of fully coupled fast-slow stochastic reaction-diffusion equations still  remain an open problem due to the lack of regularity results for elliptic Kolmogorov equations in infinite dimensional space. For further insights on LDPs and MDPs in  two- or multi-time-scale stochastic dynamics,  readers are referred to \cite{Ver99,Ver00,Lip96,HYZ11,HY14a,HMS16,Puh16,KP17,HSS19} and references therein.

The MDPs for fully coupled stochastic differential equations where the fast process is given by a state-dependent regime-switching process is still unknown. Though we still rely on Poisson equation technique, the discrete feature of fast process makes analysis more difficult and results are different compared to that of fast-slow diffusions in \cite{MS17}. 
The starting point of our approach is the variational representation for the expectation of positive functionals of Brownian motions and Poisson random measures, as developed in \cite{BD98} and \cite{BDM11}. Using this representation, we apply the weak convergence method to analyze asymptotic behavior of the system. One major advantage of this method is that it avoids proving exponential tightness  and calculating  exponential-type probability estimates, which are typically challenging with traditional methods. Instead, one only needs to establish the tightness of a controlled dynamics and analyze limiting behaviors of key quantities. In the context of moderate deviations, the controlled dynamics are described by $(\Xe, \Ye)$, the solution to \eqref{c-sde}. Our primary focus is on the limiting behavior of $\etae$, the controlled version of $\eta^\e$, defined as:
\beq{etae} 
\eta^{\e,u^\e,\phi^\e}(t):=\frac{X^{\e,u^\e,\phi^\e}(t)-\bar X(t)}{\sqe h(\e)},\quad t\in [0,T].
\eeq

Despite various advantages of the weak convergence method, there are still challenges in proving the tightness of the controlled process, particularly in two- or multi-time-scale systems. The first challenge arises from a singular term in the decomposition \eqref{b-eta} of $\etae$:
\bea\ad 
\frac{1}{\sqe h(\e)} \int_0^t \big[ b(\Xe(s),\Ye(s)) - \bar b(\Xe(s)) \big] ds,
\eea
which involves the factor $1/\sqe h(\e)$ that diverges as $\e \to 0$. To address this difficulty, we apply Poisson equation techniques combined with some approximation methods allowing us to use It\^{o} formula to manage the singular term. Compared to the case in \cite{MS17}, where the fast process is a diffusion process, our state space of fast processes is finite and discrete. Furthermore, the diffusion coefficient in the slow process can be degenerate. For the study of Poisson equation concerning to diffusion, we refer to the work \cite{PV01, PV03} and reference therein. Recently, Sun and Xie \cite{SX25} studied the Poisson equation for the fast-varying purely jump process and apply them to investigate averaging principles and CLTs for multi-scale stochastic differential equations with state-dependent switching. One of the motivation of this work comes from \cite{SX25}.

The second challenge is identifying a suitable space in which the centered and normalized control $\{\psi^\e\}$ is tight, where $\psi^\e:=(\phi^\e-1)/\sqe h(\e)$. In case of LDPs for $X^\e$, the boundedness of the control cost $L_{T,2}(\phi^\e)$ in \eqref{LT-g} ensures the tightness of $\phi^\e$ in a suitable space, which guarantees the tightness of the controlled process $\Xe\cd$. However, in the moderate deviations setting, even though we have a stronger bound on the control cost, i.e., $L_{T,2}(\phi^\e)\leq M\e h^2(\e)$ for some constant $M>0$, the  tightness of $\phi^\e$ alone does not guarantee the tightness of $\etae\cd$. Remember in moderate deviations context, the focus is on $\etae$ and we need to  establish the tightness of $\psi^\e$. However, although each $\psi^\e$ is square integrable for fixed $\e>0$, the family $\{\psi^\e\}$ does not admit a uniform $L^2([0,T]\times [0,\zeta])$ bound as $\e\to 0$. The singular coefficient $1/(\sqe h(\e))$ can cause the norms to blow up and prevents  $\{\psi^\e\}_{\e>0}$ from being tight. To resolve this, we use the idea in \cite{BDG16} to split $\psi^\e$ into two terms: one of which lies in a closed ball in $L^2([0,T]\times [0,\zeta])$, while the other approaches zero in a suitable manner. By combining the Poisson equation for fast-varying purely jump processes, the splitting technique in \cite{BDG16}, and some approximation methods, we establish the tightness of $\etae$ and subsequently prove the MDP of $X^\e$ using weak convergence method.

The remainder of the paper is structured as follows: In \secref{sec:pre}, we introduce notation and assumptions, followed by some preliminary results and our main findings. \secref{sec:rate} presents an equivalent form of the rate function and establishes its goodness. In \secref{sec:lim}, we examine the limiting behavior of the controlled process, with a focus on proving its tightness. \secref{sec:up} is dedicated to proving large deviation upper bound for $\eta^\e$, while \secref{sec:low} establishes the corresponding lower bound. In \secref{sec:ext}, we extend our main results from bounded to unbounded drift. Finally, \secref{sec:dis} extends the result to the case where the slow process is a jump diffusion, discusses the main challenges in generalizing the analysis from a finite to a countable state space for the Markov chain, and presents the corresponding result when the fast-varying process belongs to a specific class of jump diffusions in the Euclidean space.

\section{Preliminaries and main results}\label{sec:pre}
In this paper, we use the following notation. For a Polish space $\CS$, let $\CP(\CS)$ denote the space of probability measures on $\CS$, and $\CM_F(\CS)$ denote the space of finite measures on $\CS$, both equipped with the topology of weak convergence. The space of continuous functions from $[0,T]$ to $\CS$, equipped with the uniform topology, is denoted by $C([0,T];\CS)$. We denote the space of continuous and bounded functions on $\CS$ by $C_b(\CS)$. For a finite set $\LL$, $|\LL|$ represents the cardinality of $\LL$, and $\MM(\LL)$ denotes the space of real-valued functions on $\LL$. Given $f \in \MM(\LL)$ and a probability function $r : \LL \to [0,1]$, we use $r(A)$ to denote $\sum_{x\in A}r(x)$ for any subset $A \subset \LL$, and $\int_\LL f(x)r(dx)$ to represent $\sum_{x\in \LL} f(x)r(x)$. The space of Borel maps from $[0,T]$ to $\CS$ is denoted by $\MM([0,T];\CS)$. For a vector $z \in \rr^d$, $z^\top$ denotes its transpose. In the following, we use $K$ to denote a generic constant varying from time to time.

We proceed with the following assumptions.

\noindent{(H1)} For each $i \in \LL$, the function $b(\cdot, i)$ is bounded, continuously differentiable, and has a bounded first-order derivative. Moreover, there exists a constant $L_{0}\in (0,\infty)$ such that for all $i,j\in \LL$ and $x,x'\in \rr^d$, the following Lipschitz continuity holds:
\beq{lip-cond}\barray\ad 
\!\!\!\!\!\!\!\!\! \|b(x,i)-b(x',i)\|+\|\sg(x,i)-\sg(x',i)\|+|q(x,i,j)-q(x',i,j)|\leq L_{0} \|x-x'\|.
\earray\eeq
\noindent{(H2)} The generator  $Q(x)=(q_{ij}(x))_{i,j\in \LL}: \rr^d\to \rr^{\LL\times \LL}$ is measurable function such that for $x\in \rr^d$, $
q_{ij}(x)\geq 0, \; \forall i\neq j\in \LL,\; \sum_{j\in\LL} q_{ij}(x)=0, \forall i\in \LL$.
Moreover, we assume $Q(x)$ is irreducible. That is, for any $x\in\rr^d$, the following equation 
\beq{mu-Q} 
\mu^x Q(x)=0 \text{ and } \sum_{i \in \LL} \mu_i^x=1,
\eeq
has a unique solution $\mu^x=(\mu_1^x,\dots, \mu_{|\LL|}^x)$ with $\mu_i^x>0$ for all $i\in \LL$. The transition probability matrix $P_t^x:=\{P_t^x(i,j)\}_{i,j\in \LL}$ associated with $Q(x)$ is assumed to be exponentially ergodic uniformly in $x$, i.e., there exists $K>0, \varrho>0$ such that
\bea\ad 
\sup_{i\in \LL, x\in \rr^d} \|P_t^x(i,\cdot)-\mu^x\|_{\text{var}} \leq K e^{-\varrho t},\; \forall\; t>0.
\eea
\smallskip

\noindent{(H3)} There exists a constant $K_Q>0$ such that for all $x,x'\in \rr^d$,
\bea\ad 
\|Q(x)-Q(x')\|_{\ell_1}:=\sup_{i\in \LL} \sum_{j\neq i} |q_{ij}(x)-q_{ij}(x')| \leq K_Q \|x-x'\|,
\eea
and
\bea\ad 
\bar \kappa_Q:= \sup_{i\in \LL}\sup_{j\in \LL, j\neq i} \sup_{x\in \rr^d} q_{ij}(x) <\infty \quad  \underline{\kappa}_Q:= \inf_{x\in \rr^d} \min_{(i,j)\in \TT} q_{ij}(x)>0.
\eea
\begin{rem}\rm{
 Under assumptions (H2) and (H3), the invariant measure $\{\mu_i^x\}_{i\in \LL}$ is unique and has a positive lower bound, i.e., 
\beq{und-mu}
\inf_{x\in \rr^d}\min_{j\in \LL}\mu_j^x:=\underline{\mu}>0,
\eeq
see \cite{BDG18} and reference therein.
}
\end{rem}
It is well-known that the dynamics of discrete component or the switching process $Y^\e$ can be described by a stochastic differential equation driven by a finite collection of Poisson random measures (PRMs); see \cite{MY06,Sko09}. For $(i,j)\in \TT$, let $\bar N_{ij}$ be a PRM on $[0,\zeta]\times [0,T]\times \rr^+$ with intensity measure $\la_\zeta \otimes \la_T \otimes \la_\infty$, where $\la_T$ and $\la_\infty$ denote the Lebesgue measure on $[0,T]$ and $\rr^+$, respectively, on some complete filtered probability space $\{\Omega, \CF, \{\CF_t\}, \PP\}$ such that $
\bar N_{ij}(A\times [0,t]\times B) -t \la_\zeta(A)\la_\infty(B)$ 
is a $\{\CF_t\}$-martingale for all $A\in \CB([0,\zeta])$ and $B\in \CB(\rr^+)$ with $\la_\infty(B)<\infty$. It implies that $
N_{ij}^{\e^{-1}} (dz\times dt):= \bar N_{ij}(dz\times dt \times [0,\e^{-1}])$ 
is a PRM on $[0,\zeta]\times [0,T]$ with intensity measure $\e^{-1} \la_\zeta \otimes \la_T$. It can be regarded as a random variable with values  in the space of finite measure $\CM_F([0,\zeta]\times [0,T])$ on $[0,\zeta]\times [0,T]$ endowed with weak convergence topology. We assume $(\bar N_{ij})_{(i,j)\in \TT}$ are mutually independent and there is an $d$-dimensional $\CF_t$-Brownian motion $W$, independent of $\{\bar N_{ij}\}$, on $\{\Omega, \CF, \{\CF_t\}, \PP\}$. Moreover, we assume for $0\leq s\leq t\leq T$,
\bea
\{W(t)-W(s), \bar N_{ij}(A\times (s,t]\times B): A\in \CB([0,\zeta]), B\in \CB(\rr^+), (i,j)\in \TT\}
\eea
is independent of $\CF_s$.

For each $(i,j)\in \TT$, let $\phi_{ij}$ belong to the class of measurable mapping from $[0,T]\times [0,\zeta]\times \Omega \to [0,\infty)$, we define the counting process $N_{ij}^{\phi_{ij}}$ on $[0,\zeta] \times [0,T]$ by
\beq{N-phi}
N_{ij}^{\phi_{ij}}(U\times [0,t))=\int_{U\times [0,t]\times [0,\infty)} \indi_{[0,\phi(z,s)]}(r) \bar N_{ij}(dz ds dr).
\eeq
The above $N_{ij}^{\phi_{ij}}$ can be recognized as a controlled random measure with $\phi_{ij}$ selecting the intensity for points at location $z$ and time $s$ in a possible random but nonanticipating way. If $\phi_{ij}(z,s,\omega)\equiv \vartheta \in (0,\infty)$, we write $N_{ij}^{\phi_{ij}} = N_{ij}^\vartheta$.

The Markov process $(X^\e,Y^\e)=\{X^\e(t),Y^\e(t)\}_{0\leq t\leq T}$ then can be written as the unique pathwise solution of the following fully coupled stochastic differential equations driven by Brownian motions and Poisson random measures:
\beq{re-sde}
\left\{\barray
dX^\e(t)\ad =b(X^\e(t),Y^\e(t))dt+ \sqrt{\e}\sg(X^\e(t),Y^\e(t))dW(t),\quad X^\e(0)=x_0,\\
dY^\e(t) \ad =\!\! \sum_{(i,j)\in \TT} \int_{[0,\zeta]}(j-i) \indi_{\{Y^\e(t-)=i\}} \indi_{E_{ij}(X^\e(t))}(z) N_{ij}^{\e^{-1}}(dz\times dt), Y^\e(0)=y_0.
\earray\right.
\eeq
For the existence and uniqueness of \eqref{re-sde}, we refer to the work \cite[pp. 103-104]{Sko09}, as well as \cite{IW14,JS13}. The uniqueness of the solution implies that for every $\e>0$, there exists a measurable map, denoted by $\bar \CG^\e: C([0,T];\rr^d)\times (\CM_F([0,\zeta]\times [0,T]))^{|\TT|} \to C([0,T];\rr^d)$, such that $
X^\e=\bar \CG^\e(\sqrt{\e}W, \{\e N_{ij}^{\e^{-1}}\})$. Consequently, there exists a measurable map $\CG^\e:C([0,T];\rr^d) \times (\CM_F([0,\zeta] \times [0,T]))^{|\TT|} \to C([0,T];\rr^d)$ such that $\eta^\e=\CG^\e(\sqe W, \{\e N_{ij}^{\e^{-1}}\})$, where $\eta^\e$ is defined in \eqref{eta}. 

To study the LDP for $\eta^\e$ (i.e.,  moderate deviation principle for $X^\e$), we will employ the weak convergence method; see \cite{BD98,BDM11,DE97,QCY25}. To this end, we first give a variational representation for the positive exponential functional of $\eta^\e$. Denote by $\CP\CF$ the predictable $\sg$-field on $[0,T]\times \Omega$ associated with the filtration $\{\CF_t\}_{0\leq t\leq T}$. Let $
\CP\CF[\TT]:=\{\phi =(\phi_{ij})_{(i,j)\in \TT}: \phi_{ij} \text{ is } \CP\CF \otimes \CB([0,\zeta])\setminus \CB[0,\infty)\text{-measurable for all }$ $(i,j)\in\TT\}$. 
For $h:[0,T]\to \rr^d$, we let 
\beq{LT-h}
L_{T,1}(h)=\half \int_0^T \| h(s)\|^2 ds.
\eeq
For $\phi=\{\phi_{ij}\}_{(i,j)\in \TT}$ such that $\phi_{ij}:[0,T]\times [0,\zeta] \to [0,\infty)$, let
\beq{LT-g}
L_{T,2}(\phi)=\sum_{(i,j)\in \TT} \int_{[0,T]\times [0,\zeta]} \ell (\phi_{ij}(s,z))\la_\zeta(dz) ds.
\eeq
For any $\phi \in \CP\CF[\TT]$, the quantity $
L_{T,2}(\phi)=\sum_{(i,j)\in \TT} \int_{[0,T]\times [0,\zeta]} \ell(\phi_{ij}(s,z,\omega))\la_\zeta(dz)ds$ 
is well defined as a $[0,\infty]$-valued random variable. For $n\in \NN$, let $
\CP\CF[\TT]^b = \{\phi\in \CP\CF[\TT]:\exists\, n\, \in \NN, \phi_{ij}(s,z,\omega) \in [n^{-1},n],\forall\, (s,z,\omega)\in [0,T]\times [0,\zeta]\times \Omega, (i,j)\in \TT\}$.

For $\e>0$ and $M<\infty$, we will consider the following space: 
\beq{S-plus}\barray
S_{+,\e}^M\ad :=\{\phi:[0,T]\times [0,\zeta]\to \rr_{+}| L_{T,2}(\phi)\leq M \e h^2(\e)\}, \\
 S_{\e}^M \ad : =\{\psi: [0,T]\times [0,\zeta] \to \rr \, |\, \psi=(\phi-1)/\sqrt{\e}h(\e), \phi\in S_{+,\e}^M\}.
\earray\eeq
Define $
S_2^M =\{h:[0,T]\to \rr^d: L_{T,1}(h)\leq M\}$, $S_{2,\e}^M=\{h:[0,T]\to \rr^d: L_{T,1}(h)\leq M h^2(\e)\}$.
Let $\CA_2^M$ be the set of $h:[0,T]\to \rr$ such that $h(\cdot,\omega)\in S_2^M, \PP$-a.s., and define $\CA_2:=\cup_{M=1}^\infty \CA_2^M$. Similarly, let $\CA_{2,\e}$ be the set of all $h$ such that $h(\cdot,\omega)\in S_{2,\e}^M$, $\PP$-a.s., and define $\CA_{2,\e}=\cup_{M=1}^\infty \CA_{2,\e}^M $. Moreover, let $\CA_{+,\e}^M$ be the set of all $\phi(\cdot,\cdot,\omega)\in \CP\CF[\TT]^b$ such that $\phi_{ij}(\cdot,\cdot,\omega) \in S_{+,\e}^M, \PP$-a.s. for each $(i,j)\in \TT$. Define $\CA_{+,\e}:=\cup_{M=1}^\infty \CA_{+,\e}^M$. Likewise, define $\CA_\e^M$ as the set of all $\phi(\cdot,\cdot,\omega)\in \CP\CF[\TT]$ such that $\phi_{ij}(\cdot,\cdot,\omega)\in S_\e^M$, $\PP$-a.s. for each $(i,j)\in \TT$, and let $\CA_\e:=\cup_{M=1}^\infty \CA_\e^M$.

Define $\wdt \CU_b^+= \CA_{2,\e} \times \CA_{+,\e}$ and $\CU_b^+= \CA_2 \times \CA_{+,\e}$. Using these space, and recalling the variational representation formula for functionals of Brownian motions from \cite{BD98} as well as for Poisson random measures from \cite{BDM11,BCD13}, we obtain the following:
\begin{thm}\label{thm:var-rep}
Let $F$ be a bounded and measurable real-valued function, we have
\beq{vrep}\barray\ad 
\!\!\!\! -\log\EE\big[e^{-F(W\cd, N^\vartheta)}\big] \\
\aad =\!\!\!\!\! \inf_{(v,\phi)\in \wdt \CU_b^+}\!\!\!\EE\Big[L_{T,1}(v)+\vartheta L_{T,2}(\phi)+F\big(W\cd+\int_0^{\cdot} v(s)ds, N^{\vartheta \phi}\big)\Big],
\earray\eeq
where $\vartheta\in (0,\infty)$ is a constant and $N^\vartheta :=(N_{ij}^\vartheta)_{(i,j)\in \TT}$ with $N_{ij}^\vartheta$ defined under \eqref{N-phi}.
\end{thm}

In \thmref{thm:var-rep}, we set $F(W\cd, N^\vartheta): =h^2(\e) F\circ \CG^\e(\sqe W, 
\{\e N_{ij}^{\e^{-1}}\})=h^2(\e)F(\eta^\e\cd)$. Dividing by $h^2(\e)$ to both sides of \eqref{vrep}, we obtain variational representation formula for functional of $\eta^\e$:
\beq{var-rep}\barray
\ad -\frac{1}{h^2(\e)}\log \EE\big[\exp(-h^2(\e)F(\eta^\e))\big]\\
\aad\!\!\!\!\!\!=\!\!\!\! \inf_{(v,\phi)\in \wdt  \CU_b^+} \EE\Big[\frac{1}{ h^2(\e)} L_{T,1}(v)+\frac{1}{\e h^2(\e)}L_{T,2}(\phi)+F\circ \CG^\e(\sqrt{\e}W+\int_0^{\cdot}v(s)ds, \e N^{\e^{-1}\phi})\Big]\\
\aad\!\!\!\!\!\!=\!\!\!\!\!\!\! \inf_{(u,\phi)\in \CU_b^+}\!\!\! \EE\Big[L_{T,1}(u)+\frac{1}{\e h^2(\e)}L_{T,2}(\phi)+F\circ \CG^\e(\sqe W\cd+\int_{0}^{\cdot} h(\e)u^\e(s)ds, \e N^{\e^{-1}\phi})\Big].
\earray\eeq
where $u:=v/h(\e)$. For the remainder of the paper, we will use control variable $u$ in place of $v$. Using the control pair $(u^\e,\phi^\e)\in \CU_b^+$, we consider the controlled processes:
\beq{c-sde}\left\{\barray
\!\!\!\! dX^{\e,u^\e,\phi^\e}(t)\!\!\!\!\! \ad = b(X^{\e,u^\e,\phi^\e}(t), Y^{\e,u^\e,\phi^\e}(t))dt+\sqrt{\e} \sg(X^{\e,u^\e,\phi^\e}(t), Y^{\e,u^\e,\phi^\e}(t))dW(t)\\
\aad \quad+\sqrt{\e} h(\e) \sg( X^{\e,u^\e,\phi^\e}(t), Y^{\e,u^\e,\phi^\e}(t))u^\e(t)dt \\
\!\!\!\! d Y^{\e,u^\e,\phi^\e}(t)\!\!\!\!\!\! \ad =\!\!\!\! \sum_{(i,j)\in \TT}\! \int_{[0,\zeta]}\!\! (j-i) \indi_{\{Y^{\e,u^\e,\phi^\e}(t-)=i\}} \indi_{E_{ij}( X^{\e,u^\e,\phi^\e}(t))}(z) N_{ij}^{\e^{-1}\phi_{ij}^\e}(dz\times dt),
\earray
\right.
\eeq
where $X^{\e,u^\e,\phi^\e}(0) =x_0, Y^{\e,u^\e,\phi^\e}(0)=y_0$. Define 
\beq{def-etae} 
\eta^{\e,u^\e,\phi^\e}(t):=\frac{X^{\e,u^\e,\phi^\e}(t)-\bar X(t)}{\sqe h(\e)},\quad t\in [0,T].
\eeq
From the pathwise solvability of \eqref{c-sde} and Girsanov's theorem, we have $ 
X^{\e,u^\e,\phi^\e}=\bar \CG^\e(\sqe W\cd+ \int_0^{\cdot}h(\e) u^\e(s)ds, \{\e N_{ij}^{\e^{-1}\phi_{ij}^\e}\} )$ 
is the unique solution of \eqref{c-sde}. Hence $
\etae\\
=\CG^\e(\sqe W\cd+\int_0^{\cdot} h(\e) u^\e(s)ds, \{\e N_{ij}^{\e^{-1}\phi_{ij}^\e}\} )$. Therefore, the variational  representation \eqref{var-rep} yields
\beq{var-rep1}
-\frac{1}{h^2(\e)}\log\EE\big[\exp(-h^2(\e)F(\eta^\e))\big]\!\! = \!\!\!\! \!\! \inf_{(u^\e,\phi^\e)\in \CU_b^+}\!\!\!\!\! \EE\big[L_{T,1}(u^\e)+\frac{1}{\e h^2(\e)}L_{T,2}(\phi^\e)+F(\eta^{\e,u^\e,\phi^\e}))\big].
\eeq

Letting $\e\to 0$, from \eqref{var-rep1}, we need to study the asymptotic behavior of $\eta^{\e,u^\e,\phi^\e}$ by proving its tightness and characterizing its limiting behavior. To do this, based on the definition of \eqref{def-etae}, we can decompose $\etae$ as follows:
\beq{b-eta}\barray\ad 
\!\!\!\!\!\!\!\!\eta^{\e,u^\e,\phi^\e}(t) =\frac{X^{\e,u^\e,\phi^\e}(t)-\bar X(t)}{\sqrt{\e}h(\e)} \\
\ad\!\!\!\!\!\!\!\! =\int_0^t \frac{1}{\sqrt{\e}h(\e)}\big[b( X^{\e,u^\e,\phi^\e}(s),  Y^{\e,u^\e,\phi^\e}(s))-\bar b(\bar X(s))\big]ds
\\
\aad\!\!\!\!\!\!\!\!+\int_0^t \sg( X^{\e, u^\e, \phi^\e}(s), Y^{\e, u^\e, \phi^\e}(s))u^\e(s) ds + \frac{1}{h(\e)}\int_0^t \sg( X^{\e, u^\e, \phi^\e}(s), Y^{\e, u^\e, \phi^\e}(s)) dW(s) \\
\ad \!\!\!\!\!\!\!\!= \int_0^t \frac{1}{\sqrt{\e}h(\e)}\big[b(X^{\e,u^\e,\phi^\e}(s), Y^{\e, u^\e,\phi^\e}(s))-\bar b(X^{\e,u^\e,\phi^\e}(s))\big]ds \\
\aad \!\!\!\!\!\!\!\!+ \int_0^t \frac{1}{\sqrt{\e}h(\e)}\big[\bar b(X^{\e,u^\e,\phi^\e}(s))-\bar b(\bar X(s))\big]ds \\
\aad\!\!\!\!\!\!\!\!+ \int_0^t \sg(X^{\e,u^\e,\phi^\e}(s),Y^{\e,u^\e,\phi^\e}(s)) u^\e(s)ds+ \frac{1}{h(\e)} \int_0^t \sg(X^{\e,u^\e,\phi^\e}(s), Y^{\e,u^\e,\phi^\e}(s))dW(s)\\
\ad \!\!\!\!\!\!\!\!=: \sum_{k=1}^4 \eta_k^{\e,u^\e,\phi^\e}(t).
\earray\eeq
We note that the only term above causing trouble to prove tightness is $\etae_1$  whose coefficients $1/\sqe h(\e)$ diverge as $\e\to 0$. For the singular coefficient in $\eta_2^{\e,u^\e,\eta^\e}$, we note that the mean value theorem implies $
(\bar b(X^{\e, u^\e, \phi^\e}(s))-\bar b(\bar X(s)))/\sqe h(\e) = \nabla \bar b(X^{+,\e}(s)) \eta^{\e,u^\e, \phi^\e}(s)$ hold for some $\bar X^{+,\e}(s) = a X^{\e, u^\e, \phi^\e}(s)+(1-a) \bar X(s)$ with $a \in [0,1]$. Thus, the singular coefficient $1/\sqe h(\e)$ in $\eta_2^{\e,u^\e,\eta^\e}$ is effectively absorbed into $\eta^{\e, u^\e, \phi^\e}(s)$, ensuring the application of Gronwall's inequality. This term is harmless. In \secref{sec:lim}, we will first address the singular term $\eta_1^{\e, u^\e, \phi^\e}$ using the Poisson equation discussed in \cite{SX25}. Specifically, we let $\Phi(x,y)$ satisfy the following Poisson equation on finite state space $\LL$:
\beq{pq}
Q(x)\Phi(x,\cdot)(i)=-\big[b(x,i)-\bar b(x)\big] \; \text{ with } \sum_{i=1}^{|\LL|} \Phi(x,i)\mu_i^x=0,
\eeq
where $Q(x)$ is the generator of a Markov chain $\{Y^{x,i}\}$ with initial value $Y^{x,i}(0)=i$, and $\{\mu_i^x\}_{i\in \LL}$ is the unique invariant measure of $Y^{x,i}$. Define $\wdt b(x,i): =b(x,i)-\bar b(x)$, then it satisfies so-called ``central condition'':
\bea
\sum_{i\in \LL} \wdt b(x,i)\mu_i^x=0, \quad \forall x\in \rr^d.
\eea 

Under our assumptions, the Poisson equation \eqref{pq} has a unique solution $\Phi$ which is bounded and Lipschitz continuous; see \cite{PTW12} for details. Due to the lack of regularity of $\Phi$ under (H2) and (H3), we approximate it by $\Phi_m$ in \lemref{lem:app}, enabling us to apply It\^{o} formula. This leads to establishing the uniform boundedness and tightness for $\eta_1^{\e,u^\e,\phi^\e}$ in \lemref{lem:eta1}. We then prove the tightness of $\eta^{\e,u^\e,\phi^\e}$ and characterize its limit in \propref{prop:tight}.

\subsection{Main results}
Recall the definition of $\Ga_{ij}^\psi$ in \eqref{Ga}. We note that for any $\psi=\{\psi_j\}_{j\in \LL}$ with $\psi_j: [0,\zeta] \to \rr_+$ which is  measurable and integrable for each $j\in \LL$, it can be identified with $\theta=(\theta_j)_{j\in \LL}$, where $\theta_j:=\psi_j(z)dz$ takes values in $\CM_F([0,\zeta])$ for each $j\in \LL$. We define $\Ga^\theta(x)$ as follows:
\beq{Ga-th}
\Ga_{ij}^\theta(x)=\left\{\barray
\theta_j(E_{ij}(x)),\ad  \quad i \neq j,\\
-\sum_{y:y\neq j} \theta_y(E_{jy}(x)),\ad \quad i = j.
\earray
\right.
\eeq
In the above, we encode  $\psi$ as measure $\theta$. This measure formulation is more convenient when discussing topologies. 

We now in the position to state our main results.
For $\eta\in C([0,T];\rr^d)$, we define:
\beq{rate}
I(\eta)=\!\!\!\!\!\! \inf_{(u,\psi,\pi)\in \CV(\eta)} \!\! \bigg\{\sum_{i\in \LL}\half \int_0^T\!\!\! \|u_i(s)\|^2 \pi_i(s)ds +\!\!\!\! \sum_{(i,j)\in \TT} \!\!\half\int_0^T \!\!\! \int_0^\zeta |\psi_{ij}(s,z)|^2 \pi_i(s) \la_\zeta(dz)ds \bigg\},
\eeq
where $\CV(\eta)$ is the collection of all $
\{u=(u_i),  \psi=(\psi_{ij}), \pi=(\pi_i)\} \in (\MM([0,T];\rr^d))^{|\LL|} \times (L^2([0,T]\times [0,\zeta]))^{|\TT|} \times \MM([0,T];\CP(\LL))$ 
such that for each $i\in \LL$ and each $(i,j)\in \TT$, $
\int_0^T \|u_i(s)\|^2 \pi_i(s)ds < \infty$ and $\int_{[0,T]\times [0,\zeta]} |\psi_{ij}(s,z)|^2 \pi_i(s)\la_\zeta(dz)ds  <\infty$ and for each $i$,
\beq{eta-def}\barray
\eta(t)\ad\!\!\!\! =  \int_0^t \nabla \bar b(\bar X(s))\eta(s) ds + \sum_{j\in \LL} \int_0^t \sg_j(\bar X(s))u_j(s) \pi_j(s)ds  \\
\aad\!\!\!\!\!\!+ \sum_{(i,j)\in \TT} \int_{[0,t]\times [0,\zeta]} \big[\Phi(\bar X(s),j)-\Phi(\bar X(s),i)\big] \pi_i(s) \indi_{E_{ij}(\bar X(s))}(z) \psi_{ij} (s,z) \la_\zeta(dz) ds,
\earray
\eeq
and for a.e. $s\in [0,T]$ and $j\in \LL$,
\beq{fast-cond}\barray
\sum_{i\in \LL} \pi_i(s)\Ga_{ij}(\bar X(s))=0.
\earray\eeq
Here, $\Ga_{ij}$ denotes $\Ga_{ij}^{\psi}$ in \eqref{Ga} with $\psi=1$.

We say a function $I:C([0,T];\rr^d) \to [0,\infty]$ is a good rate function if it has a compact level sets on $C([0,T];\rr^d)$, that is, for every   $\al>0$, the set $\{\eta\in C([0,T];\rr^d): I(\eta)\leq \al\}$ is a compact subset of $C([0,T];\rr^d)$. The following is our main result. 

\begin{thm}\label{thm:main}
Suppose (H1)-(H3) hold.
The function defined in \eqref{rate} is a good rate function on $C([0,T]; \rr^d)$. The $\{\eta^\e\}_{\e>0}$ satisfies the Laplace principle on $C([0,T];\rr^d)$ with speed $h^{-2}(\e)$ and  rate function $I$ defined in \eqref{rate}. That is, for all bounded and continuous function $F: C([0,T];\rr^d)\to \rr$, we have
\beq{lap}
\lim_{\e\to 0} - \frac{1}{h^2(\e)} \log\EE\Big[\exp(-h^2(\e) F(\eta^\e))\Big]= \inf_{\eta\in C([0,T];\rr^d)}\big\{F(\eta)+ I(\eta)\big\}.
\eeq
Due to the equivalence between the Laplace principle and large deviation principles, we have $\{X^\e\}$ satisfies a MDP with speed $h^{-2}(\e)$ and rate function $I$.
\end{thm}

We would extend the above main result to allow for unbounded drift coefficient $b(\cdot,i)$ for each $i\in \LL$, thereby removing the boundedness condition on $b(\cdot,i)$ in (H1). We assume

\smallskip

\noindent (H1)$'$  For each $i\in \LL$, the function $b(\cdot,i)$ is continuously differentiable with bounded derivatives. Moreover, the Lipschitz conditions \eqref{lip-cond} hold.

\smallskip

\begin{cor}\label{cor:extend}
Under conditions (H1)$'$, (H2) and (H3), the process $\{X^\e\}_{\e>0}$ satisfies the moderate deviation principle on $C([0,T];\rr^d)$ with speed $h^{-2}(\e)$ and rate function $I$ given by \eqref{rate}. 
\end{cor}

For the rest of the paper, we aim to prove \eqref{lap}. To this purpose, it is sufficient to prove the Laplace principle lower bound
\beq{lap-up}
\limsup_{\e\to 0}  - \frac{1}{h^2(\e)} \log\EE\big[\exp(-h^2(\e) F(\eta^\e))\big]\geq  \inf_{\eta\in C([0,T];\rr^d)} \{F(\eta)+I(\eta)\},
\eeq
and the Laplace principle upper bound
\beq{lap-lower}
\liminf_{\e\to 0} - \frac{1}{h^2(\e)} \log\EE\big[\exp(-h^2(\e) F(\eta^\e))\big]\leq \inf_{\eta\in C([0,T];\rr^d)} \{F(\eta)+I(\eta)\}.
\eeq
The proof of \eqref{lap-up} is in \secref{sec:up} and the proof of \eqref{lap-lower} is given in \secref{sec:low}. The proof of \corref{cor:extend} is discussed in \secref{sec:ext}.

\section{The equivalent form of  the rate function and its goodness} \label{sec:rate}
In this section, we first represent an equivalent formulation of the rate function defined in \eqref{rate}. We then demonstrate the goodness of the rate function by establishing the compactness of the level set 
 $\{\eta\in C([0,T];\rr^d): I(\eta)\leq \al\}$ for $\al>0$. This alternative equivalent form proves to be more convenient to work with. Our approach follow the framework outlined in \cite[Section 2.2, 2.3]{BDG18}, but for completeness and reader's convenience, we provide detailed explanations.

\subsection{An equivalent representation of the rate function}
Recall $\la_\zeta$ is Lebesgue measure on $[0,\zeta]$. Let  $\wdh d: \CM_F([0,\zeta])\to [0,\infty]$ be defined by 
\beq{d-hat}
\wdh d (\theta):= \left\{\barray\disp
\int_{[0,\zeta]} \Big|\frac{d\theta(z)}{d\la_\zeta(z)}\Big|^2 \la_\zeta(dz), \ad \quad \text{if } \theta \ll \la_\zeta, \\
\infty, \ad \quad \text{otherwise}.
\earray\right. 
\eeq
For $\theta=(\theta_i)_{i\in \LL}$ with $\theta_i\in \CM_F([0,\zeta])$, with certain abuse of notation, we define $
\wdh d(\theta):= \sum_{i\in \LL} \wdh d(\theta_i)$. 
For $0\leq a < b \leq T$, denote by $\HH_{[a,b]}$ the space 
\bea
\HH_{[a,b]}:=[a,b]\times \LL \times (\CM_F[0,\zeta])^{\LL} \times \rr^d.
\eea
When $[a,b]=[0,t]$, we denote $\HH_{[0,t]}$ as $\HH_t$. Let $\CP_{\text{leb}}(\HH_T)$ be the space of finite measures $\Pi$ on $\HH_T$ such that for all $0\leq a \leq b \leq T$, $
\Pi(\HH_{[a,b]})=b-a$. This says that the measure $\Pi\in \CM_F(\HH_T)$ is in $\CP_{\text{leb}}(\HH_T)$ if and only if the first marginal of $\Pi$ is the Lebesgue measure $\la_T$ on $[0,T]$, i.e., $[\Pi]_1=\la_T$. Here, $[\Pi]_i$ is the marginal on the $i$-th coordinate of $\HH_T$ for $\Pi\in \CM_F(\HH_T)$. For notation simplicity, we  refer to a typical element $(s,y,\theta,z)\in \HH_T$ as $\bv$. Then the measure $\Pi$  captures the time $s$, the state of controlled fast processes $y$, the measures that governs the jump rates $\theta$, and the control variable $z$ applied to perturb the mean of the Brownian motion. For $\eta\in C([0,T];\rr^d)$, we define $\CP_s(\eta)$ as the family of all $\Pi\in \CP_{\text{leb}}(\HH_T)$ such that $
\int_{\HH_T} \|z\|^2 \Pi(d\bv) <\infty$ and $\int_{\HH_T} \wdh d(\theta)\Pi(d\bv) <\infty$,  
and 
\beq{eta-pi}\barray
\eta(t)\ad = \int_{\HH_t} \nabla \bar b(\bar X(s)) \eta(s) \Pi(d\bv)+\int_{\HH_t} \sg(
\bar X(s),y)z \Pi(d\bv) \\
\aad \quad + \int_{\HH_t}\sum_{j\in \LL} [\Phi(\bar X(s),j)-\Phi(\bar X(s),y)]\Ga_{yj}^\theta(\bar X(s))
\Pi(d\bv),
\earray\eeq
and for all $j\in \LL$ and a.e. $t\in [0,T]$,
\beq{fast-pi}
\int_{\HH_t} \Ga_{yj}(\bar X(s))\Pi(d\bv)=0 .
\eeq

Equation \eqref{eta-pi} describes the limiting  controlled dynamics, while \eqref{fast-pi} ensures that the conditional distribution of $\Pi$ with respect to the $y$-variable corresponds to the stationary distribution associated with the generator $\Ga^\theta(\bar X(s))$, given the time instant $s\in [0,T]$, the state $\bar X(s)$ of the averaged dynamics, and the rate control measure $\theta$. We now define the function $\wdh I: C([0,T];\rr^d)\to [0,\infty]$ as follows:
\beq{wdh-I} 
\wdh I(\eta)=\inf_{\Pi\in \CP_{s}(\eta)}\Big\{\int_{\HH_T}\half \big[\|z\|^2 + \wdh d(\theta)\big] \Pi(d\bv )\Big\}.
\eeq
Before proving the equivalence of $I$ and $\wdh I$, we recall regularity results for the solution $\Phi$ of Poisson equation \eqref{pq} in \cite[Theorem 2.2]{SX25} for a countable-state Markov chain.

\begin{thm}\label{thm:Phi}
Suppose conditions (H2)-(H3) hold with $|\LL|=\infty$ and $\|\wdt b(x,\cdot)\|_\infty <\infty$ for all $x\in \rr^d$, then equation \eqref{pq} admits a unique solution $
\Phi(x,i)=\int_0^\infty \EE \big[b(x,Y^{x,i}(t))-\bar b(x)\big] dt$ satisfying
\beq{Phi-bdd}
\|\Phi(x,\cdot)\|_\infty:=\sup_{i\in \LL}\|\Phi(x,i)\| \leq K\|\wdt b(x,\cdot)\|_\infty,
\eeq
where $\{Y^{x,i}\}_{t\geq 0}$ is the unique $\LL$-valued Markov chain generated by generator $Q(x)$ with initial value $Y_0^{x,i}=i$. Moreover, if $Q\in C^1(\rr^d,\rr^{|\LL|\times |\LL|})$ and $\wdt b\in C^1(\rr^d\times \LL; \rr^d)$, then there exists a constant $K$ such that for any $x\in \rr^d$,
\bea
\|\partial_x \Phi(x,\cdot)\|_\infty \leq K \big(\|\wdt b(x,\cdot)\|_\infty \|\nabla Q(x)\|+\|\partial_x \wdt b(x,\cdot)\|_\infty\big).
\eea
If $Q\in C^2(\rr^d, \rr^{|\LL|\times |\LL|})$ and $\wdt b\in C^2(\rr^d\times \LL, \rr^d)$, then there exists a $K>0$ such that for any $x\in\rr^d$,
\bea
\|\partial_x^2 \Phi(x,\cdot)\|_\infty
\ad  \leq K \Big[\|\wdt b(x,\cdot)\|_\infty \big(\|\nabla Q(x)\|+ \|\nabla Q(x)\|^2 + \|\nabla^2 Q(x)\|\big) \\
\aad \qquad \quad  +\|\partial_x \wdt b(x,\cdot)\|_\infty \|\nabla Q(x)\| + \|\partial_x^2 \wdt b(x,\cdot)\|_\infty \Big]. 
\eea
\end{thm}
\begin{rem}\label{rem:bdd-Phi}\rm{
The boundedness assumption of $b$ in (H1) for finite state Markov chain  implies that there exists a constant $K>0$ s.t. $
\|\Phi(x,\cdot)-\Phi(y,\cdot)\|_\infty \leq K \|x-y\|, \ \|\Phi(x,\cdot)\|_\infty \leq K$; 
see \cite[Remark 3.3]{SX25}.}
\end{rem}

\begin{prop}\label{prop:hat-eta}
For every $\eta\in C([0,T];\rr^d)$, $\wdh I(\eta)=I(\eta)$.
\end{prop}

\begin{proof}
Fix $\eta\in C([0,T];\rr^d)$, we first prove $\wdh I(\eta)\leq I(\eta)$. Without loss of generality, we assume $I(\eta)<\infty$, otherwise it is trivial.  For $\e>0$ and let $(u,\psi,\pi)\in \CV(\eta)$ be such that
\beq{I-e}
\sum_{i\in \LL}\half \int_0^T \|u_i(s)\|^2\pi_i(s)ds+\sum_{(i,j)\in \TT} \half \int_{[0,T]\times [0,\zeta]} |\psi_{ij}(s,z)|^2 \pi_i(s)\la_\zeta(dz)ds \leq I(\eta)+\e.
\eeq
Define $\bar \theta_{ij}(s,dz)\in \CM_F[0,\zeta]$ for $i,j\in \LL$ and $s\in [0,T]$ by
\bea
\bar \theta_{ij}(s,dz):=\left\{ \barray
\psi_{ij}(s,z)\la_\zeta(dz), \ad \quad\text{if } (i,j)\in \TT \text{ and } z\mapsto \psi_{ij}(s,z) \text{ is integrable}, \\
0,\ad  \quad \text{otherwise}.
\earray
\right.
\eea

Let $\bar \theta_i(s):=(\bar \theta_{ij}(s,\cdot))_{j\in \LL}$. Then define $\Pi\in \CP_{\text{leb}}(\HH_T)$ by 
\bea\ad 
\Pi([a,b]\times \{i\}\times A\times B):=\int_a^b \pi_i(s)\dl_{\bar\theta_i(s)}(A)\dl_{u_i(s)}(B)ds.
\eea
for $A \in \CB((\CM_F[0,\zeta])^{|\LL|}), B\in \CB(\rr^d)$, $0 \leq a \leq b \leq T$, and $i\in \LL$. It is not difficult to see that $\Pi\in \CP_{s}(\eta)$ and $
\int_{\HH_T} \half \big[\|z\|^2+\wdh d(\theta)\big] \Pi(d\bv)$ 
equals to the left side of \eqref{I-e}. Thus we proved $\wdh I(\eta)\leq I(\eta)+\e$. Since $\e$ is arbitrary small. We proved $\wdh I(\eta) \leq I(\eta)$  for all $\eta  \in C([0,T];\rr^d)$.

We now prove the inverse inequality $I(\eta)\leq \wdh I(\eta)$. Likewise, we assume, without loss of generality, $\wdh I(\eta)<\infty$. Let $\Pi\in \CP_s(\eta)$ be such that $
\int_{\HH_T} \half \big[\|z\|^2+\wdh d(\theta) \big] \Pi(d\bv)\leq \wdh I(\eta)+\e$. 
Denote by $[\Pi]_{34|12}(d\theta\times dz|y,s)$ the conditional distribution on the third and fourth coordinates given the first and second. We can disintegrate the measure $\Pi$ as 
\bea
\Pi(ds\times \{y\}\times d\theta \times dz)= ds\,\pi_y(s) [\Pi]_{34|12}(d\theta\times dz|y,s).
\eea
Define  
\bea\ad 
u_y(s):=\int_{(\CM_F[0,\zeta])^{|\LL|} \times \rr^d} z\, [\Pi]_{34|12}(d\theta\times dz|y,s),\; y\in \LL, s\in [0,T],
\eea
and for $(y,s)\in \LL\times [0,T]$, let
\bea\ad 
\bar\theta_y(s):=\int_{(\CM_F[0,\zeta])^{\LL}\times \rr^d} \theta\, [\Pi]_{34|12}(d\theta\times dz|y,s).
\eea
We write $\bar \theta_y=(\bar\theta_{y y'})_{y'\in \LL}$. The convexity of the square function implies 
\bea\ad 
\wdh d(\bar \theta_y(s)) \leq \int_{(\CM_F[0,\zeta])^{\LL} \times \rr^d} \wdh d(\theta) [\Pi]_{34|12}(d\theta \times dz|y,s).
\eea
Consequently,
\bea\ad 
\sum_{y\in \LL} \int_{[0,T]} \pi_y(s) \wdh d(\bar\theta_y(s))ds \leq \int_{\HH_T} \wdh d(\theta) \Pi(d\bv).
\eea
Define
\bea
\psi_{yy'}(s,z):=\left\{
\barray\disp 
\frac{d\bar\theta_{yy'}(s,\cdot)}{d\la_\zeta\cd}(z), \quad \ad \text{ if } \bar\theta_{yy'}(s,\cdot) \ll \la_\zeta\cd, \\
0, \quad \ad \text{ otherwise.}
\earray\right.
\eea
We obtain 
\bea\ad 
\sum_{(i,j)\in \TT} \int_{[0,T]\times [0,\zeta]} \half |\psi_{ij}(s,z)|^2 \pi_i(s)\la_\zeta(dz)ds \leq \int_{\HH_T} \half \wdh d(\theta) \Pi(d\bv).
\eea
Similarly, we can obtain $
\sum_{i\in \LL} \half \int_0^T \|u_i(s)\|^2 \pi_i(s) ds \leq \half \int_{\HH_T} \|z\|^2 \Pi(d\bv)$. 
To complete the proof, it suffices to show $
(u=(u_i),\psi = (\psi_{ij}), \pi = (\pi_i))\in \CV(\eta)$. 
Since $\Pi \in \CP_s(\eta)$, we have \eqref{eta-pi} which implies that \eqref{eta-def} holds. It remains to show \eqref{fast-cond}. Since $\Pi\in \CP_s(\eta)$, \eqref{fast-pi} implies that for all $j\in \LL$ and a.e. $s\in [0,T]$,
\bea\ad 
\sum_{y\in \LL} \pi_y(s) \int_{(\CM_F[0,\zeta])^{|\LL|}\times \rr^d}\Ga_{yj}(\bar X(s))[\Pi]_{34|12} (d\theta \times dz| y,s)=0.
\eea
It follows that
\bea\ad 
\sum_{y:y\neq j}\pi_y(s) \int_{(\CM_F[0,\zeta])^{|\LL|}\times \rr^d} \Ga_{yj}(\bar X(s))[\Pi]_{34|12}(d\theta\times dz| y,s) \\
\aad = \pi_j(s)\sum_{i: i\neq j} \int \Ga_{ji}(\bar X(s))[\Pi]_{34|12}(d\theta\times dz|y,s).
\eea
By the definition of $\Ga$, we get $
\sum_{y:y\neq j} \pi_y(s) \Ga_{yj}(\bar X(s))=\pi_j(s)\sum_{i: i\neq j}  \Ga_{ji}(\bar X(s))$. 
Consequently, we have \eqref{fast-cond}. The proof is complete.
\end{proof}

\subsection{Compact level sets}
We now show that the function $I$ (which is the same as $\wdh I$) is a good rate function.
\begin{prop}\label{prop:comp-le}
For every $M\in (0,\infty)$, the set $\{\eta\in C([0,T];\rr^d): I(\eta)\leq M\}$ is a compact set. Thus $I$ is a good rate function.
\end{prop}
\begin{proof}
Let $\{\eta_n\}$ be a sequence in the level set $\{\eta\in C([0,T];\rr^d): I(\eta)\leq M\}$. The $I(\eta_n)\leq M$ implies for each $n \in \NN$, there exists some $\Pi_n\in \CP_{s}(\eta_n)$ such that
\beq{I-n}
\int_{\HH_T}\half\Big[\|z\|^2 + \wdh d(\theta) \Big] \Pi_n(d\bv) \leq M+\frac{1}{n}.
\eeq
Then it suffices to prove that $\{\eta_n\}$ is pre-compact, and every limit point belongs to the set $\{\eta\in C([0,T];\rr^d): I(\eta)\leq M\}$. To this end, we prove: (a) $\{\Pi_n, \eta_n\}_{n\in \NN}$ is pre-compact in $\CP_{\text{leb}}(\HH_T)\times C([0,T];\rr^d)$,
and for any limit point $(\eta,\Pi)$, they satisfy (b) $\int_{\HH_T} \half \big[\|z\|^2 + \wdh d(\theta)\big] \Pi(d\bv) \leq M$, and (c) \eqref{eta-pi} and \eqref{fast-pi} holds.

We now prove (a). Since $\LL$ is a finite set (thus a compact set) and $[\Pi_n]_1=\la_T$ for all $n$, to prove the pre-compactness of $\{\Pi_n\}$, it suffices to prove for every $\dl>0$, there exists a constant $K_1>0$ such that
\beq{Pi-n}
\sup_{n\in \NN}\Pi_n \Big\{(s,y,\theta,z)\in \HH_T: \sum_{j\in \LL}\theta_j [0,\zeta] + \|z\|\geq K_1\Big\} \leq \dl.
\eeq
From \eqref{I-n}, we have 
\beq{b-z-pi}
\int_{\HH_T} \|z\|^2 \Pi_n(d\bv) \leq 2 (M+1), \quad \int_{\HH_T} \wdh d(\theta) \Pi_n(d\bv) \leq 2(M+1).
\eeq
Then by the elementary inequality $ab\leq (a^2+b^2)/2$, we have
\beq{b-th-pi}\barray\disp
\sum_{j\in \LL} \int_{\HH_T} \theta_j[0,\zeta] \Pi_n(d\bv)  
\ad 
\leq \sum_{j\in \LL} \int_{\HH_T} \half \Pi_n(d\bv) +\sum_{j\in \LL} \int_{\HH_T} \half |\theta_j[0,\zeta]|^2 \Pi_n(d\bv)\\
\aad \leq \half |\LL|^2 T +\sum_{j\in \LL} \int_{\HH_T} \half \Big|\int_{[0,\zeta]} \frac{d\theta_j}{d\la_\zeta}(z)\la_\zeta(dz)\Big|^2 \Pi_n(d\bv) \\
\aad  \leq \half |\LL|^2 T + \half \int_{\HH_T} \wdh d(\theta) \Pi_n(d\bv) 
\leq \half |\LL|^2 T+ (M+1).
\earray\eeq
Thus the Markov inequality and \eqref{b-z-pi} implies 
\eqref{Pi-n}. It yields that $\{\Pi_n\}$ is pre-compact in $\CP_{\text{leb}}(\HH_T)$. We now prove the pre-compactness of $\{\eta_n\}$. We first show
\beq{K-eta}
\sup_{n\in \NN} \sup_{t\in [0,T]}\|\eta_n(t)\|^2 =: K_\eta <\infty.
\eeq
Since $\Pi_n\in \CP_s(\eta_n)$, we have
\beq{eta-n}\barray
\eta_n(t) \ad =\int_{\HH_t} \nabla \bar b(\bar X(s)) \eta_n(s) \Pi_n(d\bv) + \int_{\HH_t} \sg(\bar X(s),y) z\Pi_n(d\bv)\\
\aad +\int_{\HH_t}\sum_{j\in \LL} [\Phi(\bar X(s),j)-\Phi(\bar X(s),y)] \Ga_{yj}^\theta (\bar X(s)) \Pi_n(d\bv).
\earray\eeq
By the boundedness of $b$ (thus boundedness of $\Phi$) and the linear growth of $\sg$, there exists a constant $K>0$ such that
\bea
\!\!\! \|\eta_n(t)\|^2 \ad\!\!\!\! \leq 3K \int_{\HH_t} \|\eta_n(s)\|^2 \Pi_n(d\bv) + 3K \int_{\HH_t}  (\|\bar X(s)\|^2 +1) \|z\|^2 \Pi_n(d\bv) \\
\aad \qquad + 3K \sum_{j\in \LL} \int_{\HH_t} |\Ga_{yj}^\theta(\bar X(s))|^2 \Pi_n(d\bv) \\
\aad\!\!\!\! \leq 3K\! \int_0^t \!\! \|\eta_n(s)\|^2 ds +6K(M+1)(1+\|\bar X\|_T)+ 3K \! \sum_{j\in \LL}\!\! \int_{\HH_t} \!\!\! |\theta_j([0,\zeta])|^2 \Pi_n(d\bv) \\
\aad\!\!\!\! \leq 3K\int_0^t \|\eta_n(s)\|^2 ds  + 6K(M+1)(1+\|\bar X\|_T)+6K(M+1),
\eea
where $\|\bar X\|_T:=\sup_{t\in [0,T]}\|\bar X(t)\|$, the second inequality follows from \eqref{b-z-pi} and the definition of $\Ga_{yj}^\theta$, the last inequality is due to \eqref{b-th-pi}. The inequality \eqref{K-eta} then follows by the Gronwall inequality. 
Next for $0\leq t_1\leq t_2 \leq T$, we have
\bea\ad 
\|\eta_n(t_2)-\eta_n(t_1)\| \\ \aad \leq \int_{\HH_{[t_1,t_2]}} \|\nabla \bar b(\bar X(s))\eta_n(s)\| \Pi_n(d\bv)+\int_{\HH_{[t_1,t_2]}} \|\sg(\bar X(s),y)z\| \Pi_n(d\bv)\\
\aad\quad + \int_{\HH_{[t_1,t_2]}} \Big\|\sum_{j\in \LL}[\Phi(\bar X(s),j)-\Phi(\bar X(s),i)]\Ga_{yj}^\theta(\bar X(s))\Big\| \Pi_n(d\bv)\\
\aad \leq K \int_{\HH_{[t_1,t_2]}}\|\eta_n(s)\| \Pi_n(d\bv)+ \int_{\HH_{[t_1,t_2]}}\|\sg(\bar X(s),y)\| \|z\| \Pi_n(d\bv) \\
\aad \quad + K \int_{\HH_{[t_1,t_2]}} \sum_{j\in \LL} |\theta_j(E_{ij}(\bar X(s)))| \Pi_n(d\bv)\\
\aad \leq K \Big( \int_{\HH_{[t_1,t_2]}} \Pi_n(d\bv)\Big)^{1/2} \Big( \int_{\HH_{[t_1,t_2]}} \|\eta_n(s)\|^2 \Pi_n(d\bv)\Big)^{1/2} \\
\aad\quad +  \Big(\int_{\HH_{[t_1,t_2]}} (K(1+\|\bar X\|_T))^2 \Pi_n(d\bv)\Big)^{1/2} \Big( \int_{\HH_T} \|z\|^2 \Pi_n(d\bv)\Big)^{1/2} \\
\aad\quad + \sum_{j\in \LL}\Big(\int_{\HH_{[t_1,t_2]}} \Pi_n(d\bv) \Big)^{1/2} \Big( \int_{\HH_T} |\theta_j[0,\zeta]|^2 \Pi_n(d\bv)\Big)^{1/2} \\
\aad \leq K|t_2-t_1|^{1/2},
\eea
where the third inequality follows from the H\"{o}lder inequality and the linear growth of $\sg$, the last inequality follows from the uniform estimates \eqref{K-eta}, \eqref{b-z-pi}, and \eqref{b-th-pi}. This estimates with the uniform boundedness of \eqref{K-eta} implies that $\{\eta_n\}$ is pre-compact in $C([0,T];\rr^d)$. Thus (a) is proved. Suppose that $(\eta,\Pi)$ is a limit point of $\{\eta_n, \Pi_n\}_{n\in \NN}$. Part (b) can be proved by \eqref{I-n} following the Fatou's lemma and the lower semi-continuity of $\wdh d$. Let us now prove part (c). Without loss of generality, we can assume that the full sequence of $\{\eta_n, \Pi_n\}$ converges to $(\eta,\Pi)$. Then,
\bea\ad\!\!\!\!\!\!\!
\int_{\HH_T} \|\nabla \bar b(\bar X(s)) (\eta_n(s)-\eta(s))\| \Pi_n(d\bv)\\
\aad\!\!\!\!\!\!\!\!\! \leq K\sup_{s\in [0,T]}\|\eta_n(s)-\eta(s)\|\int_{\HH_T}\Pi_n(d\bv)\leq K\sup_{s\in [0,T]}\|\eta_n(s)-\eta(s)\| \sqrt{2T(M+1)} \to 0
\eea
as $n\to \infty$. Following the boundedness of $\Phi$ in \remref{rem:bdd-Phi} and \eqref{b-th-pi}, the map
\bea
(s,y,\theta) \mapsto  \sum_{j\in \LL}[\Phi(\bar X(s),i)-\Phi(\bar X(s),y)] \Ga_{yj}^\theta(\bar X(s))
\eea
is a continuous and bounded map. From the convergence of $\Pi_n$ to $\Pi$, we obtain
\beq{conv-Phi-Ga}\barray
\int_{\HH_t} \sum_{j\in \LL} [\Phi(\bar X(s),j)-\Phi(\bar X(s),y)] \Ga_{yj}^\theta(\bar X(s))\Pi_n(d\bv) \\
\to \int_{\HH_t} \sum_{j\in \LL} [\Phi(\bar X(s),j)-\Phi(\bar X(s),y)] \Ga_{yj}^\theta(\bar X(s))\Pi(d\bv).
\earray\eeq
To prove the convergence of the second term on the right-hand side of \eqref{eta-n}, we take the uniform integrability argument. Fix $R'>0$, the mapping $(s,y,z) \mapsto \sg(\bar X(s), y)z \indi_{\{\|z\|\leq R'\}}$ is continuous and bounded, thus as $n\to \infty$,
\bea
\int_{\HH_t} \sg(\bar X(s), y)z \indi_{\{\|z\|\leq R'\}} \Pi_n(d\textbf{v})\to \int_{\HH_t} \sg(\bar X(s), y)z \indi_{\{\|z\|\leq R'\}} \Pi(d\textbf{v}).
\eea
Therefore, for arbitrary small $\e'>0$, there exits a $n(R')>0$ such that for all $n\geq n(R')$,
\bea
\Big|\int_{\HH_t} \sg(\bar X(s), y)z \indi_{\{\|z\|\leq R'\}} \Pi_n(d\textbf{v})- \int_{\HH_t} \sg(\bar X(s), y)z \indi_{\{\|z\|\leq R'\}} \Pi(d\textbf{v})\Big| \leq \e'/3.
\eea
Moreover, the Cauchy-Schwarz inequality, the Lipschitz continuity of $\sg$, the finiteness of $\LL$, and \eqref{b-z-pi} imply there exists a constant $K'>0$ such that
\bea\ad 
\Big|\int_{\{||z||>R'\}} \sg(\bar X(s),y)z \Pi_n(d\bv)\Big| \leq K' \int_{\{||z||>R'\}} \|z\|\Pi_n(d\bv) \\
\aad \leq K' \Big(\int_{\HH_t} \|z\|^2 \Pi_n(d\bv)\Big)^{1/2} \Big(\Pi_n(\|z\|\geq R')\Big)^{1/2} \leq \frac{2K'(M+1)}{R'}.
\eea
The same estimates above hold with $\Pi$ in place of $\Pi_n$. Therefore choose $R'$ large enough such that $2K'(M+1)/R'<\e'/3$. For this choice of $R'$ and for $n \geq n(R')$, we have 
\beq{uni-arg}\barray\ad 
\Big|\int_{\HH_t} \sg(\bar X(s),y)z \Pi_n(d\bv) - \int_{\HH_t} \sg(\bar X(s),y)z \Pi(d\bv)\Big| \\
\aad \leq \Big|\int_{\HH_t} \sg(\bar X(s), y)z \indi_{\{\|z\|\leq R'\}} \Pi_n(d\textbf{v})- \int_{\HH_t} \sg(\bar X(s), y)z \indi_{\{\|z\|\leq R'\}} \Pi(d\textbf{v})\Big| \\
\aad\quad + \Big|\int_{\{||z||>R'\}} \sg(\bar X(s),y)z \Pi_n(d\bv)\Big| +\Big|\int_{\{||z||>R'\}} \sg(\bar X(s),y)z \Pi(d\bv)\Big| \\
\aad \leq \e'/3+\e'/3+\e'/3 =\e'.
\earray\eeq
Since $\e'$ is arbitrary small, the convergence $\int_{\HH_t}\sg(\bar X(s), y)z\, \Pi_n(d\bv) \to \int_{\HH_t}\sg(\bar X(s), y)z\, \Pi_n(d\bv) $ hold.
Consequently, \eqref{eta-pi} holds. Similarly, we can prove \eqref{fast-pi}. The proof is complete.
\end{proof}

\section{Limiting behaviors of controlled processes}\label{sec:lim}
In this section, we focus on establishing the tightness of $\etae$ and characterizing its limit. To achieve this, we first use the Poisson equation \eqref{pq} to analyze $\eta_1^{\e, u^\e, \phi^\e}$. Given that we only assume the Lipschitz continuity for $Q(x)$, the regularity of $\Phi(x,i)$ with respect to $x$ in \thmref{thm:Phi} is not smooth enough to use It\^{o} formula. We proceed the following approximation technique.

Let $\rho: \rr^d\to [0,1]$ be a smooth function such that for $m\in \NN_+$,
\bea\ad 
\int_{\rr^d} \rho(x)dx=1, \int_{\rr^d} |x|^m \rho(x)dx \leq K_m, \text{ and } \|\nabla^m \rho(x)\|\leq K_m \rho(x).
\eea
Define the mollified approximation for $\Phi:\rr^d\times \LL \to 
\rr^d$ as:
\beq{Phi-m}
\Phi_m(x,i)=\int_{\rr^d} \Phi(x-z,i)\rho^m(z)dz,
\eeq
where $\rho^m(z):=m^d \rho(mz)$. For instance, we can take the mollification function $\rho$ as
\bea\ad\!\!\!
\rho(x)=\frac{1}{Z_d}\left\{\barray \exp\Big(\disp -\frac{1}{1-|x|^2}\Big)\quad \ad \text{ if } |x|<1 \\
0 \quad \ad \text{ if } |x| \geq 1
\earray \right. \text{with} \quad Z_d:= \int_{|x|<1} \exp(-1/(1-|x|^2)) dx.
\eea

\begin{lem}\label{lem:app}
For $x\in \rr^d, m\in \NN_+$, we have
\bea
\|\Phi(x,\cdot)-\Phi_m(x,\cdot)\|_\infty \ad \leq K/m, \quad \|\Phi_m(x,\cdot)\|_\infty \leq K, \\
 \|\nabla_x \Phi_m(x,\cdot)\|_\infty \ad \leq K,\quad \|\nabla_x^2 \Phi_m(x,\cdot)\|_\infty \leq Km, 
\eea
where $K$ is a constant independent of $x$ and $m$. 
\end{lem}
\begin{proof}
The proof can be found in \cite[Lemma 4.3]{SX25}, thus details are omitted.
\end{proof}

To continue, we recall following inequalities in \cite[Lemma 3.1, Lemma 3.2]{BDG16}.
\begin{lem}\label{lem:ineq}
\begin{itemize}
\item[\rm{(i)}] For $a, b\in (0,\infty)$ and $\sg\in [1,\infty)$, $ab\leq e^{\sg a} + \frac{1}{\sg} \ell(b)$.
\item[\rm{(ii)}] For every $\beta>0$, there exists $\kappa(\beta), \kappa_1'(\beta)\in (0,\infty)$ such that $\kappa(\beta), \kappa_1'(\beta) \to 0$ as $\beta\to \infty$, and 
\bea
|x-1| \ad \leq \kappa_1(\beta) \ell(x) \quad \text{ if } |x-1|\geq \beta, x\geq 0, \\
x \ad \leq k_1'(\beta) \ell (x), \quad \text{ if } x\geq \beta >1.
\eea
\item[\rm{(iii)}] There is a non-decreasing function $\kappa_2: (0,\infty) \to (0,\infty)$ such that for each $\beta>0$, 
\bea
|x-1|^2 \leq \kappa_2(\beta) \ell(x), \quad \text{if } |x-1|\leq \beta, x\geq 0. 
\eea
\item[\rm{(iv)}] There exists $\kappa_3\in (0,\infty)$ such that for all $x\geq 0$,
\bea
\ell(x)\leq \kappa_3 |x-1|^2, \quad |\ell(x)-(x-1)^2/2| \leq \kappa_3 |x-1|^3.
\eea
\end{itemize}
\end{lem}
\begin{lem}\label{lem:ineq-2}
Suppose $\phi\in \CS_{+,\e}^M$ for some $M>0$, where $\CS_{+,\e}^M$ is defined in  \eqref{S-plus}. Let $\psi=(\phi-1)/\sqe h(\e)$, we have the following inequalities:
\begin{itemize}
\item[\rm{(i)}] For all $\beta>0$, 
\beq{psi-abs}
\int_{[0,T]\times [0,\zeta]} |\psi| \indi_{\{|\psi|\geq \beta/\sqe h(\e)\}} \la_\zeta(dz)ds \leq M \sqe h(\e) \kappa_1(\beta), 
\eeq
and 
\beq{psi-sq} 
\int_{[0,T]\times [0,\zeta]} |\psi|^2 \indi_{\{|\psi| \leq \beta/\sqe h(\e)\}} \la_\zeta(dz)ds \leq M \kappa_2(\beta).
\eeq
\item[\rm{(ii)}] For all $\beta>1$,
\beq{phi-geq}
\int_{[0,T]\times [0,\zeta]} \phi \indi_{\{\phi \geq \beta\}} \la_\zeta(dz)ds \leq M \e h^2(\e) k_1'(\beta).
\eeq
\end{itemize}
Here $\kappa_1(\beta), \kappa_1'(\beta)$, and $\kappa_2(\beta)$ are as in \lemref{lem:ineq}.
\end{lem}

For any $\e>0$, let $\Xe, \Ye$ be solution of \eqref{c-sde} and define $\etae$ as in \eqref{def-etae}. The following lemma establishes the uniform boundedness for the controlled slow process $X^{\e, u^\e,\phi^\e}$.
\begin{lem}\label{lem:X-bdd}
Suppose (H1)-(H3) hold. For every $M \in \NN$, we have
\bea\ad 
\sup_{\e\in (0,1)} \sup_{\substack{(u^\e,\phi^\e)\in \CU_b^+: \\
L_{T,1}(u^\e)\leq M, L_{T,2}(\phi^\e)\leq M\e h^2(\e)}}  \EE\Big(\sup_{t\in [0,T]}\|\Xe(t)\|^2 \Big) <\infty.
\eea
\end{lem} 
\begin{proof}
The proof can be found in \cite[Lemma 3.1]{BDG18}, thus details are omitted.
\end{proof}

Before demonstrating the uniform boundedness and tightness of $\{\etae\}$, we first present the following key lemma concerning $\etae_1$.

\begin{lem}\label{lem:eta1}
Under conditions (H1)-(H3), let $(u^\e,\phi^\e) \in \CU_b^+$ and $\beta>0$, define
\beq{N-psi}\barray\disp
\!\!\! N^{\psi^\e}_{\beta}(t)\ad\!\!\! :=
\int_{[0,T]\times [0,\zeta]} \sum_{(i,j)\in \TT}
\big[\Phi(\Xe(s), j)-\Phi(\Xe(s),i)\big] \\
\aad\!\!\!\!\!\! \times \indi_{\{\Ye(s-)=i\}}\indi_{ E_{ij}(\Xe(s))}(z) \psi_{ij}^\e(s,z) \indi_{\{|\psi_{ij}^\e(s,z)|\leq \beta/(\sqe h(\e))\}}\la_\zeta(dz)ds.
\earray\eeq 
Then for any $m\geq 0$, there exists a constant $K=K(T,\zeta,M,\beta,|\LL|)>0$ (independent of $m$) such that 
\beq{eta1-m}\barray\ad 
\EE \Big( \sup_{t\in [0,T]} \big\|\etae_1(t)- N_\beta^{\psi^\e}(t)\big\|\Big) \\
\aad \leq \frac{K}{m}\Big(\frac{T}{\sqe h(\e)}+\sqrt{\zeta T} |\LL|^2 \sqrt{M \kappa_2(\beta)}\Big)+\frac{K m \e^2}{\sqe h(\e)}+K\Big(\e+\frac{\e}{h(\e)}+\frac{1}{h(\e)}+\sqe h(\e)\Big).
\earray\eeq
It follows that there exists a $m_0(\e)>0$ $($for example $m_0(\e)=1/(\e^{1/2+\varrho} h(\e)), \forall \varrho \in (0,1))$ and $\e_1>0$ such that for $m\geq m_0(\e)$ and $\e \in (0,\e_1)$, 
\beq{eta1-bdd}
\EE\Big(\sup_{t\in [0,T]}\|\etae_1(t)\|\Big) \leq K+1,
\eeq
where $K$ is a constant independent of $m$. Moreover, for $h\in (0,T)$ and $0 <t_2 \leq t_1+h\leq T$, we have 
\beq{eta1-tight}
\EE\Big(\sup_{0\leq t_1\leq t_2 \leq T, t_2\leq t_1+h}\big\|\etae_1(t_2)-\etae_1(t_1)\big\| \Big) \leq  K h^{1/2}.
\eeq
\end{lem}
\begin{proof}
\thmref{thm:Phi} implies that for the function $\wdt b$, there exists a unique solution $\Phi(x,i)$ for the Poisson equation \eqref{pq} with $\Phi(x,i)$ being bounded and Lipschitz continuous with respect to $x$. To apply the It\^{o} formula in \cite[p.29]{YZ10}, we work with the approximated sequence $\Phi_m(\cdot,\cdot)$ defined in \eqref{Phi-m} such that \lemref{lem:app} holds. Then 
\beq{Ito-Phi}\barray\ad 
\Phi_m(X^{\e, u^\e, \phi^\e}(t), Y^{\e, u^\e, \phi^\e}(t))-\Phi_m(x_0, y_0)\\
\aad =\int_0^t \nabla \Phi_m(X^{\e, u^\e, \phi^\e}(s), Y^{\e, u^\e, \phi^\e}(s)) b(X^{\e, u^\e, \phi^\e}(s), Y^{\e, u^\e, \phi^\e}(s))ds\\
\aad + \sqe h(\e) \int_0^t \nabla \Phi_m(\Xe(s), \Ye(s))  \sg(\Xe(s), \Ye(s))u^\e(s)  ds\\ 
\aad + \frac{\e}{2} \int_0^t \trace\Big[ \nabla^2 \Phi_m(\Xe, \Ye)(\sg\sg^\top)(\Xe(s),\Ye(s))\Big]ds \\
\aad + \sqe \int_0^t \nabla \Phi_m(\Xe(s), \Ye(s)) \sg(\Xe(s), \Ye(s))dW(s)\\
\aad +\frac{1}{\e}\int_0^t Q(\Xe(s)) \Phi_m(\Xe(s), \cdot)(\Ye(s)) ds  \\
\aad + \frac{1}{\e}\int_{[0,T]\times [0,\zeta]} \sum_{(i,j)\in \TT} \big[\Phi_m(\Xe(s), j)-\Phi_m(\Xe(s),i)\big]\\
\aad \qquad\qquad\qquad\quad  \times \indi_{\{\Ye(s)=i\}} \indi_{E_{ij}(\Xe(s))}(z) \big[\phi_{ij}^\e(s,z)-1\big]\la_\zeta(dz)ds \\
\aad + \int_{[0,T]\times [0,\zeta]} \sum_{(i,j)\in \TT} \big[\Phi_m(\Xe(s), j)-\Phi_m(\Xe(s),i)\big] \\
\aad \qquad\qquad\qquad\quad  \times \indi_{\{\Ye(s)=i\}} \indi_{E_{ij}(\Xe(s))}(z) \wdt N_{ij}^{\e^{-1}\phi_{ij}^\e}(dz \times ds),
\earray\eeq
where 
$\wdt N_{ij}^{\e^{-1}\phi_{ij}^\e}(dz\times ds )= N_{ij}^{\e^{-1} \phi_{ij}^\e}(dz \times ds ) - \e^{-1}\phi_{ij}^\e(s,z)dz ds$. For the six line in \eqref{Ito-Phi}, we used the definition of the generator $Q(x)$ and $E_{ij}(x)$ to obtain
\bea
Q(x)\Phi_m(x,\cdot)(y)
\ad = \sum_{j\in \LL}(\Phi_m(x,j)-\Phi_m(x,y)) q_{yj}(x) \\
\aad = \sum_{(i,j)\in \TT} (\Phi_m(x,j)-\Phi_m(x,i)) \indi_{\{y=i\}} \int_{[0,\zeta]}\indi_{E_{ij}(x)}(z)\la_\zeta(dz). 
\eea
Let $\psi_{ij}^\e(s,z):=(\phi_{ij}^\e(s,z)-1)/\sqe h(\e)$ and define $\wdt M_m^{\phi^\e}$ as following:
\beq{M-wdt}\barray\disp
\wdt M_m^{\phi^\e}(t): = \int_{[0,T]\times [0,\zeta]} \ad \sum_{(i,j)\in \TT}\big[\Phi_m(\Xe(s),j)-\Phi_m(\Xe(s),i)\big] \\
\aad \times \indi_{\{\Ye(s-)=i\}} \indi_{E_{ij}(\Xe(s))}(z) \wdt N_{ij}^{\e^{-1}\phi_{ij}^\e}(dz\times ds).
\earray\eeq

To proceed, we rewrite $\eta_1^{\e,u^\e, \phi^\e}$ in terms of the solution of the Poisson equation \eqref{pq}. Using \eqref{pq} with $x$ being $\Xe(s)$ and $i$ being $\Ye(s)$, we have 
\beq{eta1-1}\barray
\eta_1^{\e,u^\e,\phi^\e}(t) 
\ad  = \frac{1}{\sqe h(\e)} \int_0^t b(\Xe(s), \Ye(s)) - \bar b(\Xe(s))ds \\
\ad = \frac{1}{\sqe  h(\e)} \int_0^t - Q(\Xe(s)) \Phi(\Xe(s),\cdot) (\Ye(s)) ds \\
\ad = \frac{1}{\sqe h(\e)} \int_0^t -Q(\Xe(s)) [\Phi-\Phi_m](\Xe(s),\cdot)(\Ye(s))ds \\
\aad\quad + \frac{1}{\sqe h(\e)} \int_0^t - Q(\Xe(s)) \Phi_m(\Xe(s),\cdot)(\Ye(s))ds. 
\earray\eeq
For the integral in the last line above, it appears in the six line of \eqref{Ito-Phi} with different constant coefficient. To rewrite it in \eqref{eta1-1}, we multiply $\e$ and divide $\sqe h(\e)$ to both sides of \eqref{Ito-Phi}, and then rearrange terms (moving the integral in the six line of \eqref{Ito-Phi} with new coefficients to the left-hand side of equality \eqref{Ito-Phi} and moving all other terms to its right-hand side). We can get
\beq{eta1-2}\barray\ad
\frac{1}{\sqe h(\e)} \int_0^t -Q(\Xe(s))\Phi_m(\Xe(s), \cdot)(\Ye(s))ds \\
\aad = -\frac{\sqe}{h(\e)} (\Phi_m(\Xe(t), \Ye(t)) - \Phi_m(x_0, y_0)) \\
\aad\quad + \frac{\sqe}{ h(\e)} \int_0^t \nabla \Phi_m (\Xe(s), \Ye(s)) b(\Xe(s), \Ye(s))ds  \\
\aad \quad + \e \int_0^t \nabla \Phi_m(\Xe(s), \Ye(s)) \sg(\Xe(s), \Ye(s)) u^\e(s)ds \\
\aad \quad + \frac{\e^{3/2}}{2 h(\e)} \int_0^t \text{Tr}[\nabla^2 \Phi_m(\Xe(s), \Ye(s)) (\sg\sg^\top)(\Xe(s), \Ye(s))] ds \\
\aad \quad + \frac{\e}{h(\e)} \int_0^t  \nabla \Phi_m(\Xe(s), \Ye(s)) \sg(\Xe(s), \Ye(s))dW(s) \\
\aad\quad + \int_0^t \sum_{(i,j)\in \TT} \big[\Phi_m(\Xe(s),j) - \Phi_m(\Xe(s), i)\big]  \\
\aad \qquad \qquad \qquad \times \indi_{\{\Ye(s-)=i\}} \indi_{E_{ij}(\Xe(s))}(z) \psi_{ij}^\e(s,z) \la_\zeta(dz)ds + \frac{\sqe}{h(\e)} \wdt{M}_m^{\phi^\e}(t)
\earray\eeq
As $\e\to 0$, from \eqref{dev-h}, we note that all terms except the second-to-last term on the right-hand side of \eqref{eta1-2} converges to zero. Moreover, as we discussed in second paragraph above \secref{sec:pre}, $\psi^\e$ will not converge in the $L^2$ sense as $\e\to 0$. We split $\psi^\e(s,z)$ into two parts: $\psi^\e(s,z)\indi_{\{|\psi^\e(s,z)|\leq \beta/(\sqe h(\e))\}}$ and $\psi^\e(s,z)\indi_{\{|\psi^\e(s,z)|>\beta/(\sqe h(\e))\}}$ for some $\beta>0$. Then, we have the decomposition
\beq{eta1-3}\barray\ad
\int_0^t \sum_{(i,j)\in \TT} \Big[\Phi_m(\Xe(s), j) - \Phi_m(\Xe(s), i)\Big] \\
\aad \qquad \qquad \times \indi_{\{\Ye(s-)=i\}} \indi_{E_{ij}(\Xe(s))}(z) \psi_{ij}^\e(s,z) \la_\zeta(dz)ds \\
\aad = \int_0^t \sum_{(i,j)\in \TT}\Big[\Phi_m(\Xe(s),j)-\Phi_m(\Xe(s),i)\Big] \\
\aad \qquad \qquad \times \indi_{\{\Ye(s-)=i\}} \indi_{E_{ij}(\Xe(s))}(z) \psi_{ij}^\e(s,z) \indi_{\{|\psi^\e(s,z)| > \beta/(\sqe h(\e))\}}\la_\zeta(dz)ds \\
\aad\quad + \int_0^t \sum_{(i,j)\in \TT}\Big[\Phi_m(\Xe(s),j)-\Phi_m(\Xe(s),i)\Big] \\
\aad \qquad \qquad \times \indi_{\{\Ye(s-)=i\}} \indi_{E_{ij}(\Xe(s))}(z) \psi_{ij}^\e(s,z) \indi_{\{|\psi^\e(s,z)| \leq  \beta/(\sqe h(\e))\}}\la_\zeta(dz)ds \\
\aad = \int_0^t \sum_{(i,j)\in \TT}\Big[\Phi_m(\Xe(s),j)-\Phi_m(\Xe(s),i)\Big] \\
\aad \qquad \qquad \times \indi_{\{\Ye(s-)=i\}} \indi_{E_{ij}(\Xe(s))}(z) \psi_{ij}^\e(s,z) \indi_{\{|\psi^\e(s,z)| > \beta/(\sqe h(\e))\}}\la_\zeta(dz)ds \\
\aad \quad + N_{m,\beta}^{\psi^\e}(t) - N_\beta^{\psi^\e}(t) + N_\beta^{\psi^\e}(t)
\earray
\eeq
where $N_{m,\beta}^{\psi^\e}(t)$ is defined as 
\beq{Nm-psi}\barray\disp
\ad\!\!\!\!\!\! N^{\psi^\e}_{m,\beta}(t):=
\!\! \int_{[0,T]\times [0,\zeta]} \sum_{(i,j)\in \TT}
\big[\Phi_m(\Xe(s), j)-\Phi_m(\Xe(s),i)\big] \\
\aad\qquad\!\! \times \indi_{\{\Ye(s-)=i\}}\indi_{ E_{ij}(\Xe(s))}(z) \psi_{ij}^\e(s,z) \indi_{\{|\psi_{ij}^\e(s,z)|\leq \beta/(\sqe h(\e))\}}\la_\zeta(dz)ds.
\earray\eeq 
and $N_\beta^{\psi^\e}$ is defined in \eqref{N-psi}, which replace $\Phi_m$ by $\Phi$ in \eqref{Nm-psi}. The purpose of introducing $N_\beta^{\psi^\e}$ in the final equality of \eqref{eta1-3} is to obtain a limit that no longer depends on $m$.

Consequently, combining \eqref{eta1-1},\eqref{eta1-2},\eqref{eta1-3} yields
\beq{eta1}\barray \ad 
\eta_1^{\e, u^\e, \phi^\e}(t) \\
\aad = \frac{1}{\sqe h(\e)}\int_0^t -Q(\Xe(s))\big[(\Phi-\Phi_m)(\Xe(s),\Ye(s))\big]ds\\ 
\aad\quad -\frac{\sqe}{ h(\e)}\big(\Phi_m(\Xe(t), \Ye(t))-\Phi_m(x_0,y_0)\big) \\
\aad\quad + \frac{\sqe}{ h(\e)} \int_0^t \nabla \Phi_m(\Xe(s), \Ye(s)) b(\Xe(s), \Ye(s))ds \\
\aad\quad + \e \int_0^t \nabla \Phi_m(X^{\e, u^\e, \phi^\e}(s), Y^{\e, u^\e, \phi^\e}(s)) \sg(\Xe(s),\Ye(s))u^
\e(s)ds  \\
\aad\quad + \frac{\e^{3/2}}{2h(\e)}\int_0^t \trace \big[\nabla^2  \Phi_m(\Xe(s), \Ye(s)) (\sg\sg^\top)(\Xe(s), \Ye(s))\big]ds \\
\aad\quad + \frac{\e}{h(\e)}\int_0^t \nabla \Phi_m(\Xe(s), \Ye(s)) \sg(\Xe(s), \Ye(s))dW(s)\\
\aad\quad + \int_0^t \sum_{(i,j)\in \TT}\big[\Phi_m(\Xe(s),j)-\Phi_m(\Xe(s),i)\big] \\
\aad\qquad\qquad\quad\times \indi_{\{\Ye(s-)=i\}} \indi_{E_{ij}(\Xe(s))}(z) \psi_{ij}^\e(s,z)\indi_{\{|\psi_{ij}^\e(s,z)|>\beta/(\sqe h(\e))\}} \la_\zeta(dz)ds \\
\aad\quad + N_{m,\beta}^{\psi^\e}(t) - N_{\beta}^{\psi^\e}(t)+ N_{\beta}^{\psi^\e}(t)+  \frac{\sqe}{h(\e)} \wdt M_{m}^{\phi^\e}(t).
\earray
\eeq 

We now provide estimates for terms on the right-hand side of equality  \eqref{eta1}. 
From \lemref{lem:app} and assumption (H3), we have
\bea\ad\!\!\! 
\EE\bigg(\sup_{t\in [0,T]}\Big\|\frac{1}{\sqe h(\e)} \int_0^t Q(\Xe(s))\big[(\Phi-\Phi_m)(\Xe(s),\Ye(s))\big]ds \Big\|\bigg)\\
\aad \leq \frac{K T}{m \sqe h(\e)} \EE\bigg(1+\sup_{t\in [0,T]}\|\Xe(t)\|\bigg) \leq \frac{K T}{m \sqe h(\e)}.
\eea
Since $u^\e\in \CS_2^M$, the linear growth condition of $\sg$, the uniform boundedness of $\nabla \Phi_m$, and  the H\"{o}lder inequality yield
\bea\ad\!\!\!\!
\EE\bigg(\sup_{t\in [0,T]} \Big\|\e \int_0^t \nabla \Phi_m (\Xe(s),\Ye(s))\sg(\Xe(s),\Ye(s)) u^\e(s)ds\Big\| \bigg) \\
\aad\!\!\!\! \leq  K\, \e\, \EE\bigg(1+\sup_{s\in [0,T]}\|\Xe(s)\|\bigg)\int_0^T u^\e(s)ds \leq K \sqrt{MT}\e.
\eea
Moreover, the boundedness of $\nabla^2 \Phi_m$ gives
\bea\ad 
\EE\Big(\sup_{t\in [0,T]} \Big\|\frac{\e^{3/2}}{2h(\e)} \int_0^t\text{Tr}\big[\nabla^2 \Phi_m(\Xe(s),\Ye(s)\\
\aad \qquad\qquad\qquad\qquad\qquad \quad\quad \times (\sg\sg^\top)(\Xe(s), \Ye(s))\big] ds \Big\|\Big) \\
\aad \leq \frac{K  m \e^2}{\sqe h(\e)} \EE\Big(1+\sup_{s\in [0,T]}\|\Xe(s)\|^2\Big) \leq \frac{K m \e^2}{\sqe h(\e)}.
\eea
By the Lenglart-Lepingle-Pratelli inequality, \lemref{lem:app}, and (H1), we have
\bea\ad 
\EE\bigg(\sup_{t\in [0,T]}\Big\|\frac{\sqe}{h(\e)}\wdt  M_m^{\phi^\e}(t)\Big\|\bigg) \\
\aad \leq  \frac{K}{h(\e)} \EE\bigg(\sup_{t\in [0,T]}\|\Phi_m(\Xe(t),\cdot)\|_\infty\bigg) \bigg(\sum_{(i,j)\in \TT} \int_{[0,T]\times [0,\zeta]} \phi_{ij}^\e(s,z)\la_{\zeta}(dz)ds\bigg)^{1/2} \\
\aad \leq \frac{K}{h(\e)}\bigg( \sum_{(i,j)\in \TT} \int_{[0,T]\times [0,\zeta]} (e+\ell(\phi_{ij}^\e(s,z)))\la_\zeta(dz)ds \bigg)^{1/2} \\
\aad \leq  \frac{K}{h(\e)}  \big(\zeta T e+ M\e h^2(\e)\big)^{1/2},
\eea
for some $K>0$, where the second inequality uses inequality (i) in \lemref{lem:ineq} and last inequality is because $\phi^\e\in S_{+,\e}^M$. Moreover, the Doob martingale inequality, the uniform boundedness of $\Phi_m$, and the linear growth condition of $\sg$ imply that
\bea\ad\!\!\!\!\!\!\! 
\EE\bigg(\sup_{t\in [0,T]}\! \Big\|\frac{\e}{h(\e)}\!\! \int_0^t\! \nabla \Phi_m(\Xe(s), \Ye(s))\sg(\Xe(s), \Ye(s))dW(s)\Big\| \bigg) \\
\aad\!\!\!\!\!\!\! \leq \frac{\e}{h(\e)} \bigg( \int_0^T \EE \big\|\nabla \Phi_m(\Xe(s), \Ye(s)) \sg(\Xe(s), \Ye(s))\big\|^2 ds \bigg)^{1/2} \\
\aad\!\!\!\!\!\!\!  \leq \frac{K \sqrt{T} \e }{h(\e)} \EE \bigg( 1+ \sup_{t\in [0,T]} \|\Xe(t)\|\bigg) \leq \frac{K\sqrt{T} \e}{h(\e)}.
\eea
In the light of \eqref{psi-abs} in \lemref{lem:ineq-2}, we have
\bea\ad\!\!\!\!\!\!\!
\EE\bigg(\sup_{t\in [0,T]}\Big\|\int_{[0,T]\times [0,\zeta]} \sum_{(i,j)\in \TT} [\Phi_m(\Xe(s),j)-\Phi_m(\Xe(s),i)] \\
\aad\quad\qquad \times \indi_{\{\Ye(s-)=i\}} \indi_{E_{ij}(\Xe(s))}(z) \psi_{ij}^\e(s,z) \indi_{\{|\psi_{ij}^\e(s,z)| > \beta/(\sqe h(\e))\}} \la_\zeta(dz)ds \Big\|\bigg)\\
\aad\!\!\!\!\!\!\! \leq K \sum_{(i,j)\in \TT} \int_{[0,T]\times [0,\zeta]} |\psi_{ij}^\e(s,z)|\indi_{\{|\psi_{ij}^\e(s,z)|> \beta/(\sqe h(\e))\}} \la_\zeta(dz)ds \\
\aad\!\!\!\!\!\!\! \leq  K |\LL|^2 M \sqe h(\e) \kappa_1(\beta).
\eea
Furthermore, \lemref{lem:app}, and H\"{o}lder inequality  implies
\beq{Nm-N}\barray\ad 
\EE\bigg(\sup_{t\in [0,T]}\|N_{m,\beta}^{\psi^\e}(t)-N_\beta^{\phi^\e}(t)\|\bigg)\\
\aad \leq \EE \int_{[0,T]\times [0,\zeta]} \sum_{(i,j)\in \TT} \big\|(\Phi_m-\Phi)(\Xe(s),j)-(\Phi_m-\Phi)(\Xe(s),i)\big\| \\
\aad \qquad \times  \indi_{\{\Ye(s-)=i\}} \indi_{E_{ij}(\Xe(s))}(z) |\psi_{ij}^\e(s,z)| \indi_{\{|\psi_{ij}^\e(s,z)|\leq \beta \sqe h(\e)\}} \la_\zeta(dz)ds\\
\aad \leq \frac{K}{m} \sum_{(i,j)\in \TT} \int_{[0,T]\times [0,\zeta]} |\psi_{ij}^\e(s,z)| \indi_{\{|\psi_{ij}(s,z)|\leq \beta/\sqe h(\e)\}} \la_\zeta(dz)ds \\
\aad \leq \frac{K\sqrt{\zeta T}}{m} \sum_{(i,j)\in \TT} \bigg(\int_{[0,T]\times [0,\zeta]} |\psi_{ij}^\e(s,z)|^2 \indi_{\{|\psi_{ij}^\e(s,z)|\leq \beta/(\sqe h(\e))\}}\la_\zeta(dz)ds\bigg)^{1/2}\\
\aad \leq \frac{K \sqrt{\zeta T}}{m}|\LL|^2 \sqrt{M \kappa_2(\beta)}.
\earray\eeq

Combining all above estimates, the boundedness of $\Phi_m$, and \eqref{eta1} implies that there exists a constant $K=K(T,\zeta, M, \beta, |\LL|)$ depending on $\zeta,T,M,\beta,|\LL|$, independent of $m$ such that
\bea\ad 
\EE\bigg(\sup_{t\in [0,T]}\big\|\eta_1^{\e, u^\e, \phi^\e} (t) -   N_{\beta}^{\psi^\e}(t)\big\|\bigg) \\
\aad \leq \frac{K}{m}\bigg(\frac{T}{\sqe h(\e)}+\sqrt{\zeta T}|\LL|^2 \sqrt{M \kappa_2(\beta)}\bigg)+ \frac{KT m \e^2}{\sqe h(\e)}+ K\bigg(\e+\frac{\e}{h(\e)}+\frac{1}{h(\e)}+\sqe h(\e)\bigg).
\eea
Thus for $\e>0$, there exists a $m_0(\e)>0$ such that for $m\geq m_0(\e)$,
\beq{eta-N-beta}
\EE\bigg(\sup_{t\in [0,T]} \|\etae(t)-N_\beta^{\psi^\e}(t)\|\bigg) \leq K o(\e) \to 0, \quad\text{as } \e \to 0.
\eeq 
For example, we can take $m_0(\e)=1/(\e^{1/2+\varrho} h(\e)), \forall\, \varrho \in (0,1)$. For $N_{\beta}^{\psi^\e}$, similar to the computation in \eqref{Nm-N}, we have
\beq{E-N-beta}\barray\ad  
\EE\Big(\sup_{t\in [0,T]}\|N_{\beta}^{\psi^\e}(t)\| \Big) \leq K  \sum_{(i,j)\in \TT} \int_{[0,T]\times [0,\zeta]}  |\psi_{ij}^\e(s,z)| \indi_{\{|\psi_{ij}^\e(s,z)| \leq \beta/(\sqe h(\e))\}} \la_\zeta(dz)ds \\
\aad \leq K \sqrt{\zeta T}|\LL|^2 \bigg(\int_{[0,T]\times [0,\zeta]} |\psi_{ij}^\e(s,z)|^2  \indi_{\{|\psi_{ij}^\e(s,z)|\leq \beta/(\sqe h(\e))\}} \la_\zeta(dz)ds \bigg)^{1/2} \\
\aad \leq  K \sqrt{\zeta T} |\LL|^2 \sqrt{M \kappa_2(\beta)}<\infty.
\earray\eeq
Then the triangle inequality, \eqref{eta-N-beta}, and \eqref{E-N-beta} implies there exists $\e_1>0$ such that for all $\e\in (0,\e_1)$,
\bea\ad 
\EE\bigg(\sup_{t\in [0,T]}\|\etae_1(t)\|\bigg)\\
\aad  \leq  \EE\bigg(\sup_{t\in [0,T]}\|\etae_1(t)-N_\beta^{\psi^\e}(t)\|\bigg) +\EE\bigg(
\sup_{t\in [0,T]}\|N_{\beta}^{\psi^\e}(t)\|\bigg) \leq K+1,
\eea
where $K>0$ is a constant independent of $m$. To prove \eqref{eta1-tight}, we still apply \eqref{Ito-Phi} to rewrite $\etae_1(t_2)-\etae_1(t_1)$ for $t_2\leq t_1+h$ and $0\leq t_1\leq t_2\leq T$. We then use the same estimation method for \eqref{eta1-bdd} to derive \eqref{eta1-tight}. Thus, the proof is complete.
\end{proof}

Building on estimates for $\etae_1$ in \lemref{lem:eta1}, the following lemmas aim to establish the uniform boundedness and  tightness of sequence $\{\eta^{\e,u^\e,\phi^\e}\}_{\e>0}$.
\begin{lem}\label{lem:bdd}
For every $M>0$, there exists $\wdt \e>0$ such that for all $\e \in (0,\wdt \e)$,
\bea\ad 
\sup_{\e\in (0,\wdt\e)}\sup_{\substack{(u^\e,\phi^\e) \in \CU_b^+\\ L_{T,1}(u^\e)\leq M, L_{T,2}(\phi^\e)\leq M\e h^2(\e)}} \EE\bigg(\sup_{t\in [0,T]} \|\eta^{\e, u^\e, \phi^\e}(t)\| \bigg) <\infty.
\eea  
\end{lem}
\begin{proof}
Fix $M>0$ and  let $(u^\e,\phi^\e) \in \CU_b^+$ with $L_{T,1}(u^\e)\leq M$ and $L_{T,2}(\phi^\e)\leq M \e h^2(\e)$. The definition of $\eta^{\e,u^{\e},\phi^\e}$ yields that $
X^{\e,u^\e, \phi^\e}=\bar X(t)+\sqe h(\e) \eta^{\e, u^\e, \phi^\e}(t)$. The Lipschitz continuity of function $b$ and so the $\bar b$ gives
\bea\disp
\sup_{s\in[0,t]} \|\eta_2^{\e, u^\e, \phi^\e}(s)\| 
\ad = \sup_{s\in [0,t]} \Big\|\int_0^s \frac{1}{\sqe h(\e)}\big(\bar b(\Xe(r)) - \bar b(\bar X(r)) \big)dz \Big\|  \\
\aad  \leq L_0 \int_0^t \sup_{r\in [0,s]}\|\eta^{\e, u^\e, \phi^\e}(r)\| ds.
\eea
In addition, the H\"{o}lder inequality implies 
\bea\ad \!\!\!
\EE\bigg(\sup_{s\in [0,t]}\|\etae_3(s)\| \bigg)\\
\ad\!\!\!\! \leq \EE\bigg( \int_0^t \|\sg(\Xe(s),\Ye(s))\|^2 ds\bigg)^{\frac{1}{2}} \bigg( \int_0^t \|u^\e(s)\|^2 ds\bigg)^{\frac{1}{2}} \\
\aad \leq  \sqrt{2MT} \EE\bigg(1+\sup_{s\in [0,t]}\|\Xe(s)\|\Big) \leq K \sqrt{2MT},
\eea
and there exists $\e_2>0$ such that for all $\e\in (0,\e_2)$,
\bea\disp 
\EE\bigg(\sup_{s \in [0,t]} \|\etae_4(s)\|\bigg) 
\ad \leq \frac{1}{h(\e)} \bigg(\int_0^t \EE \|\sg(\Xe(s),\Ye(s))\|^2  ds \bigg)^{1/2} \\
\aad \leq \frac{K\sqrt{T}}{h(\e)} \EE\bigg(1+\sup_{s\in [0,t]} \|\Xe(s)\|\bigg)\leq \frac{K\sqrt{T}}{h(\e)} \leq K\sqrt{T}.
\eea

Combing above estimates and \eqref{eta1-bdd}, \lemref{lem:X-bdd} implies for all $\e\in (0,\e_2)$,
\bea\ad \!\!\!\!
\EE\bigg(\sup_{s\in [0,t]}\|\etae(s)\| \bigg) \\
\aad\!\!\!\!\! \leq \EE\bigg(\sup_{s \in [0,t]}\|\etae_1(s)\|\bigg) + L_0 \int_0^t \EE\bigg(\sup_{r \in [0,s]}\|\etae(r)\| \bigg)  ds + K\sqrt{2MT} + K\sqrt{T}.
\eea
Take $\wdt \e =\e_1\wedge \e_2$, the lemma follows from the Gronwall inequality.
\end{proof} 

\begin{lem}\label{lem:tight}
Suppose (H1)-(H3) hold. For every $M\in \NN$, the sequence $\{\etae;\phi^\e\in \CS_{+,\e}^M,u^\e\in S_2^M, 0\leq\e \leq \wdt \e\}$ is tight in $C([0,T];\rr^d)$. 
\end{lem}
\begin{proof}
For $0\leq h\leq T$ and $\e>0$, we note that 
\bea
\eta^{\e,u^\e,\phi^\e}(t_2)-\eta^{\e,u^\e,\phi^\e}(t_1)=\sum_{k=1}^4 \big(\eta_k^{\e, u^\e,\phi^\e}(t_2)-\eta_k^{\e, u^\e, \phi^\e}(t_1)\big).
\eea
Then the linear growth condition on $\sg$  and \lemref{lem:X-bdd} implies that
\bea \ad
\int_0^t \sg(X^{\e, u^\e, \phi^\e}(s), Y^{\e, u^\e, \phi^\e}(s))dW(s)
\eea
is a square integrable martingale. Thus the Doob maximal inequality gives
\bea \ad 
\EE\Big(\sup_{t \in [0,T]} \|\eta_4^{\e, u^\e, \phi^\e}(t)\|^2 \Big) \leq   \frac{4}{h^2(\e)}\int_0^t \|\sg(X^{\e, u^\e, \phi^\e}(t), Y^{\e, u^\e, \phi^\e}(t))\|^2 dt \to 0
\eea
as $\e\to 0$. \lemref{lem:eta1} gives the tightness of $\{\etae_1\}_{\e \in (0,\e_1)}$ in $C([0,T];\rr^d)$. Hence, it remains to show the tightness of $\{\eta_k^{\e, u^\e, \phi^\e}\} $ for $k=2,3$. For $\eta_3^{\e, u^\e, \phi^\e}$, we obtain 
\bea\ad 
\EE \Big(\sup_{0\leq t_1 \leq t_2 \leq T, t_2\leq t_1+h} \|\eta_3^{\e, u^\e, \phi^\e}(t_2)- \eta_3^{\e, u^\e,  \phi^\e}(t_1)\|^2 \Big) \\
\aad = \EE \Big( \sup_{0\leq t_1\leq t_2 \leq T, t_2\leq t_1+h} \Big\| \int_{t_1}^{t_2} \sg(X^{\e, u^\e, \phi^\e}(s), Y^{\e,u^\e,\phi^\e}(s)) u^\e(s)ds \Big\|^2 \Big) \\
\aad \leq \EE \sup_{0\leq t_1\leq T-h} \int_{t_1}^{t_1+h} \|\sg(X^{\e, u^\e, \phi^\e}(s), Y^{\e, u^\e, \phi^\e}(s))\|^2 ds \int_0^T \|u^\e(s)\|^2 ds \\
\aad \leq 2M h \EE\Big(1+ \sup_{0\leq t \leq T} \|X^{\e, u^\e, \phi^\e}(t)\|^2  \Big) \to 0, \quad \text{ as } h\to 0.
\eea
For $\etae_2$, \lemref{lem:bdd} gives for all $\e \in (0,\wdt \e)$,
\bea\ad
\EE\bigg(\sup_{\substack{0\leq t_1\leq t_2 \leq T \\ t_2\leq t_1+h}}\|\etae_2(t_2)-\etae_2(t_1)\|\bigg)\\
\aad =\EE \bigg(\sup_{\substack{0\leq t_1\leq t_2\leq T \\ t_2\leq t_1+h}} \Big\|\int_{t_1}^{t_2} \frac{1}{\sqe h(\e)}[\bar b(\Xe(s))-\bar b(\bar X(s))]ds\Big\| \bigg) \\
\aad = \EE \bigg(\sup_{\substack{0\leq t_1\leq t_2\leq T \\ t_2\leq t_1+h}} \int_{t_1}^{t_2} L_0 |\etae(s)|ds \bigg) \\
\aad \leq L_0 h\,  \EE\Big(\sup_{s\in [0,T]} \|\etae(s)\|\Big) \to 0, \text{ as } h \to 0.
\eea
Thus, we complete the proof.
\end{proof}

Given $\e>0$ and let $(u^\e, \phi^\e)\in \CU_b^+$, we note the equation \eqref{c-sde} indicates that only the controlled rates $\phi_{\Ye(t-),\cdot }^\e$ can influence the dynamics of $(\Xe, \Ye)$. Thus in proving the Laplace upper bound, without loss of generality, we assume that
\bea\!\!\!
\phi_{ij}^\e(t,z)=1, \forall\; (i,j)\in \TT \text{ such that } i\in \LL \setminus \{\Ye(t-)\}, \forall \; (t,z)\in [0,T]\times [0,\zeta].
\eea
Moreover, we set $\phi_{ij}^\e(t,z)=1$
 if $(i,j)\notin \TT$ by convention. This implies that
\beq{psi-conv} 
\psi_{ij}^\e(t,z):=\frac{\phi_{ij}^\e(t,z)-1}{\sqe h(\e)}=0, \text{ if } (i,j)\notin \TT \text{ or } (i,j)\in \TT \text{ with } i\in \LL \setminus \{\Ye(t-)\}.
\eeq
For $t\in [0,T]$ and a fixed $\beta \in(0,\infty)$, we adopt the approach from \cite{BDG16} to decompose $\psi_{ij}^\e$ into two terms: $\psi_{ij}^\e\indi_{\{|\psi_{ij}^\e|\leq \beta/\sqe h(\e)\}}$ and $\psi_{ij}^\e \indi_{\{|\psi_{ij}^\e|> \beta/\sqe h(\e)\}}$. We then define the occupation measure $\theta^\e(t):=(\theta_j^\e(t))_{j\in \LL}$  as
\beq{the} 
\theta_j^\e(t)(G)=\sum_{i \in \LL} \indi_{\{\Ye(t-)=i\}}  \int_G \psi_{ij}^\e(t,z)\indi_{\{|\psi_{ij}^\e(t,z)|\leq \beta/\sqe h(\e)\}} \la_\zeta(dz), \; G\in \CB[0,\zeta].
\eeq
Using this decomposition, $\theta^\e$ belongs to $(\CM_F[0,\zeta])^{|\LL|}$. 
Now define $\Pi^\e\in \PP_{\text{leb}}(\HH_T)$ by 
\beq{Pie}
\Pi^\e(A\times B\times C\times D) := \int_{[0,T]} \indi_A(s) \indi_B (\Ye(s)) \indi_{C}(\theta^\e(s)) \indi_{D}(u^\e(s))ds. 
\eeq

In proving the large deviation upper bound, the main step will be characterizing the limiting point of $\Pi^\e$. From \eqref{Pie}, we have $
\half \int_0^T \|u^\e(s)\|^2 ds =\half \int_{\HH_T} \|z\|^2 \Pi^\e(d\bv)$. Moreover, from \eqref{psi-conv} and \eqref{the}, we obtain the following:
\bea\ad 
\half \sum_{(i,j)\in \TT} \int_{[0,T]\times [0,\zeta]} |\psi_{ij}^\e(s,z)|^2  \la_\zeta(dz)ds  \\
\ad = \half \sum_{j\in \LL}\sum_{i\in \LL} \int_{[0,T]\times [0,\zeta]} \indi_{\{\Ye(s-)=i\}} |\psi_{ij}^\e(s,z)|^2  \la_{\zeta}(dz)ds  \\
\ad = \half \int_0^T 
\wdh{d}(\theta^\e(s)) ds\\
\aad\ + \half  \sum_{j\in \LL}\sum_{i\in \LL} \int_{[0,T]\times [0,\zeta]} \indi_{\{\Ye(s-)=i\}}|\psi_{ij}(s,z)|^2 \indi_{\{|\psi_{ij}^\e(s,z)|>\beta/(\sqe h(\e))\}} \la_\zeta(dz)ds \\
\ad = \half \int_{\HH_T} \wdh{d}(\theta) \Pi^\e(d\bv)\\
\aad\ +\half \sum_{j\in \LL}\sum_{i\in \LL} \int_{[0,T] \times [0,\zeta]} \indi_{\{\Ye(s-)=i\}}|\psi_{ij}(s,z)|^2 \indi_{\{|\psi_{ij}^\e(s,z)|>\beta/(\sqe h(\e))\}} \la_\zeta(dz)ds.
\eea
Consequently,
\beq{LTe}\barray\ad
\half \int_0^T  \| u^\e(s)\|^2 ds + \half \sum_{(i,j)\in \TT}\int_{[0,T]\times [0,\zeta]} |\psi_{ij}^\e(s,z)|^2 \la_\zeta(dz)ds\\
\ad =\int_{\HH_T}
\half \Big[ \|z\|^2 + \wdh d(\theta)\Big] \Pi^\e(d\bv) \\
\aad + \half \sum_{j\in \LL}\sum_{i\in \LL} \int_{[0,T]\times [0,\zeta]} \indi_{\{\Ye(s-)=i\}}|\psi_{ij}(s,z)|^2 \indi_{\{|\psi_{ij}^\e(s,z)|>\beta/(\sqe h(\e))\}} \la_\zeta(dz)ds.
\earray\eeq
To proceed, we prove the tightness of $\{\etae, \Pi^\e\}$ and characterize its limiting point.
\begin{prop}\label{prop:tight}
Fix $M\in (0,\infty)$. Define $\Pi^\e$ as in \eqref{Pie}. For $\e \in (0,1)$  and $(u^\e, \phi^\e)\in \CU_b^+$ such that $L_{T,1}(u^\e)\leq M$ and $L_{T,2}(\phi^\e)\leq M \e h^2(\e)$. Then the sequence $\{(\etae, \Pi^\e)\}_{\e \in (0,\wdt \e)}$ is a tight family of $C([0,T];\rr^d)\times \CM_F(\HH_T)$-valued random variables. For a weak limit $(\eta, \Pi)$ of $(\etae, \Pi^\e)$, we have 
\beq{Pi-M}
\int_{\HH_T}\half \big[\|z\|^2 +  \wdh d(\theta)\big] \Pi(d\bv)\leq M+\half \kappa_2(\beta) M,
\eeq
and the following holds:
\beq{eta-limit}\barray
\eta(t)\ad = \int_{\HH_t} \nabla \bar b(\bar X(s)) \eta(s)ds +  \int_{\HH_t} \sg(\bar X(s),y)z \Pi(d\bv) \\
\aad \quad + \int_{\HH_t} \sum_{j\in \LL} \big[\Phi(\bar X(s),j)-\Phi(\bar X(s),y)\big] \Ga_{yj}^\theta(\bar X(s)) \Pi(d\bv),
\earray\eeq
and 
\beq{limit-fast}
\int_{\HH_t} \Ga_{yj}(\bar X(s))\Pi(d\bv)=0, \quad \text{for all } j, \text{ and a.e. }  t \in [0,T].
\eeq
\end{prop}
\begin{proof}
The tightness of $\{\etae\}$ is given by \lemref{lem:tight}. We now prove the tightness of $\{\Pi^\e\}$. From \eqref{LTe}, the (iii) in  \lemref{lem:ineq}  implies that 
\beq{Pi-e}\barray\ad\!\!\!\!\!  
\int_{\HH_T}\Big[\half \|z\|^2 +\wdh d(\theta)\Big]\Pi^\e(d\bv)\\
\aad\!\!\!\!\! =\half \int_0^T \|u^\e(s)\|^2 ds +\half \sum_{(i,j)\in \TT}\int_{[0,T]\times [0,\zeta]} |\psi_{i j}^\e(s,z)|^2 \indi_{\{|\psi_{ij}^\e(s,z)|\leq \beta/(\sqe h(\e))\}} \la_\zeta(dz)ds \\
\aad\!\!\!\!\!  \leq \half \int_0^T \|u^\e(s)\|^2 ds + \half \frac{\kappa_2(\beta)}{\e h^2(\e)}\sum_{(i,j)\in \TT} \int_{[0,T]\times [0,\zeta]}\ell(\phi_{ij}^\e(s,z))\la_\zeta(dz)ds \\
\aad\!\!\!\!\!  = L_{T,1}(u^\e)+\half \frac{\kappa_2(\beta)}{\e h^2(\e)} L_{T,2}(\phi^\e) 
\leq M + \half \kappa_2(\beta) M <\infty.
\earray\eeq

To prove the tightness of $\{\Pi^\e\}$, it suffices to show that for any $\dl\in (0,\infty)$, there exists $K'>0$ such that $
\sup_{\e} \Pi^\e\big\{(s,y,\theta,z)\in \HH_T: \sum_{j\in \LL} \theta_j[0,\zeta] +\|z\|> K' \big\} \leq \dl$. 
This can be proved in a similar manner as \eqref{Pi-n}, by using \eqref{Pi-e} in place of \eqref{I-n}. Then the inequality of \eqref{Pi-M} follows from the Fatou's lemma and the lower semi-continuity of $\wdh d$. Consequently, for any subsequence of $\{\etae,\Pi^\e\}$, there exists a further subsequence of $\{\etae,\Pi^\e\}$ converging weakly to some limit $(\eta,\Pi)$. We still index the convergent subsequence by $\e$, with certain abuse of notation. By the Skorokhod representation theorem, we assume $\{\etae,\Pi^\e\}$ converges to $(\eta,\Pi)$ w.p.1. We now in the position to prove \eqref{eta-limit}. From the definition \eqref{the}, we have 
\bea\ad \!\!\!\!\!\!
N_{\beta}^{\psi^\e}(t)\\
\ad\!\!\!\!\!\! =\! \int_{[0,T]\times [0,\zeta]} \sum_{(i,j)\in \TT} \big[\Phi(\Xe(s), j)-\Phi(\Xe(s),i)\big]\\
\aad\qquad \times \indi_{\{\Ye(s-)=i\}} \indi_{E_{ij}(\Xe(s))}(z) \psi_{ij}^\e(s,z) \indi_{\{|\psi_{ij}(s,z)|\leq \beta/(\sqe h(\e))\}} \la_\zeta(dz)ds \\
\ad\!\!\!\!\!\! =\!\int_0^t \sum_{(i,j)\in \TT} \big[\Phi(\Xe(s),j)-\Phi(\Xe(s),i)\big]\\
\aad\qquad\qquad \times \indi_{\{\Ye(s-)=i\}} \theta_j^\e(E_{ij}(\Xe(s)))ds \\
\ad\!\!\!\!\!\! =\!\int_0^t\!\! \sum_{j\in \LL}\!\big[\Phi(\Xe(s),j)-\! \Phi(\Xe(s),\Ye(s-))\big]\Ga_{\Ye(s-)\, j}^{\theta^\e(s)}\! (\Xe(s))ds \\
\ad\!\!\!\!\!\! =\int_{\HH_t} \sum_{j\in \LL}\big[\Phi(\Xe(s),j)-\Phi(\Xe(s),y)\big]\Ga_{yj}^\theta(\Xe(s)) \Pi^\e(d\bv).
\eea
Moreover, one can obtain 
\bea\ad 
\int_0^t \sg(\Xe(s),\Ye(s))u^\e(s)ds = \int_{\HH_t} \sg(\Xe(s), y) z\Pi^\e(d\bv).
\eea
For the term $\etae_2$, applying the Taylor expansion, we have 
\beq{eta2-Tay}
\int_0^t \frac{1}{\sqe h(\e)} \big[\bar b(\Xe(s))-\bar b(\bar X(s))\big]ds = \int_0^t \nabla \bar b(\bar X(s))\etae(s)ds + \CR^\e(t),
\eeq
where $\CR^\e(t)$ is the reminder of Taylor expansion. Due to the boundedness of the derivative of $\bar b$, \lemref{lem:X-bdd} and \lemref{lem:bdd} implies $\sup_{t\in [0,T]} \CR^\e(t)\to 0$ in probability.  From \eqref{b-eta} and \lemref{lem:tight}, taking $\e\to  0$ and applying the argument of the proof of part (c) in \propref{prop:comp-le}, we have 
\beq{eta-lim}\barray
\eta(t)\ad = \int_{\HH_t} \nabla \bar b(\bar X(s)) \eta(s)ds + \int_{\HH_t} \sg(\bar X(s), y)z\Pi(d\bv) \\
\aad \quad + \int_{\HH_t} \sum_{j \in \LL}[\Phi(\bar X(s),j)-\Phi(\bar X(s),y)] \Ga^\theta_{yj}(\bar X(s)) \Pi(d\bv).
\earray\eeq
To prove \eqref{limit-fast}, let us consider any map $V$ from $\LL$ to $\rr$, we have
\beq{dig-fast}\barray\ad\!\!\!\!\!\! 
V(\Ye(t))-V(y_0)\\
\ad\!\!\!\!\!\! = \sum_{(i,j)\in \TT} \big(V_j -V_i \big) \int_{[0,t]\times [0,\zeta]} \indi_{\{\Ye(s-)=i\}} \indi_{E_{ij}(\Xe(s))}(z) N_{ij}^{\e^{-1}\phi_{ij}^\e}(dz\times ds)\\
\ad\!\!\!\!\!\! = \frac{1}{\e}\!\!\! \sum_{(i,j)\in \TT}\!\! (V_j-V_i)\!\! \int_0^t\!\! \int_0^\zeta \!\!\! \indi_{\{\Ye(s-)=i\}} \indi_{E_{ij}(\Xe(s))}(z) \phi_{ij}^\e(s,z)\la_\zeta(dz)ds + \wdh M_{\phi}^\e(t) \\
\ad\!\!\!\!\!\! =\frac{1}{\e}\!\! \sum_{(i,j)\in \TT} (V_j-V_i) \int_0^t\!\! \int_0^\zeta \indi_{\{\Ye(s-)=i\}} \indi_{E_{ij}(\Xe(s))}(z) [\phi_{ij}^\e(s,z)-1]\la_\zeta(dz)ds\\
\aad\!\!\!\!+ \frac{1}{\e} \sum_{(i,j)\in \TT}(V_j-V_i) \int_0^t\int_0^\zeta \indi_{\{\Ye(s-)=i\}} \indi_{E_{ij}(\Xe(s))}(z)\la_\zeta(dz)ds + \wdh M_{\phi}^\e(t),
\earray\eeq
where $V_i:=V(i)$ and  $\wdh M_{\phi}^\e$ is the martingale given by 
\bea\ad\!\!\!\!\!\!
\wdh M_{\phi}^\e(t)=\!\!\!\!\sum_{(i,j)\in \TT}\!\! (V_j-V_i)\! \int_0^t \int_0^\zeta \!\!\! \indi_{\{\Ye(s-)=i\}} \indi_{E_{ij}(\Xe(s))}(z) \wdt N_{ij}^{\e^{-1}\phi_{ij}^\e}(dz\times ds).
\eea
Multiplying $\e$ to both sides of \eqref{dig-fast}, we have 
\bea\ad\!\!\!\!\!\!\! 
\e (V(\Ye(t))-V(y_0)) \\
\aad\!\!\!\!\!\!\!  = \!\!\! \sum_{(i,j)\in \TT}\!\! (V_j-V_i)\int_0^t\!\! \int_0^\zeta \indi_{\{\Ye(s-)=i\}} \indi_{\{E_{ij}(\Xe(s))\}}(r) \sqe h(\e)\psi_{ij}^\e(s,z)\la_\zeta(dz)ds \\
\aad\!\!\!\!\!\!\!  +\sum_{(i,j)\in \TT} (V_j-V_i) \int_0^t\!\! \int_0^\zeta \indi_{\{\Ye(s-)=i\}} \indi_{\{E_{ij}(\Xe(s))\}}(r) \la_\zeta(dz)ds+ \e \wdh M_{\phi}^\e(t). 
\eea
Then the Doob inequality gives 
\bea \disp
\EE\Big(\sup_{t\in [0,T]} \Big\|\e \wdh M_{\phi}^\e(t)\Big\|^2\Big) 
\ad \leq 4 \e \sum_{(i,j)\in \TT} (V_j-V_i)^2\int_{[0,T]\times [0,\zeta]} \phi_{ij}^\e(s,z)\la_\zeta(dz)ds \\
\aad \leq 16 \e \|V\|_\infty^2 \sum_{(i,j)\in \TT} \int_{[0,\zeta]\times[0,T]} \phi_{ij}^\e(s,z)\la_\zeta(dz)ds\\
\aad \leq 16\e \|V\|_\infty^2 (\zeta T e+M\e h^2(\e))\to 0, \quad \text{ as } \e\to 0,
\eea
where the last inequality uses (i) in \lemref{lem:ineq} and $L_{T,2}(\phi^\e)\leq M\e h^2(\e)$.
Thus,
\beq{e-M}
\sup_{t\in [0,T]} \|(\e \wdh  M_{\phi}^\e)(t)\| \to 0, \text{in probability as } \e \to 0. 
\eeq

Furthermore, \lemref{lem:ineq-2} and H\"{o}lder inequality yields
\beq{E-F}\barray\ad\!\!\!\!\!\! 
\EE\!\!\sup_{t\in [0,T]}\!\Big|\!\!\! \sum_{(i,j)\in \TT}\!\!\!\!(V_j-V_i) \!\! \int_0^t\!\!\! \int_0^\zeta\!\! \indi_{\{\Ye(s)=i\}}\!\indi_{E_{ij}(\Xe(s))}(z)\sqe h(\e) \psi_{ij}^\e(s,z)\la_\zeta(dz)ds\Big| \\
\aad\!\!\!\!\!\! \leq 2 \sqe h(\e) \|V\|_\infty \EE\!\! \sum_{(i,j)\in \TT} \int_0^T\!\! \int_0^\zeta \indi_{\{\Ye(s)=i\}} \indi_{E_{ij}(\Xe(s))}(z) |\psi_{ij}^\e(s,z)| \la_\zeta(dz)ds \\
\aad\!\!\!\!\!\! \leq 2\sqe h(\e) \|V\|_\infty \sum_{(i,j)\in \TT} \int_{[0,T]\times [0,\zeta]} |\psi_{ij}^\e(s,z)|\indi_{\{|\psi_{ij}^\e(s,z)|\leq \beta/\sqe h(\e)\}}\la_\zeta(dz)ds \\
\aad\quad\!\!\!\!\!\! + 2\sqe h(\e) \|V\|_\infty \sum_{(i,j)\in \TT} \int_{[0,T]\times [0,\zeta]} |\psi_{ij}^\e(s,z)| \indi_{\{|\phi_{ij}^\e(s,z)|> \beta/\sqe h(\e)\}} \la_\zeta(dz)ds \\
\aad\!\!\!\!\!\! \leq 2\sqe h(\e) \|V\|_\infty |\LL|^2 \Big(M\sqe h(\e) \kappa_1(\beta) +  \sqrt{M|\zeta|T \kappa_2(\beta)}
\Big) \to 0, \text{ as } \e\to 0.
\earray\eeq
For $j\in \LL$, take $V=\indi_{\{j\}}$, \eqref{e-M} and \eqref{E-F} give that 
\bea\ad \!\!\!
\sup_{t\in [0,T]} \Big| \sum_{i\in \LL} \int_0^t  \indi_{\{\Ye(s)=i\}} \Ga_{ij}(\Xe(s))ds \Big|\\
\aad\!\!\!\! =\sup_{t\in [0,T]} \Big|\int_0^t \Ga_{\Ye(s),j}(\Xe(s)))ds\Big| = \sup_{t\in [0,T]} \Big| \int_{\HH_t} \Ga_{yj}(\Xe(s))\Pi^\e(d\bv)\Big|
\eea
converges to $0$ as $\e\to 0$. Consequently, $\int_{\HH_t} \Ga_{yj}(\Xe(s))\Pi^\e(d\bv) \to 0$, uniformly in $t \in [0,T]$. Since  $\int_{\HH_t} \Ga_{yj}(\Xe(s))\Pi^\e(d\bv) \to \int_{\HH_t}\Ga_{yj}(\bar X(s))\Pi(d\bv)$, we have \eqref{limit-fast}. The proof is complete.
\end{proof}

\section{Large deviation lower bound}\label{sec:up}
In this section, we establish the Laplace principle lower bound \eqref{lap-up} which is equivalent to large deviation lower bound, using the variational representations \eqref{var-rep}. 

\begin{thm}\label{thm:up}
For any bounded and continuous function $F: C([0,T];\rr^d)\to \rr$, the following Laplace principle lower bound holds:
\beq{laplace-up} 
\liminf_{\e\to 0}-\frac{1}{h^2(\e)} \log\EE\Big[\exp\big(-h^2(\e) F(\eta^\e)\big)\Big] \geq \inf_{\eta\in C([0,T];\rr^d)}[I(\eta)+F(\eta)].
\eeq
\end{thm}
\begin{proof}
From the representation formula \eqref{var-rep}, for every $\e>0$, we can find $(\wdt u^\e, \wdt \phi^\e)\in \CU_b^+$ such that 
\beq{up-e}\barray\ad
-\frac{1}{h^2(\e)} \log\EE\Big[\exp(-h^2(\e)F(\eta^\e))\Big]\\
\aad \geq  \EE\Big[L_{T,1}(\wdt u^\e)+\frac{1}{\e h^2(\e)} L_{T,2}(\wdt \phi^\e)+ F(\eta^{\e,\wdt u^\e,\wdt \phi^\e})\Big]-\e.
\earray\eeq
Since $F$ is bounded, then \eqref{up-e} implies that 
\bea\ad 
\sup_{\e>0}\EE\Big[L_{T,1}(\wdt u^\e)+\frac{1}{\e h^2(\e)} L_{T,2}(\wdt \phi^\e)\Big] \leq 2 \|F\|_\infty +1:=\wdt K< \infty.
\eea
Fix $\dl>0$, we define the stopping time
\bea\ad 
\tau^\e=\inf\Big\{t\in [0,T]: \frac{1}{\e h^2(\e)} L_{t,2}(\wdt \phi^\e) > \frac{2\wdt K \|F\|_\infty}{\dl} \Big\} \wedge T.
\eea
Then for $(s,z)\in [0,T]\times [0,\zeta]$, let $
\phi^\e(s,z)=\wdt \phi^\e(s,z) \indi_{\{s\leq \tau^\e\}}+ \indi_{\{s>\tau^\e\}}$. 
We note that $\phi^\e\in \CU_b^+$ and $L_{T,2}(\phi^\e)/(\e h^2(\e))\leq (2\wdt K \|F\|_\infty/\dl)$. Moreover,
\bea
\PP(\phi^\e \neq \wdt \phi^\e) 
\ad \leq \PP\Big(\frac{1}{\e h^2(\e)} L_{T,2}(\wdt\phi^\e) > \frac{2\wdt K \|F\|_\infty}{\dl} \Big) \leq \frac{\dl\EE L_{T,2}(\wdt \phi^\e)}{\e h^2(\e) 2\wdt K \|F\|_\infty} \leq \frac{\dl}{2\|F\|_\infty}.
\eea
For any $(s,z)\in [0,T]\times [0,\zeta]$, we then define
\bea\ad 
\wdt \psi_{ij}^\e(s,z):=\frac{\wdt \phi_{ij}^\e(s,z)-1}{\sqe h(\e)} \text{ and } \psi_{ij}^\e(s,z):=\frac{\phi^\e(s,z)-1}{\sqe h(\e)}=\wdt \psi_{ij}^\e(s,z)\indi_{\{s\leq \tau^\e\}}.
\eea
Fix $\beta\in (0,1]$. By 
\eqref{up-e} and $\kappa_2(1)\geq \kappa_2(\beta)$, we have 
\beq{up-e1}\barray\ad
-\frac{1}{h^2(\e)} \log\EE\Big[\exp \big(-h^2(\e) F(\eta^\e)\big) \Big] \\
\aad \geq \EE\Big[ L_{T,1}(\wdt u^\e)+ \frac{1}{\e h^2(\e)} L_{T,2}(\wdt \phi^\e)+F(\eta^{\e, \wdt{u}^\e, \wdt{\phi}^\e}) \Big]-\e \\
\aad\geq \EE\Big[L_{T,1}(\wdt u^\e) + F(\eta^{\e, \wdt{u}^\e, \wdt{\phi}^\e})\Big]-\e \\
\aad\quad + \EE \Big[\frac{1}{\e h^2(\e)} \sum_{(i,j)\in \TT}\int_{[0,T]\times [0,\zeta]}  \ell(\phi_{ij}^\e)\indi_{\{|\psi_{ij}^\e|\leq \beta/(\sqe h(\e))\}} \la_\zeta(dz)ds\Big] \\
\aad \geq \EE\Big[L_{T,1}(\wdt u^\e) + F(\eta^{\e, \wdt u^\e, \phi^\e})+ F(\eta^{\e, \wdt u^\e, \wdt{\phi}_\e})-F(\eta^{\e, \wdt u^\e, \phi^\e}) \Big]-\e \\
\aad\quad  + \EE\Big[ \half \sum_{(i,j)\in \TT} \int_{[0,T]\times [0,\zeta]} \big[(\psi_{ij}^\e)^2 -\kappa_3  \sqe h(\e) |\psi_{ij}^\e|^3 \big] \indi_{\{|\psi_{ij}^\e|\leq \beta/(\sqe h(\e))\}} \la_\zeta(dz)ds\Big]\\

\aad\geq \EE \Big[L_{T,1}(\wdt u^\e)+\half \sum_{(i,j)\in \TT} \int_{[0,T]\times [0,\zeta]} (\psi_{ij}^\e(s,z))^2 \indi_{\{|\psi_{ij}^\e|\leq \beta/(\sqe h(\e)) \}} \la_\zeta(dz)ds \Big] \\
\aad \quad + \EE\big[F(\eta^{\e, \wdt u^\e, \phi^\e})\big] -\e -\dl -\half \beta \kappa_3 |\LL|^2 M \kappa_2(1) \\
\aad= \EE\Big[\int_{\HH_T}\half \Big[\|z\|^2 + \wdh d(\theta) \Big]\Pi^\e(d\bv) + F(\eta^{\e, \wdt u^\e, \phi^\e})  \Big]-\e -\dl -\half \beta \kappa_3 |\LL|^2 M \kappa_2(1),
\earray\eeq
where the second line follows from \eqref{up-e}. The third line is due to the definition of $\phi^\e$. Line four is from \lemref{lem:ineq} (iv). Line six is due to the following inequality
\bea
|\EE[F(\eta^{\e, \wdt u^\e, \wdt \phi_\e})- F(\eta^{\e, \wdt u^\e, \phi^\e})]| \leq 2 \|F\|_\infty \PP\Big(\phi^\e\neq \wdt \phi^\e \Big) \leq \dl,
\eea
and \eqref{psi-sq}. The last line follows from the definition of $\Pi^\e$ in \eqref{Pie} using $(\wdt u^\e, \phi^\e)$.  With the control $(\wdt u^\e,\phi^\e)$, \propref{prop:tight} implies that $(\eta^{\e, \wdt u^\e, \phi^\e}, \Pi^\e)$ is tight and any limit point $(\eta,\Pi)$ satisfies \eqref{eta-limit} and \eqref{limit-fast}. Thus $\Pi \in \CP_s(\eta)$. Taking liminf to \eqref{up-e1} yields
\bea\ad 
\liminf_{\e\to 0}-\frac{1}{h^2(\e)} \log\EE \Big[\exp\big(-h^2(\e) F(\eta^\e)\big)\Big] \\
\aad \geq \EE\Big[ \int_{\HH_T}\half \Big[ \|z\|^2 + \wdh d(\theta) \Big] \Pi(d\bv)+ F(\eta)\Big]-\dl -\half \beta \kappa_3 |\LL|^2 M \kappa_2(1) \\
\aad \geq \EE\Big[\inf_{\Pi\in \CP_s(\eta)} \Big\{\int_{\HH_T} \half \Big[ \|z\|^2 +\wdh d(\theta)\Big] \Pi(d\bv) \Big\} +F(\eta)\Big]-\dl-\half \beta \kappa_3 M \kappa_2(1) \\
\aad \geq \EE[I(\eta)+F(\eta)]-\dl -\half \beta \kappa_3 |\LL|^2 M \kappa_2(1) \\
\aad \geq \inf_{\eta\in C([0,T];\rr^d)} [I(\eta)+F(\eta)]-\dl -\half \beta \kappa_3 |\LL|^2 M \kappa_2(1),
\eea
where the first inequality is due to Fatou's lemma and \propref{prop:tight}. The second inequality uses the property that $\Pi \in \CP_s(\eta)$ a.s., and the third inequality follows from the definition of $I$. Finally, sending $\dl$ and $\beta$ to $0$ yields \eqref{laplace-up}.We finish the proof.
\end{proof} 

\section{Large deviation upper bound}\label{sec:low}
In this section, we establish the large deviation upper bound by proving the Laplace principle upper bound. To achieve this, we consider a path $\eta$ that is near the infimum of the right side of \eqref{lap-lower} and construct a sequence of controls $(u^\e, \phi^\e)$ such that $\eta^{\e, u^\e,\phi^\e}$ converges weakly to $\eta$, where $\eta^{\e,u^\e,\phi^\e}$ is defined \eqref{def-etae}. By using this sequence of controls and applying variational representations, we demonstrate the convergence of cost, leading to the result in \eqref{lap-lower}. For a nearly optimal $\eta$, let $(u,\psi,\pi)\in \CV(\eta)$ represent the near-infimum for the expression in \eqref{rate}. The control pair $(u,\psi)$ naturally defines the sequence of controls $(u^\e,\phi^\e)$, which will be used to define $\etae$ and $\Pi^\e$. Next, we need to show that any limit point $(\bar\eta, \bar\Pi)$ of the sequence $(\etae,\Pi^\e)$ solves \eqref{eta-def} and \eqref{fast-cond}. By arguing the uniqueness of solutions of \eqref{eta-def} for a given pair $(u,\psi)$, we can prove that $(\bar\eta,\bar\Pi)=(\eta,\Pi)$ a.s. To proceed, we require the following result.
\begin{prop}\label{prop:low}
Let $\eta \in C([0,T];\rr^d)$ be such that $I(\eta)<\infty$. Fix $\ga\in (0,1)$, then there exists a $\eta^* \in C([0,T];\rr^d)$ such that $
\|\eta-\eta^*\|_T: =\sup_{t\in [0,T]}\|\eta(t)-\eta^*(t)\| \leq \ga$, 
and there exists $(u^*, \psi^*=(\psi_{ij}^*), \pi^*=(\pi_i^*))\in \CV(\eta^*)$ such that (i) $\pi_i^*(s)$ has a lower bound $\underline{\mu}\dl$ for some $\dl>0$, where $\underline{\mu}$ is defined in \eqref{und-mu}; (ii) for given $(u^*,\psi^*)$, if equation \eqref{eta-def} and \eqref{fast-cond} are satisfied by any other pair $(\wdt\eta,\wdt\pi)$, then $(\wdt\eta,\wdt\pi)=(\eta^*,\pi^*)$; (iii) for each $j\in \LL, s\in [0,T]$, $
\sum_{i\in \LL}\pi_i^*(s)\Ga_{ij}(\bar X(s))=0$, and the cost associated with $(u^*,\psi^*)$ satisfies
\beq{cost-star}\barray\ad\!\!\! 
\sum_{i\in \LL} \half \int_0^T \|u_i^*(s)\|^2 \pi_i^*(s)ds +\half\! \sum_{(i,j)\in \TT} \int_0^T\!\!\! \int_0^\zeta |\psi_{ij}^*(s,z)|^2 \pi_i^*(s) \la_\zeta(dz) ds \leq I(\eta)+\ga.
\earray\eeq
\end{prop}
\begin{rem}\rm{
An analogue of above proposition can be found in \cite[Theorem 4.1]{BDG18} for large deviations. Their situation is more complicated than ours in that the second restriction \eqref{ldp-Ga} in the rate function rely on both the state process $\xi$ and the control $\phi$. The merely simple small perturbation $(u, \phi,\pi)$ to $(u^\dl,\phi^\dl,\pi^\dl)$, so the small perturbation of $\xi$ to $\xi^\dl$, cannot make $(u^\dl, \phi^\dl, \pi^\dl)$ belong to $\CV(\xi^\dl)$ in general. Further modifications are needed; see \cite[Remark 4.2]{BDG18}. In our situation, the restriction \eqref{fast-cond} is independent of the state $\eta$ and the control $\psi$. Our approximation is relatively simpler.}
\end{rem}
\begin{proof}
Let $\eta\in C([0,T];\rr^d)$ be such that $I(\eta)<\infty$. Fix $\ga\in (0,1)$, let $(u,\psi,\pi)\in \CV(\eta)$ satisfy
\beq{I-ga}
\sum_{i\in \LL} \half \int_0^T \|u_i(s)\|^2 \pi_i(s)ds + \half \sum_{(i,j)\in \TT} \int_{[0,T]\times [0,\zeta]} |\psi_{ij}(s,z)|^2 \pi_i(s) \la_\zeta(dz) ds  \leq I(\eta)+\frac{\ga}{2}.
\eeq

Without loss of generality, we can assume for some $\wdt\kappa\in (0,\infty)$,
\beq{pi-psi}
\sup_{(s,z)\in [0,T]\times [0,\zeta]} \max_{(i,j)\in \TT} (\psi_{ij}(s,z)\pi_i(s)) \leq \wdt\kappa.
\eeq

We will first prove the proposition based on the claim mentioned above, with the proof of the claim deferred to the end. For $x\in \rr^d$, let $\mu(x)$ denote the stationary distribution of the fast-varying Markov chain when the slow process is fixed at $x\in \rr^d$. Given a fixed $\dl>0$, we define $
\pi_j^\dl(s):=(1-\dl) \pi_j(s)+\dl \mu_j(\bar X(s))$. 
It follows that
\beq{diff-pi-dl}
\sup_{t\in [0,T]}\sum_{j} |\pi_j^\dl(s)-\pi_j(s)|\leq 2\dl.
\eeq

It is worthy noting that, although the second condition in \eqref{fast-cond} is independent of the dynamics of $\eta$, it is still necessary to perturb the $\pi$ to $\pi^\dl$ in order to obtain a lower bound for $\pi^\dl$. 
We then define $
u_j^\dl (s):=u_j(s)\frac{\pi_j(s)}{\pi_j^\dl(s)}$, $s\in [0,T], j\in \LL.$
Thus $u_j^\dl(s)\pi_j^\dl(s)=u_j(s)\pi_j(s)$.
In addition, for $(i,j)\in \TT$ and $(s,z)\in [0,T]\times [0,\zeta]$, define
\bea\ad 
\psi_{ij}^\dl(s,z)=(1-\dl) \frac{\pi_j(s)}{\pi_i^\dl(s)}\psi_{ij}(s,z) +\dl \frac{\mu_{i}(\bar X(s))}{\pi_i^\dl(s)}.
\eea
Consequently,
\beq{diff-pi-psi}\barray\ad 
|\pi_i^\dl(s)\psi_{ij}^\dl(s,z)-\pi_i(s)\psi_{ij}(s,z)|\\
\aad \leq \dl \sup_{(s,z)\in [0,T]\times [0,\zeta]} \max_{(i,j)\in \TT}|\psi_{ij}(s,z)\pi_i(s)|+\dl |\mu_i(\bar X(s))|  \leq \dl(\wdt\kappa+1).
\earray\eeq
Define
\bea
\eta^\dl(t)\ad\!\!\!\!\! =\int_0^t \nabla \bar b(\bar X(s)) \eta^\dl(s) ds + \sum_{j\in \LL} \int_0^t \sg_j(\bar X(s)) u_j^\dl(s) \pi_j^\dl(s)ds \\
\aad \!\!\!\!\!\!\!\!\! + \sum_{(i,j)\in \TT} \int_{[0,\zeta] \times [0,T]} [\Phi(\bar X(s),j)-\Phi(\bar X(s),i)] \pi_i^\dl(s) \indi_{E_{ij}(\bar X(s))}(r) \psi_{ij}^\dl(s,z) \la_\zeta(dz)ds.
\eea
From above and our assumptions, the Gronwall inequality implies $
\sup_{t\in [0,T]} \|\eta(t)-\eta^\dl(t)\| \leq K \dl$. 
In addition, from \eqref{mu-Q}, we have for each $j\in \LL$ and $s\in [0,T]$
\bea\disp
\sum_{i\in \LL}\pi_i^\dl(s)\Ga_{ij}(\bar X(s))\ad = (1-\dl)\sum_{i\in \LL} \pi_i(s)\Ga_{ij}(\bar X(s))+\dl \sum_{i\in \LL}\mu_i(\bar X(s))\Ga_{ij}(\bar X(s)) \\
\ad = \dl \sum_{i\in \LL}\mu_i(\bar X(s)) q_{ij}(\bar X(s))=0.
\eea
Thus $\pi^\dl(s) \Ga(\bar X(s))=0$ implies $\pi^\dl(s)$ is stationary for $\Ga(\bar X(s))$. Hence $(u^\dl,\psi^\dl,\pi^\dl) \in \CV(\eta^\dl)$. Following the elementary inequality $|a|^2 - |b|^2 \leq ||a|^2-|b|^2| \leq |a-b|^2$ and previous estimates, we can get
\bea\ad 
\|\psi_{ij}^\dl(s,z)\|^2\pi_i^\dl(s)-|\psi_{ij}(s,z)|^2 \pi_i(s) \\
\aad =\|\psi_{ij}^\dl(s,z)\|^2 \pi_i^\dl(s)-\frac{1}{\pi_i^\dl(s)} \|\psi_{ij}(s,z)\pi_i(s,z)\|^2 \\
\aad \quad +\frac{1}{\pi_i^\dl(s)} \|\psi_{ij}(s,z)\pi_i(s)\|^2-|\psi_{ij}(s,z)|^2 \pi_i(s) \\
\aad \leq \frac{1}{\pi_i^\dl(s)}\|\pi_{ij}^\dl(s,z)\pi_i^\dl(s)-\psi_{ij}(s,z)\pi_i(s)\|^2 + |\psi_{ij}(s,z)|^2 \pi_i(s)\big|\frac{\pi_i(s)}{\pi_i^\dl(s)}-1\big| \\
\aad \leq \frac{\dl^2 (\wdt\kappa+1)^2}{\underline{\mu}}+ |\psi_{ij}(s,z)|^2 \pi_i(s) \frac{2\dl}{\underline{\mu}},
\eea
where the last inequality is due to \eqref{diff-pi-dl}, \eqref{diff-pi-psi}, and $\pi_i^\dl(s)\geq \underline{\mu}\dl$. Thus,
\bea\ad 
\half \sum_{(i,j)\in \TT} \int_{[0,T]\times [0,\zeta]} \|\psi_{ij}^\dl (s,z)\|^2 \pi_i(s)\la_\zeta(dz)ds \\
\aad \leq \sum_{(i,j)\in \TT} \int_{[0,T]\times [0,\zeta]} |\psi_{ij}(s,z)|^2 \pi_i(s) \la_\zeta(dz)ds+\frac{\dl^2(\wdt\kappa+1)^2}{2\underline{\mu}}\zeta T |\LL|^2+ \frac{\dl}{\underline{\mu}} M,
\eea
where $M=I(\eta)+1$. Furthermore, we note that if $\pi_i(s)\leq \pi_i^\dl(s)$, 
\bea
\|u_i^\dl(s)\|^2 \pi_i^\dl(s)=\|u_i(s)\|^2 \pi_i(s)\frac{\pi_i(s)}{\pi_i^\dl(s)}\leq \|u_i(s)\|^2 \pi_i(s),
\eea
and if $\pi_i(s)\geq \pi_i^\dl(s)$, then $\pi_i(s)\geq \mu_i(\bar X(s))\geq \underline{\mu}$. Thus, 
\bea
\|u_i^\dl(s)\|^2\pi_i^\dl(s)
\ad \leq \|u_i(s)\|^2 \pi_i(s)+ \|u_i(s)\|^2 \pi_i(s) \Big|\frac{\pi_i(s)}{\pi_i^\dl(s)}-1\Big| \\
\aad \leq \|u_i(s)\|^2 \pi_i(s)+\frac{2\dl}{\underline{\mu}} \|u_i(s)\|^2 \pi_i(s).
\eea
It follows that 
\bea\ad 
\half \sum_{i \in \LL}\int_0^T \|u_i^\dl(s)\|^2 \pi_i^\dl(s)ds \leq \sum_{i\in \LL} \half \int_0^T \|u_i(s)\|^2 \pi_i(s) ds + \frac{\dl}{\underline{\mu}} M.
\eea
Take $\dl^*=\min \{\ga \underline{\mu}/4M, \sqrt{\ga \underline{\mu}/(\zeta T |\LL|^2 (\wdt\kappa+1)^2)}\}$,  with 
\bea
(\eta^*,u^*,\psi^*,\pi^*):=(\eta^{\dl^*}, u^{\dl^*}, \psi^{\dl^*}, \pi^{\dl^*}),
\eea
then $(u^*,\psi^*,\pi^*)\in \CV(\eta^*)$. This choice of $\dl^*$ and \eqref{I-ga} yield that
\bea\ad 
\half \sum_{i\in \LL} \int_0^T \| u_i^*(s)\| \pi_i^*(s)ds + \half \sum_{(i,j)\in \TT} \int_{[0,T]\times [0,\zeta]} |\psi_{ij}^*(s,z)|^2 \pi_i^*(s) \la_\zeta(dz)ds \\
\aad \leq \half \sum_{i\in \LL} \int_0^T \|u_i(s)\|^2 \pi_i(s) ds + \half \sum_{(i,j)\in \TT} \int_{[0,T]\times [0,\zeta]} |\psi_{ij}(s,z)|^2 \pi_i(s) \la_\zeta(dz)ds+ \ga/2 \\
\aad \leq I(\eta)+\ga.
\eea

Moreover, if \eqref{eta-def} and \eqref{fast-cond} are satisfied by another pair $(\wdt\eta,\wdt\pi)$ with control $(u^*,\psi^*)$, it follows that $\wdt\pi=\wdt\pi^*$. By the Lipschitz continuity in (H1) and the boundedness of $\Phi$, we can also prove $\wdt\eta=\eta^*$. This proves statements (i)-(iii).

We now turn to prove the claim in \eqref{pi-psi}. First, we observe that the dynamics of $\eta$ remain unchanged if we redefine the control $\psi$ as follows: with $\psi$ on the right equal to the old version and $\psi$ on the left the new, for $(i,j)\in \TT$, let 
\bea
\psi_{ij}(s,z)=\left\{\barray
\Ga^{\psi_i}_{ij}(\bar X(s))/\Ga_{ij}(\bar X(s)),\ad \quad  z \in E_{ij}(\bar X(s)), \\
0,\ad \quad z \in [0,\zeta] \setminus E_{ij}(\bar X(s)).
\earray\right. 
\eea

This can be interpreted as follows: within $E_{ij}(\bar X(s))$, the controlled jump rates are constant in $r$; outside $E_{ij}(\bar X(s))$, the controlled jump rates are the same. The overall jump rates does not change, which serves to reduce the cost while maintaining the dynamics. Let $\nu(s):=\sum_{(i,j)\in \TT} \pi_i(s) \Ga_{ij}^{\psi_i}(\bar X(s))$, and $\lbar\Ga_{ij}^{\psi_i}(\bar X(s)):=\Ga_{ij}^{\psi_i}(\bar X(s))/\nu(s)$. For $\al>0$ and $(i,j)\in \TT$, let
\bea
\psi_{ij}^{\al}(s,z)=\left\{\barray\disp 
\al \frac{\lbar\Ga_{ij}^{\psi_i}(\bar X(s))}{\Ga_{ij}(\bar X(s))}, \ad \quad z \in E_{ij}(\bar X(s)), \\
\disp
\frac{\al}{\nu(s)}-1, \ad \quad  z \in [0,\zeta]\setminus E_{ij}(\bar X(s)).
\earray\right. 
\eea
When $\al=\nu(s),  \psi^\al=\psi$. We obtain 
\bea\ad 
\!\!\!\!\!\!\! \inf_{\al>0}\!\! \sum_{(i,j)\in \TT}\! \half \pi_i(s)\!\! \int_{[0,\zeta]} \!\! |\psi_{ij}^\al(s,z)|^2 \pi_i(s) \la_\zeta(dz) \leq\!\!\! \sum_{(i,j)\in \TT}\! \half \pi_i(s) \int_{[0,\zeta]}\!\! |\psi_{ij}(s,z)|^2 \pi_i(s) \la_\zeta(dz).
\eea

We now compute the infimum on the left side. Set $\varrho_{ij}:=\zeta-\la_\zeta(E_{ij}(\bar X(s)))$, then $1\leq \varrho_{ij} \leq  \zeta$. Take differentiation w.r.t. $\al$ and set the derivative to $0$, we get
\bea\ad  
\al \sum_{(i,j)\in \TT} \pi_i(s)\bigg[\bigg(\frac{\lbar\Ga^{\psi_i}_{ij}(\bar X(s))}{\Ga_{ij}(\bar X(s))}\bigg)^2 (\zeta-\varrho_{ij})+\frac{1}{\nu^2(s)} \varrho_{ij}\bigg] =\sum_{(i,j)\in \TT} \pi_i(s)\frac{\varrho_{ij}}{\nu(s)}.
\eea
Thus
\bea\disp
\al=\frac{\disp \sum_{(i,j)\in \TT} \pi_i(s)\Big[\big(\lbar\Ga^{\psi_i}_{ij}(\bar X(s))/\Ga_{ij}(\bar X(s))\big)^2 (\zeta-\varrho_{ij})+\frac{1}{\nu^2(s)} \varrho_{ij}\Big]}{\disp \sum_{(i,j)\in \TT} \pi_i(s)\varrho_{ij}/(\nu(s))}.
\eea

Thus, we can verify that there exists a constant $\al_0\in (0,\infty)$ such that for all $s\in[0,T]$, $\al\leq \al_0$. From the definition of $\psi^\al$, we observe that $\pi_i(s)\bar\Ga_{ij}^{\psi_i}(\bar X(s))\leq 1$. Therefore, $
\pi_i(s) \psi_{ij}^\al(s,z)\leq \al_0/\underline{\kappa}_Q=:\wdt\kappa$.
This confirms the claim, completing the proof of the proposition.
\end{proof}

\begin{thm}\label{thm:ld-lower}
For any continuous and bounded function $F:C([0,T];\rr^d) \to \rr$, the following Laplace principle upper bound holds:
\beq{ld-lower}
\liminf_{\e\to 0} - \frac{1}{h^2(\e)} \log\EE\Big[\exp\big(-h^{-2}(\e) F(\eta^\e)\big)\Big] \leq \inf_{\eta\in C([0,T];\rr^d)}[I(\eta)+F(\eta)].
\eeq
\end{thm}
\begin{proof}
From \cite[Corollary 1.2.5]{DE97}, without loss of generality, the function $F$ can be assumed to be Lipschitz continuous. That is, there exists $K_F\in (0,\infty)$ such that
\bea
|F(\eta_1)-F(\eta_2)| \leq K_F \|\eta_1-\eta_2\|_T, \quad \text{for all } \eta_1, \eta_2 \in C([0,T];\rr^d). 
\eea

Moreover, the infimum on the right-hand side of \eqref{ld-lower} is finite; otherwise, \eqref{ld-lower} is trivial. Fix $\ga\in (0,1)$, the definition of infimum implies there exists a $\eta\in C([0,T];\rr^d)$  such that $
I(\eta)+F(\eta) \leq \inf_{\eta\in C([0,T];\rr^d)} [I(\eta)+ F(\eta)]+\frac{\ga}{2}$. 
Since $I(\eta)<\infty$, \propref{prop:low} then implies there exists $(u^*,\psi^*, \pi^*)\in \CV(\eta^*)$, $\eta^* \in C([0,T];\rr^d)$ such that
\beq{I-eta-star}\barray\ad 
\!\!\!\!\!\! \sum_{i\in \LL}\half \int_0^T\!\! \|u_i^*(s)\|^2 \pi_i^*(s)ds + \half \!\! \sum_{(i,j)\in \TT} \int_{[0,\zeta] \times [0,T]} \!\! |\psi_{ij}^*(s,z)|^2 \pi_i^*(s) \la_\zeta(dz)ds \leq I(\eta) + \frac{\ga}{2},
\earray\eeq
and for $i\in\LL$ and $(i,j)\in \TT$, $ 
\int_0^T \|u_i^*(s)\|^2 \pi_i^*(s) ds < \infty$ and $\int_{[0,\zeta] \times [0,T]}\|\psi_{ij}^*(s,z)\|^2 \pi_i^*(s)\\ \la_\zeta(dz)ds <\infty.$ Here, for such $(u^*,\psi^*)$, the pair $(\eta^*, \pi^*)$ is unique by statement (ii) in \propref{prop:low}. For $\beta \in (0,1]$, define
\beq{psi-star} 
\psi^{*,\e}_{ij} :=\psi^*_{ij} \indi_{\{|\psi^*_{ij}|\leq \beta/\sqe h(\e)\}}, \quad \phi^{*,\e}_{ij} :=1+\sqe h(\e) \psi^{*,\e}_{ij}, \quad (i,j)\in \TT.
\eeq

Let $(\Xe, \Ye)$ be the solution of \eqref{c-sde}, where the controls $(u^\e, \phi^\e)\in \CU_b^+$ are defined in the following feedback form: for $(i,j)\in \TT$ and $s\in [0,T]$,
\beq{u-phi}\barray
u^\e(s)\ad :=\sum_{j=1}^{|\LL|} \indi_{\{\Ye(s-)=j\}} u_j^*(s),\\
\phi_{ij}^{\e}(s)\ad :=\phi_{ij}^{*,\e}(s)\indi_{\{\Ye(s-)=i\}}+ \indi_{\{\Ye(s-)\neq i\}}.
\earray
\eeq
Using the control pair $(u^\e,\phi^\e)$ defined above, the variational representation \eqref{var-rep} gives 
\bea\ad 
-\frac{1}{h^2(\e)}\log \EE\big[-h^2(\e) F(\eta^\e)\big]\\
\aad \leq \EE\Big[L_{T,1}(u^\e)+\frac{1}{\e h^2(\e)} L_{T,2}(\phi^\e)+F (\eta^{\e, u^\e, \phi^\e})\Big].
\eea

Define $\Pi^\e\in \CP_{\text{leb}}(\HH_T)$ by \eqref{Pie}, where $\theta^\e$ is introduced by \eqref{the}. From \eqref{u-phi},
\beq{LT1-ue}\barray
L_{T,1}(u^\e)\ad =\half \int_0^T  \| u^\e(s)\|^2 ds  = \half \sum_{j=1}^{|\LL|} \int_0^t \|u_j^{*}(s)\|^2 \indi_{\{\Ye(s-)=i\}}ds \leq I(\eta)+1,
\earray\eeq
and 
\beq{LT2-phie}\barray
L_{T,2}(\phi^\e) \ad  = \half \sum_{(i,j)\in\TT}\int_{[0,T]\times [0,\zeta]} \ell(\phi_{ij}^\e(s,z))\la_\zeta(dz)ds \\
\aad \leq \kappa_3 \half \sum_{(i,j)\in \TT} \int_{[0,T]\times [0,\zeta]} |\phi_{ij}^\e(s,z)-1|^2 \la_\zeta(dz)ds   \\
\aad \leq  \kappa_3 \e h^2(\e) \half \sum_{(i,j)\in \TT}  \int_{[0,T]\times [0,\zeta]} |\psi_{ij}^{*,\e}(s,z)|^2  \indi_{\{\Ye(s-)=i\}} \la_\zeta(dz)ds  \\
\aad \leq 2\kappa_3(I(\eta)+1) \e h^2(\e).
\earray\eeq

Take $M:=\max\{I(\eta)+1, 2\kappa_3 (I(\eta)+1)\}$. It then follows from \propref{prop:tight} that $(\eta^\e, \Pi^\e)$ is a tight family of $C([0,T];\rr^d)\times \CM_F(\HH_T)$-valued random variables. Suppose $(\bar \eta, \bar \Pi)$ is a weak limit point of $\{\eta^\e, \Pi^\e\}$, equations \eqref{eta-pi} and \eqref{fast-pi} hold  with $(\eta,\Pi)$ replaced with $(\bar\eta, \bar \Pi)$. Disintegrate $\bar\Pi$ as 
\beq{dis-Pi-bar}
\bar\Pi(ds\times \{y\}\times d\theta\times dz)=ds \bar\pi_y(s)[\bar\Pi]_{34|12}(d\theta\times dz).
\eeq
In the following, we will show that
\beq{eta-equal}
(\bar\eta(s),\bar\pi(s))=(\eta^*(s), \pi^*(s)) \quad \text{ for a.e. }  s \in [0,T].
\eeq
Without loss of generality, we can assume the convergence of $(\eta^\e,\Pi^\e)$ to $(\bar\eta, \bar \Pi)$ holds along the full sequence. For every $y\in \LL, t\in [0,T]$, and any continuous map $h:[0,T]\to \rr^d$, we first demonstrate that
\beq{h-Pi-bar}
\int_{\HH_t} h(s)^\top \indi_{\{j\}}(y)z \bar\Pi(d\bv)= \int_0^t h(s)^\top u_j^*(s) \bar\pi_j(s)ds.
\eeq

From \eqref{LT1-ue} and \eqref{LT2-phie} and the argument of \propref{prop:tight}, we have for all $t\in [0,T]$ and $j\in \LL$,
\beq{cov-h-Pi}
\int_{\HH_t} h(s)^\top \indi_{\{j\}}(y)z \Pi^\e(d\bv)\to \int_{\HH_t} h(s)^\top \indi_{\{j\}}(y)z \Pi(d\bv).
\eeq
We note that by the definition of $\Pi^\e$, 
\bea\ad \!\!\!\!\!\!\! 
\int_{\HH_t} h(s)^\top \indi_{\{j\}}(y)z \Pi^\e(d\bv) \\
\aad \!\!\!\!\!\!\! = \int_{[0,t]} h(s)^\top \indi_{\{\eta^{\e,u^*,\phi^*}(s-)=j\}} u_j^*(s)ds \\
\aad\!\!\!\!\!\!\! = \int_{[0,t]\times \LL} h(s)^\top u_j^*(s) \indi_{\{j\}}(y) [\Pi^\e]_{12}(dy\times ds) \\
\aad\!\!\!\!\!\!\! \to \int_{[0,t]\times \LL}\!\! h(s)^\top u_j^*(s) \indi_{\{j\}}(y) [\bar \Pi]_{12}(dy\times ds)=\! \int_{[0,t]}\!\! h(s)^\top u_j^*(s)\bar\pi_j(s)ds
\quad \text{a.s. as } \e \to 0,
\eea
where the almost sure convergence is following
Lemma 5.2 and the square integrability of $u_j^*$. Thus we obtain \eqref{h-Pi-bar}. Applying \eqref{h-Pi-bar} to $\sg_j(\bar X(s))$ and summing over $j$, we have the following:
\bea\ad 
\sum_{j\in \LL} \int_0^t \sg_j(\bar X(s)) u_j^*(s)\bar\pi_j(s)ds=\int_{\HH_t} \sg(\bar X(s),y)z\bar\Pi(d\bv).
\eea
We also note that as $\e\to 0$, 
\beq{Pie-phi}\barray\ad 
\int_{\HH_t} \sum_{j \in \LL} [\Phi(\bar X(s),j)-\Phi(\bar X(s),y)] \Ga_{yj}^\theta(\bar X(s)) \Pi^\e(d\bv) \\
\aad \to \int_{\HH_t} \sum_{j \in \LL} [\Phi(\bar X(s),j)-\Phi(\bar X(s),y)] \Ga_{yj}^\theta(\bar X(s)) \Pi(d\bv).
\earray\eeq
By the definition of $\Pi^\e$, the following holds:
\beq{phi-Pie-exp}\barray\ad \!\!\!\! 
\int_{\HH_t} \sum_{j \in \LL} [\Phi(\bar X(s),j)-\Phi(\bar X(s),y)] \Ga_{yj}^\theta(\bar X(s)) \Pi^\e(d\bv)\\
\aad\!\!\!\! = \int_0^t \sum_{i\in\LL} \sum_{j\in \LL}[\Phi(\bar X(s),j)-\bar\Phi(\bar X(s),i)] \indi_{\{\Ye(s-)=i\}}\\
\aad\!\!\!\!\qquad\qquad\qquad \times  \int_{E_{ij}(\bar X(s))} \psi_{ij}^{\e,*}(s,z)\indi_{\{|\psi_{ij}^{\e,*}(s,z)|\leq \beta/\sqe h(\e)\}} \la_\zeta(dz)ds \\
\aad\!\!\!\! = \int_{[0,t]\times \LL} \sum_{j\in \LL}[\Phi(\bar X(s),j)-\Phi(\bar X(s),i)] \indi_{\{i\}}(y) \\
\aad\!\!\!\!\qquad\qquad\qquad\times \int_{E_{ij}(\bar X(s))} \psi_{ij}^{\e,*}(s,z)\indi_{|\psi_{ij}^{\e,*}(s,z)|\geq \beta/(\sqe h(\e))} \la_\zeta(dz)[\Pi^\e]_{12}(dy\times ds)\\
\aad\!\!\!\! \to\int_{[0,t]\times \LL} \sum_{j\in \LL}[\Phi(\bar X(s),j)-\Phi(\bar X(s),i)] \indi_{\{i\}}(y) \int_{E_{ij}(\bar X(s))}\!\! \psi_{ij}^*(s,z)\la_\zeta(dz) [\bar\Pi](dy\times ds) \\
\aad\!\!\!\! = \int_0^t \sum_{j\in \LL}[\Phi(\bar X(s),j)-\Phi(\bar X(s),i)] \bar\pi_i(s) \int_{E_{ij}(\bar X(s))} \psi_{ij}^*(s,z) \la_\zeta(dz)ds.
\earray\eeq
Combining \eqref{Pie-phi} and \eqref{phi-Pie-exp} gives 
\bea\ad 
\int_{\HH_t} \sum_{j\in \LL}[\Phi(\bar X(s),j)-\Phi(\bar X(s),y)] \Ga_{yj}^\theta(\bar X(s)) \Pi(d\bv) \\
\aad = \int_0^t \sum_{j\in \LL}[\Phi(\bar X(s),j)-\Phi(\bar X(s),i)]\bar\pi_i(s) \int_{E_{ij}(\bar X(s))} \psi_{ij}^*(s,z)\la_\zeta(dz)ds.
\eea

Thus we have shown \eqref{eta-def} is satisfied with $(\eta,\pi,u,\psi)$ replaced by $(\bar\eta, \bar\pi, u^*, \psi^*)$. We next show that \eqref{fast-cond} holds as well, i.e., $ 
\sum_{i\in \LL} \bar\pi_i(s) \Ga_{ij}(\bar X(s))=0, \text{ for a.e. } s\in [0,T] \text{ and } j\in \LL$. 
Since \eqref{fast-pi} hold with $\Pi$ replaced by $\bar \Pi$, we have 
\beq{fast-bpi}
\int_{\HH_t} \Ga_{yj}(\bar X(s))\bar \Pi(d\bv)=0, \quad \text{ for a.e. } s\in  [0,T], \; j\in \LL
\eeq
As $\e\to 0$, we note that
\bea\ad 
\int_{\HH_t} \Ga_{yj}(\bar X(s))\Pi^\e(d \bv)=\sum_{i\in \LL} \int_0^t \indi_{\{\Ye(s)=i\}} \int_{E_{ij}(\bar X(s))} \la_\zeta(dz) ds \\
\aad =\int_{[0,t]\times \LL} \int_{E_{ij}(\bar X(s))} \la_\zeta(dz)[\Pi^\e]_{12}(dy\times ds) \\
\aad \to \int_{[0,t]\times  \LL}  \int_{E_{ij}(\bar X(s))} \la_\zeta(dz)[\bar\Pi]_{12}(dy\times ds)\\
\aad= \int_0^t \sum_{i\in \LL}\pi_i(s) \int_{E_{ij}(\bar X(s))} \la_\zeta(dz)ds=\int_0^t \sum_{i\in \LL}\bar\pi_i(s)\Ga_{ij}(\bar X(s))ds.
\eea

From \eqref{fast-bpi}, we have for a.e. $s\in [0,T]$ and $j\in \LL$, $
\sum_{i\in \LL} \bar\pi_i(s) \Ga_{ij}(\bar X(s))=0$. 
Consequently, we have shown that \eqref{eta-def} and \eqref{fast-cond} hold with $(\eta,\pi,u,\psi)$ replaced  by $(\bar\eta,\bar\pi,u^*,\psi^*)$. Then \propref{prop:low} implies \eqref{eta-equal}. Finally, let us study the convergence of costs. The \eqref{u-phi} implies that 
\bea
L_{T,1}(u^\e)\ad 
=\half\int_0^T \|u^\e(s)\|^2 ds =\half \sum_{j\in \LL} \int_0^T \indi_{\{\Ye(s-)=j\}} \|u_j^*(s)\|^2 ds \\
\aad = \half \sum_{j\in \LL} \int_{[0,T]\times \LL} \|u_{j}^*(s)\|^2 \indi_{\{j\}}(y)[\Pi^\e]_{12}(dy\times ds) \\
\aad \rightarrow \half \sum_{j\in \LL} \int_{[0,T]\times \LL} \|u_j^*(s)\|^2 \indi_{\{j\}}(y) [\bar \Pi]_{12}(dy\times ds)\\
\aad = \half \sum_{j\in \LL} \int_0^T \|u_j^*(s)\|^2 \bar \pi_j(s)ds=\half \sum_{j\in \LL} \int_0^T \|u_j^*(s)\|^2 \pi_j^*(s)ds,
\eea
where the last equality is due to \eqref{eta-equal}. For the second term in the cost, the \lemref{lem:ineq} (iv) and \eqref{psi-star} yields 
\bea\ad 
\frac{1}{\e h^2(\e)}L_{T,2}(\phi^\e) \\
\ad =\frac{1}{\e h^2(\e)} \sum_{(i,j)\in \TT} \int_{[0,T]\times [0,\zeta]} \ell(\phi_{ij}^\e(s,z))\la_\zeta(dz)ds \\
\aad = \frac{1}{\e h^2(\e)}\sum_{(i,j)\in \TT} \int_{[0,T]\times [0,\zeta]} \indi_{\{\Ye(s-)=i\}} \ell(\phi_{ij}^{*,\e}(s,z))\la_\zeta(dz)ds \\
\aad \leq  \half \sum_{(i,j)\in \TT}\int_{[0,T]\times [0,\zeta]} \indi_{\{\Ye(s-)=i\}} |\psi_{ij}^{*,\e}(s,z)|^2 \la_\zeta(dz)ds \\
\aad \quad + \kappa_3\sqe h(\e) \sum_{(i,j)\in \TT}\int_{[0,T]\times [0,\zeta]} \indi_{\{\Ye(s-)=i\}} |\psi_{ij}^{*,\e}(s,z)|^3 \la_\zeta(dz)ds \\
\aad \leq \half (1+2\kappa_3 \beta) \sum_{(i,j)\in \TT} \int_{[0,T]\times [0,\zeta]} \indi_{\{\Ye(s-)=i\}} |\psi_{ij}^{*,\e}(s,z)|^2 \la_\zeta(dz)ds \\
\aad \leq \half (1+2\kappa_3 \beta) \sum_{(i,j)\in \TT} \int_{[0,T]\times [0,\zeta]} \indi_{\{\Ye(s-)=i\}} |\psi_{ij}^{*}(s,z)|^2  \la_\zeta(dz)ds
\eea
Furthermore for $j\in \LL$, we have 
\bea\ad\!\!\!\!\!\!\! 
\sum_{i\in \LL} \int_{[0,T]\times [0,\zeta]} \indi_{\{\Ye(s-)=i\}} |\psi_{ij}^*(s,z)|^2 \la_\zeta(dz)ds\\
\aad\!\!\!\!\!\!\!\! = \int_{[0,T]\times \LL} \Big(\int_{[0,\zeta]} |\psi_{ij}^*(s,z)|^2  \la_\zeta(dz)\Big) [\Pi^\e]_{12}(dy\times ds)\\
\aad\!\!\!\!\!\!\!\! \rightarrow \int_{[0,T]\times \LL} \Big(\int_{[0,\zeta]} |\psi_{ij}^*(s,z)|^2  \la_\zeta(dz)\Big) [\bar\Pi]_{12}(dy\times ds) \\
\aad\!\!\!\!\!\!\!\! = \sum_{i\in \LL} \int_{[0,T]\times [0,\zeta]} |\psi_{ij}^*(s,z)|^2 \bar \pi_i(s) \la_\zeta(dz)ds = \sum_{i\in \LL} \int_{[0,T]\times [0,\zeta]} |\psi_{ij}^*(s,z)|^2 \pi_i^*(s) \la_\zeta(dz)ds.
\eea
Thus, as $\e\to 0$, we obtain the following
\bea\ad\!\! 
L_{T,1}(u^\e)+\frac{1}{\e h^2(\e)}L_{T,2}(\phi^\e)  \\
\aad\!\!\!\!\! \to \half \sum_{j\in \LL} \int_0^T \|u_j^*(s)\|^2 \pi_j^*(s)ds+\half (1+2\kappa_3 \beta)\! \sum_{(i,j)\in \TT} \int_0^T\!\!\! \int_0^\zeta |\psi_{ij}^*(s,z)|^2 \pi_i^*(s) \la_\zeta(dz)ds,
\eea
in probability. Therefore, employing the dominated convergence theorem as $\e\to 0$ and then taking $\beta\to 0$ lead to 
\bea\ad 
\EE\Big[L_{T,1}(u^\e)+\frac{1}{\e h^2(\e)}L_{T,2}(\phi^\e)\Big]\\
\aad \to \half \sum_{j\in \LL} \int_0^T \|u_j^*(s)\|^2 \pi_j^*(s)ds +\half \sum_{(i,j)\in \TT} \int_{[0,T] \times [0,\zeta]} |\psi_{ij}^*(s,z)|^2 \pi_i^*(s) \la_\zeta(dz)ds \\
\aad\quad \leq I(\eta)+\frac{\ga}{2}.
\eea
The above inequality and the convergence of $\etae$ to $\eta^*$ gives
\bea \ad 
\limsup_{\e\to 0} - \frac{1}{h^2(\e)} \log\EE\Big[\exp\big( -h^{-2}(\e) F(\eta^\e) \big)\Big]\\ 
\ad  \leq I(\eta)+F(\eta^*)+\frac{\ga}{2} \leq I(\eta)+F(\eta) + \frac{\ga}{2}+K_F \ga \\
\aad \leq \inf_{\eta\in C([0,T];\rr^d)}[I(\eta)+F(\eta)] + \ga + K_F\ga.
\eea
Letting $\ga\to 0$, we have the desired upper bound \eqref{ld-lower}. 
\end{proof}

\section{Extension to unbounded drift}\label{sec:ext}
In this section, we prove \corref{cor:extend}. Under condition (H1)$'$, (H2) and (H3), the authors in \cite{BDG18} proved that $\{X^\e\}_{\e>0}$ satisfies a LDP with speed $\e^{-1}$ and the rate function $J$ defined in \eqref{I-LDP} in the space $C([0,T];\rr^d)$. It then implies exponential tightness of $X^\e$:
\bea\ad 
\lim_{A\to \infty} \limsup_{\e \to 0} \e \log\PP\bigg(\sup_{t\in [0,T]}\|X^\e\| > A\bigg) = -\infty.
\eea

\begin{lem}\label{lem:Xe-A1}
Suppose (H1)$'$, (H2) and (H3) hold, there exits a constant $A_1>0$ such that
\beq{Xe-inf}
\limsup_{\e\to 0} \frac{1}{h^2(\e)}\log\PP\bigg(\sup_{t\in [0,T]}\|X^\e(t)\| > A_1 \bigg) =-\infty.
\eeq
\end{lem}
\begin{proof}
For any $M>0$, the compactness of the level set $L=\{\phi: J(\phi)\leq 2M\}$ implies that there exists a constant $A_1>0$ such that for every $\phi \in L$, $\|\phi\|_T \leq A_1$ (Recall that $\|\phi\|_T=\sup_{t\in [0,T]}\|\phi(t)\|$).  Thus for every $\phi$ with $\|\phi\|_T \geq A_1$, we have $J(\phi)>2M$. Define the set $\CO_{A_1}:=\{\phi:\|\phi\|_T \geq A_1\}$, it gives $
\inf_{\phi\in \CO_{A_1}} J(\phi) \geq 2M$. 
Applying the LDP upper bound for $X^\e$ to the closed set $\CO_{A_1}$ implies
\bea\ad 
\limsup_{\e \to 0} \e \log \PP\bigg(\sup_{t\in[0,T]}\|X^\e\| \geq A_1\bigg) \\
\aad =\limsup_{\e \to 0} \e \log\PP(X^\e \in \CO_{A_1}) \leq - \inf_{\phi\in \CO_{A_1}} J(\phi) \leq -2M.
\eea
Hence there exists $\e_3>0$ such that for all $\e \in (0,\e_3)$, 
\beq{ineq-A1}
\e\log\PP\Big(\sup_{t\in [0,T]}\|X^\e(t)\| \geq A_1\Big) \leq -M
\eeq
Dividing $\e h^{2}(\e)$ to both sides of \eqref{ineq-A1} and taking $\e\to 0$, the proof is complete.
\end{proof}

For any $B > A_1$ where $A_1$ is the constant in \lemref{lem:Xe-A1} such that \eqref{Xe-inf} holds. Let $\CS_B=\{x: \|x\| \leq B\}$ be the sphere with center at the origin and radius $B$. We define the truncated version of the drift function $b$ as $b^B$:
\beq{bB} b^B(x,i):= b(x,i)\varpi^B(x), \quad \forall\; i \in \LL,
\eeq
where $\varpi^B\cd$ being a smooth function satisfying 
\bea\varpi^B(x)=\left\{\barray
1 \quad \text{ if } x\in \CS^B,\\
0 \quad \text{ if } x\in \rr^d\setminus \CS_{B+1}.
\earray\right.\eea
Denote by $(X^{\e,B}\cd, Y^{\e,B}\cd)$ the solution of \eqref{sde} with $b$ replaced by $b^B$, and $\bar X^B\cd$ the solution of ODE \eqref{bar-X} with $\bar b$ replaced by $\bar b^B$, where $\bar b^B: =\sum_{i=1}^\LL b^B(x,i)\mu_i^x$. Because of the continuity of $b$ and the finiteness of $x_0$, we note that $\bar X$ is bounded on $[0,T]$. Choose $B\geq B^*:=\|\bar X\|_T \vee A_1$, we then have $\varpi^B(\bar X(t))=1, \forall\, t\in [0,T]$. It yields $\bar b(\bar X(t))=\bar b^B(\bar X(t))$. By uniqueness of the ODE \eqref{bar-X} with $\bar b$ replaced by $\bar b^B$, we have $\bar X(t) =\bar X^B(t)$ for all $t\in [0,T]$.

Define $\eta^{\e,B}$ as
\bea\ad 
\eta^{\e,B}(t):=\frac{X^{\e,B}(t)-\bar X^B(t)}{\sqe h(\e)}= \frac{X^{\e,B}(t)-\bar X(t)}{\sqe h(\e)}\quad t\in [0,T].
\eea
\thmref{thm:main} implies that $\eta^{\e,B}$ satisfies a LDP with speed $h^{-2}(\e)$ and rate function $I_B$ defined in \eqref{rate} with the first constraint \eqref{eta-def} becomes 
\beq{eta-def-B}\barray
\eta(t)\ad = \int_0^t \nabla \bar b^B(\bar X(s)) \eta(s)ds + \sum_{j\in \LL} \int_0^t \sg_j(\bar X(s)) u_j(s) \pi_j(s)ds \\
\aad +\sum_{(i,j)\in \TT} \int_{[0,t]\times [0,\zeta]} [\Phi_B(\bar X(s),j)-\Phi_B(\bar X(s), i)] \pi_i(s) \indi_{E_{ij}(\bar X(s))}(z) \psi_{ij}(s,z) \la_\zeta(dz)ds,
\earray\eeq
where $\Phi_B$ is the solution of the Poisson equation \eqref{pq} with $b(x,i)$ and $\bar b$ replaced by $b^B(x,i)$ and $\bar b^B$, respectively. We note that second constraint \eqref{fast-cond} remains the same as $\bar X(t)=\bar X^B(t), \forall\, t \in [0,T]$ for our choice of $B$.

We then claim $\eta^{\e,B}$ is an exponential good approximation of $\eta^{\e}$ with respect to the speed $h^{-2}(\e)$; see \cite[Section 4.22]{DZ09}. That is, for any $a>0$,
\bea\ad 
\lim_{B\to \infty} \limsup_{\e \to 0} \frac{1}{h^2(\e)}\log\PP\bigg(\sup_{t\in [0,T]}\|\eta^\e(t) -\eta^{\e,B}(t)\|>a \bigg)=-\infty.
\eea
Indeed, for all $B\geq B^*$, we have 
\bea\ad
\limsup_{\e \to 0} \frac{1}{h^2(\e)}\log\PP\bigg(\sup_{t\in [0,T]} \|\eta^\e(t)-\eta^{\e,B}(t)\| > a \bigg) \\
\aad = \limsup_{\e\to 0}\frac{1}{h^2(\e)} \log\PP\bigg(\sup_{t\in [0,T]} \|X^\e(t)- X^{\e,B}(t)\|> a \sqe h(\e)\bigg) \\
\aad \leq \limsup_{\e \to 0}\frac{1}{h^2(\e)} \log\PP\bigg(\sup_{t\in [0,T]}\|X^\e(t)-X^{\e,B}(t)\|>0\bigg) \\
\aad \leq \limsup_{\e \to 0}\frac{1}{h^2(\e)} \log\PP\bigg(\sup_{t\in [0,T]} \|X^\e(t)-X^{\e,B}(t)\|>0; \sup_{t\in [0,T]}\|X^\e(t)\| \leq B \bigg) \\
\aad\quad \vee \limsup_{\e\to 0} \frac{1}{h^2(\e)}\log\PP\bigg(\sup_{t\in [0,T]} \|X^\e(t)-X^{\e,B}(t)\|>0 ; \sup_{t\in [0,T]} \|X^\e(t)\|>B\bigg) \\
\aad \leq \limsup_{\e\to 0}\frac{1}{h^2(\e)}\log\PP\bigg(\sup_{t\in [0,T]}\|X^\e(t)\| > B \bigg)\\
\aad \leq \limsup_{\e \to 0}\frac{1}{h^2(\e)} \log\PP\bigg(\sup_{t\in [0,T]}\|X^\e(t)\| > A_1 \bigg)=-\infty,
\eea
where the second to last inequality follows from the fact that under the condition $\sup_{t\in [0,T]} \|X^\e(t)\|\leq B, X^\e(t)=X^{\e,B}(t), \forall\, t\in [0,T]$ which follows by the uniqueness of SDE \eqref{sde} with $b$ replaced by $b^B$. It is similar to the argument how we obtain $\bar X(t)=\bar X^B(t)$. The last inequality follows from \eqref{Xe-inf}. Thus, $\eta^{\e,B}$ is an exponential good approximation of $\eta^{\e}$. By \cite[Theorem 4.2.16]{DZ09}, it then implies $\eta^\e$ satisfies a weak LDP with speed $h^{-2}(\e)$ and rate function $\wdt I$ given by
\beq{I-unb}
\wdt I(\eta):=\sup_{\dl>0}\liminf_{B\to \infty} \inf_{g\in B_{\eta,\dl}} I_B(g),
\eeq
where $B_{\eta,\dl}$ denotes the ball $\{g\in C([0,T];\rr^d): \sup_{t\in [0,T]} \|\eta(t)-g(t)\|\leq \dl \}$ for a given $\eta$. By \propref{prop:hat-eta}, $\wdt I$ can also be defined by 
\beq{I-unb-wdh}
\wdt I(\eta):=\sup_{\dl>0} \liminf_{B\to \infty} \inf_{g\in B_{\eta,\dl}} \wdh I_B(g),
\eeq
where $\wdh I_B(\eta)$ is defined as \eqref{wdh-I} with the dynamics \eqref{eta-pi} becomes
\beq{eta-pi-B}\barray
\eta(t)\ad = \int_{\HH_t} \nabla \bar b^B(\bar X(s)) \eta(s) \Pi(d\bv)+\int_{\HH_t} \sg(
\bar X(s),y)z \Pi(d\bv) \\
\aad \quad + \int_{\HH_t}\sum_{j\in \LL} [\Phi_B(\bar X(s),j)-\Phi_B (\bar X(s),y)]\Ga_{yj}^\theta(\bar X(s))
\Pi(d\bv),
\earray\eeq
and the constraint \eqref{fast-pi} remains the same.

In what follows, we  characterize the rate function $\wdt I$ in \eqref{I-unb-wdh} as $\wdh I$ defined in \eqref{rate} in \eqref{wdh-I} and prove the full LDP in \lemref{lem:full-LDP}. Before that, we need the stability result of the solution of  the Poisson equation \eqref{pq}.

\begin{lem}\label{lem:Phi-B-Phi}
Under conditions (H1)$'$, (H2)-(H3), for every compact set $\CK \subset \rr^d$, we have
\bea\ad 
\sup_{x\in \CK, i\in \LL} |\Phi_B(x,i)-\Phi(x,i)| \to 0, \quad \text{ as } B\to \infty.
\eea
\end{lem}
\begin{proof}
From \thmref{thm:Phi}, (H2) and the definition of $\bar b,\bar b^B$ imply that 
\bea
|\Phi_B(x,i)-\Phi(x,i) | \ad \leq \int_0^\infty |\EE(\wdt b^B -\wdt b)(x, Y_t^{x,i})|dt \\
\ad = \int_0^\infty \big|P_t^{x}[\wdt b^B- \wdt b](x,\cdot)(i)- \mu^x ([\wdt b^B-\wdt b](x,\cdot))\big|dt \\
\ad \leq \int_0^\infty \|[\wdt b^B-\wdt b](x,\cdot)\|_\infty \|P_t^x(i,\cdot)- \mu^x\|_\text{var} dt  \\
\ad \leq K \|[\wdt b^B -\wdt b](x,\cdot)\|_\infty \int_0^\infty e^{-\varrho t} dt \\
\ad \leq \frac{K}{\varrho}  \|\wdt b^B(x,\cdot)- \wdt b(x,\cdot)\|_\infty \\
\aad \leq \frac{K}{\varrho} \Big(\|b^B(x,\cdot)-b(x,\cdot)\|_\infty +\|\bar b^B(x)- \bar b(x) \| \Big),
\eea
where $\wdt b^B:= b^B -\bar b^B$. Thus the definition of $b^B$ and $\bar b^B$ completes the proof.
\end{proof}

\begin{lem}\label{lem:wdt-I}
For the rate function $\wdt I $ defined \eqref{I-unb-wdh}, it takes the form of \eqref{wdh-I}. It follows that the rate function $\wdt I$ defined in \eqref{I-unb} takes the form of \eqref{rate}.
\end{lem}
\begin{proof}
First, we prove $\wdt I(\eta) \leq \wdh I(\eta)$. Assume $\wdh I(\eta)<\infty$, otherwise it is trivial. Fix $\e>0$. By definition of $\wdh I(\eta)$ in \eqref{wdh-I}, there exists $\Pi\in \CP_s(\eta)$ such that \eqref{eta-pi} and \eqref{fast-pi} hold and 
\bea
C(\Pi) := \int_{\HH_T} \half (\|z\|^2+ \wdh d(\theta)) \Pi(d\bv) \leq \wdh I(\eta)+\e
\eea
Since $\bar X$ is continuous on $[0,T]$, its range is in a compact set. We choose $B$ large enough such that $b^B=b$ on a fixed compact set $\CK$ containing the range of $\bar X$. It then implies $\Phi_B=\Phi$ on this compact set. Thus, $\Pi \in \CP_s(\eta)$ where $\eta$ now satisfies the equation \eqref{eta-pi-B} and the cost does not change. Therefore, $
\wdh I_B(\eta) \leq C(\Pi) \leq \wdh I(\eta)+\e$. 
For every $\dl>0$ and $B$ large enough, we have 
\bea\ad
\inf_{g\in B_{\eta,\dl}} \wdh I_B(g) \leq \wdh I_B(\eta)\leq \wdh I(\eta)+\e.
\eea
Taking $\liminf_{B\to \infty}$ and supermum over $\dl>0$ to above inequality, it gives $\wdt I(\eta) \leq \wdh I(\eta) +\e$ by the definition \eqref{I-unb-wdh}. Since $\e>0$ is arbitrary, $\wdt I(\eta) \leq \wdh I(\eta)$.

We now prove the reversed inequality $\wdh I(\eta)\leq \wdt I(\eta)$. Assume $\wdt I(\eta)<\infty$. By definition of $\wdt I$ in \eqref{I-unb-wdh}, there exists sequences $B_n \to \infty, \dl_n \downarrow 0$ and $g_n \in B_{\eta,\dl_n}$ such that
\beq{I-Bn}
\lim_{n \to \infty} \wdh I_{B_n}(g_n) = \wdt I(\eta).
\eeq
For each $n\in \NN$, the definition of $\wdh I_{B_n}(g_n)$ implies there exists a $\Pi_n\in \CP_s(g_n)$ such that 
\beq{def-CPi-n}
C(\Pi_n):= \int_{\HH_T} \half (\|z\|^2 + \wdh d(\theta)) \Pi_n(d\bv) \leq \wdh I_{B_n}(g_n)+1/n,
\eeq
and $g_n$ satisfies the equation \eqref{eta-pi-B} with $\Pi, B$ replaced by $\Pi_n, B_n$, respectively. That is,
\beq{gn}\barray
g_n(t)\ad = \int_{\HH_t} \nabla \bar b^{B_n}(\bar X(s)) g_n(s) \Pi_n (d\bv)+\int_{\HH_t} \sg(
\bar X(s),y)z \Pi_n(d\bv) \\
\aad \quad + \int_{\HH_t}\sum_{j\in \LL} [\Phi_{B_n}(\bar X(s),j)-\Phi_{B_n}(\bar X(s),y)]\Ga_{yj}^\theta(\bar X(s))
\Pi_n(d\bv).
\earray\eeq
For sufficient large $n$, the equation \eqref{I-Bn} and \eqref{def-CPi-n} imply
\beq{C-Pi-n}
C(\Pi_n) \leq \wdh I_{B_n}(z_n)+1/n \leq \wdt I(\eta)+2/n \leq \wdt I(\eta)+2 <\infty.
\eeq
Thus $C(\Pi_n)$ admits a uniform bound independent of $n$. It follows that $\{\Pi_n\}$ is pre-compact in $\CP_{\text{leb}}(\HH_T)$. Then there exists a convergent subsequence (still denoted by $n$) such that $\Pi_n \to \bar \Pi$ weakly for some $\bar \Pi$. By the lower semicontinuity of the cost functional, \eqref{I-Bn} and \eqref{def-CPi-n} yield
\bea\ad
C(\bar \Pi) \leq \liminf_{n\to \infty} C(\Pi_n) \leq  \liminf_{n\to \infty} \Big(\wdh I_{B_n}(g_n) + 1/n\Big)=\wdt I(\eta).
\eea
If we proved $\bar\Pi\in \CP_s(\eta)$, we can get $\wdh I(\eta)=\inf_{\bar \Pi\in \CP_s(\eta)} C(\bar\Pi) \leq \wdt I(\eta)$, thus finish the proof.

Let us proceed to prove $\bar \Pi \in \CP_s(\eta)$.
Because $g^n$ satisfies the equation \eqref{gn} and $g^n$ converges to $\eta$ uniformly on $[0,T]$ as $g_n\in B_{\eta,\dl_n}$ and $\dl_n \downarrow 0$ when $n\to \infty$, it is sufficient to prove $\eta$ satisfies the un-truncated equation \eqref{eta-pi} with $\Pi$ replaced by $\bar\Pi$. From \eqref{gn}, we first note that
\bea\ad 
\int_{\HH_t} \nabla \bar b^{B_n}(\bar X(s)) g_n(s) \Pi_n(d\bv)- \int_{\HH_t} \nabla \bar b(\bar X(s)) \eta(s) \bar \Pi(d\bv) \\
\aad =\int_{\HH_t} \big[ \nabla \bar b^{B_n}(\bar X(s))- \nabla \bar b(\bar X(s))\big] g_n(s) \Pi_n (d\bv) \\
\aad \quad + \int_{\HH_t} \nabla\bar b(\bar X(s)) \big[g_n(s) -\eta(s) \big] \Pi_n(d\bv) + \int_{\HH_t} \nabla \bar b(\bar X(s)) \eta(s) \big[\Pi_n (d\bv ) - \bar \Pi(d\bv) \big] \\
\aad =: \CR_1^n+ \CR_2^n + \CR_3^n.
\eea
For $\CR_1^n$, for sufficiently large $n$ (thus $B_n$), we have $\bar b^{B_n}(\bar X(s)) =\bar b(\bar X(s))$,  which implies $\nabla \bar b^{B_n}(\bar X(s))=\nabla \bar b(\bar X(s))$. Hence $g_n\in B_{\eta,\dl_n}$ with given $\eta\in C([0,T];\rr^d)$ implies $\CR_1^n \to 0$ as $n\to \infty$. The uniform convergence of $g_n$ to $\eta$ and the boundedness of $\nabla \bar b(\bar X(s))$ yield that $\CR_2^n \to 0$ as $n\to \infty$. The $\CR_3^n \to 0$ followed by the weak convergence of $\Pi_n$ to $\bar\Pi$.

Moreover, the uniform integrability argument in \eqref{uni-arg} implies
\bea\ad 
\int_{\HH_t} \sg(\bar X(s),y) z \Pi_n(d\bv) \to \int_{\HH_t} \sg(\bar X(s),y)z \bar \Pi(d\bv)
\eea
Furthermore, we note that 
\bea\ad 
\int_{\HH_t}\sum_{j\in \LL} [\Phi_{B_n}(\bar X(s),j)-\Phi_{B_n} (\bar X(s),y)]\Ga_{yj}^\theta(\bar X(s))
\Pi_n(d\bv)\\
\aad \quad  - \int_{\HH_t}\sum_{j\in \LL} [\Phi(\bar X(s),j)-\Phi (\bar X(s),y)]\Ga_{yj}^\theta(\bar X(s)) \bar \Pi(d\bv) \\
\aad = \int_{\HH_t}\sum_{j\in \LL} \big[(\Phi_{B_n}-\Phi)(\bar X(s),j)-(\Phi_{B_n}-\Phi) (\bar X(s),y) \big] \Ga_{yj}^\theta(\bar X(s))
\Pi_n(d\bv) \\
\aad \quad + \int_{\HH_t}\sum_{j\in \LL} [\Phi(\bar X(s),j)-\Phi (\bar X(s),y)] \Ga_{yj}^\theta(\bar X(s))
\big[\Pi_n(d\bv) -\bar \Pi(d\bv)\big]=: \CR_4^n+ \CR_5^n.
\eea
For $\CR_4^n$, \lemref{lem:Phi-B-Phi} and the H\"{o}lder inequality yield
\bea
\CR_4^n \ad \leq 2\sup_{s\in [0,T],j\in \LL} \|\Phi_{B_n}(\bar X(s),j)-\Phi(\bar X(s),j)\| \sum_{j\in \LL} \int_{\HH_t}  \Ga_{yj}^\theta(\bar X(s))\Pi_n(d\bv) \\
\aad \leq 2\sup_{s\in [0,T],j\in \LL} \|\Phi_{B_n}(\bar X(s),j)-\Phi(\bar X(s),j)\| \\
\aad \qquad\quad  \times \sum_{j\in \LL} \Big(\int_{\HH_t} 1^2 \Pi_n(d\bv)\Big)^{1/2} \Big(\int_{\HH_t} |\Ga_{yj}^\theta(\bar X(s))|^2 \Pi_n(d\bv)\Big)^{1/2} \\
\aad \leq 2\sup_{s \in [0,T], j\in \LL}  \|\Phi_{B_n}(\bar X(s),j)-\Phi(\bar X(s),j)\| \\
\aad \qquad \quad \times \Big(\sum_{j\in \LL}\int_{\HH_T} \Pi_n(d\bv)\Big)^{1/2} \Big(\sum_{j\in \LL} \int_{\HH_T} |\Ga_{yj}^\theta(\bar X(s))|^2 \Pi_n(d\bv)\Big)^{1/2} \\
\aad \leq 2\sqrt{T\zeta (\wdt I(\eta)+2)} \sup_{s \in [0,T], j\in \LL}  \|\Phi_{B_n}(\bar X(s),j)-\Phi(\bar X(s),j)\| \to 0 \quad \text{ as } n \to \infty,
\eea
where the last inequality follows from the argument in \eqref{b-th-pi} and \eqref{C-Pi-n}. For $\CR_5^n$, it converges to zero by the weak convergence of $\Pi_n$ to $\bar\Pi$ as in \eqref{conv-Phi-Ga}. Consequently, $\eta$ satisfies the \eqref{eta-pi} with $\Pi$ replaced by $\bar \Pi$. The proof is complete.
\end{proof}

\begin{lem}\label{lem:full-LDP}
Under assumptions (H1)$'$, (H2) and (H3), we have 
\begin{itemize}
\item[(i)] $I\cd$ is a good rate function
\item[(ii)] For every closed set $F$, 
\bea \ad 
\inf_{\eta\in F} I(\eta) \leq \limsup_{B\to \infty} \inf_{\eta \in F} I_B(\eta).
\eea  
\end{itemize}
Therefore, $\eta^\e$ satisfies the full LDP and thus \corref{cor:extend} hold. \end{lem}
\begin{proof}
The proof of (i) follows the argument of \propref{prop:comp-le}. The different things are that under the Lipschitz condition of $b(\cdot,i), i\in \LL$, the corresponding solution of Poisson equation $\Phi(\cdot,i), i\in \LL$ has a linear growth; see \eqref{Phi-bdd}. But since $\Phi(\cdot,i), i\in \LL$ acts on the deterministic path $\bar X$, the boundedness of $\eta_n$ in \eqref{K-eta} still hold. The rest of the proof remains the same, thus details are omitted. 

We now proceed to prove (ii). Let $F\subset C([0,T];\rr^d)$ be closed. Put \beq{ell}
\varkappa:= \limsup_{B\to \infty} \inf_{\eta\in F} I_B(\eta).
\eeq
Assume $\varkappa <\infty$, otherwise it is trivial. By definition of $\limsup$, there is a sequence $B_n\to \infty$ such that
\bea\ad 
\lim_{n\to \infty} \inf_{\eta \in F} I_{B_n}(\eta)= \varkappa
\eea
For each $n\in \NN$, the definition of infimum implies that we can choose $\eta_n \in F$ such that $I_{B_n}(\eta_n) \leq \inf_{\eta\in F} I_{B_n}(\eta)+1/n$. Thus for sufficiently large $n$, $I_{B_n}(\eta_n)$ is bounded by $\varkappa+1$. The compactness of level set in \propref{prop:comp-le} implies that $\{\eta_n\}$ is  relatively compact. Thus there exists a convergent subsequence (still denoted by $n$) such that $\eta_n \to \bar \eta$ uniformly on $[0,T]$ for some $\bar \eta$. Since $\eta_n \in F$ and $F$ is a closed set, we have $\bar \eta\in F$. The argument in \lemref{lem:wdt-I} implies that 
\bea\ad 
I(\bar \eta) \leq \liminf_{n\to \infty} I_{B_n}(\eta_n) = \lim_{n\to \infty}\Big(\inf_{\eta\in F} I_{B_n}(\eta)+1/n\Big)=\varkappa.
\eea
Therefore, we have 
\bea\ad 
\inf_{\eta \in F} I(\eta) \leq I(\bar \eta) \leq \varkappa = \limsup_{B\to \infty} \inf_{\eta\in F} I_B(\eta).
\eea
Thus we finished the proof of (i) and (ii), yielding the full LDP by \cite[Theorem 4.2.16]{DZ09}.
\end{proof}

\section{Discussions and remarks}\label{sec:dis}
\subsection{Jump diffusions in fast-varying Markovian environment}
The result in this paper can be extended to jump diffusion in the fast-varying Markovian  environment. Specifically, we can consider
\beq{jump}\barray
dX^\e(t)\ad = b(X^\e(t), Y^\e(t))dt+ \sqrt{\e} \sg(X^\e(t), Y^\e(t))dW(t)\\
\aad\quad + \e \int_{[0,t]\times \Upsilon} \chi(X^\e(s-), Y^\e(s-), \xi) N^{1/\e}(ds, d\xi),
\earray\eeq
where $b: \rr^d\times \LL \to \rr^d$, $\sg: \rr^d\times\LL \to \rr^{d\times d}$, and $\chi:\rr^d \times \LL \times \Upsilon \to \rr^d$ are suitable functions, $W$ is a standard $d$-dimensional Wiener process, $N^{1/\e}$ is a PRM on $\Upsilon_T := [0,T]\times \Upsilon$ with intensity measure $\e^{-1} \nu_T := \e^{-1}\la_T \times \nu$, where $\Upsilon$ is s locally compact Polish space and $\nu$ is a locally finite measure on $\Upsilon$, $\la_T$ is the Lebesgue measure on $[0,T]$. The $Y^\e\cd$ is the fast-varying Markov chain on finite state space $\LL$ satisfying conditions (H2) and (H3). Besides the condition (H1)$'$, we also assume 
\smallskip

\noindent{(H4)} For each $i \in \LL$ and $\xi\in \Upsilon$, the function $\chi(\cdot,i,\xi)$ is a continuously differentiable function with bounded derivative. Moreover, for each $i\in \LL$,  there exits a function $L_\chi(i,\cdot), M_\chi(i,\cdot) \in L^1(\nu)\cap L^2(\nu)$ such that for all $x,x'\in \rr^d$ and $\xi\in \Upsilon$,
\bea
\|\chi(x,i,\xi)-\chi(x',i,\xi)\| \ad \leq L_\chi(i, \xi) \|x-x'\|, \\
\|\chi(x,i,\xi)\| \ad \leq M_\chi(i,\xi) (1+\|x\|).
\eea

\noindent{(H5)} The function $L_\chi, M_\chi$ are in the class $\CH^\varsigma$ for some $\varsigma>0$, where $\CH^\varsigma$ is defined as 
\beq{CH-varthe}
\CH^\varsigma:=\bigg\{h: \LL\times \Upsilon \to \rr: \forall \; Z \in \CB(\Upsilon) \text{ with } \nu(Z)<\infty, \int_{Z} \exp(\varsigma |h(i,\xi)|) \nu(d\xi) <\infty, \forall\, i\in \LL \bigg\}.
\eeq

\begin{rem}
The condition (H4)-(H5) are motivated by \cite[Condition 2.6, p. 1733]{BDG16}. 
\end{rem}

We denote by $\bar X$, the solution of the following ODE 
\beq{ode-jump}
d\bar X(t) = \bar b(\bar X(t))dt+ \int_{\Upsilon} \bar \chi(\bar X(s), \xi) \nu(d\xi)dt,
\eeq
where $\bar b$ and $\bar \chi$ are defined as 
\bea\ad 
\bar b(x):=\sum_{i\in \LL} b(x,i) \mu_i^x, \quad \bar \chi(x,\xi) :=\sum_{i\in \LL} \chi(x,i,\xi)\mu_i^x, \quad \forall, x\in \rr^d, \xi\in \Upsilon.
\eea
Under (H1)$'$, (H2)-(H4), there exist unique solutions $X^\e$ and $\bar X$ that satisfy \eqref{jump} and \eqref{ode-jump}, respectively. Letting  $\e \to 0$, one can prove that $X^\e$ in \eqref{jump} converges to $\bar X$ in suitable sense. We refer to the work \cite{YY04} for this averaging result.

Define $\eta^\e$ as in \eqref{eta}, we are interested in studying the large deviations of $\eta^\e$ as $\e \to 0$, which is the moderate deviations of $X^\e$ in \eqref{jump}. 

Let $\Phi$ be the solution of the Poisson equation of \eqref{pq}. For each fixed $x\in \rr^d$ and fixed $\xi\in \Upsilon$, let $\Psi$ be the solution of the Poisson equation
\beq{pq-chi}
Q(x)\Psi(x,\cdot,\xi)(i)= -[\chi(x,i,\xi)-\bar\chi(x,\xi)], \quad \text{ with } \sum_{i=1}^{|\LL|} \Psi(x,i,\xi)\mu_i^x= 0.
\eeq
Under conditions (H4)-(H5), the Poisson equation \eqref{pq-chi} has a unique solution $\Psi$. Then, we would obtain the following moderate deviation principle for the solution of the jump diffusion $X^\e$ in \eqref{jump}. 
\begin{thm}\label{thm:mdp-jump}
Assume condition (H1)$'$, (H2)-(H5) hold. Then the solution $\{X^\e\}_{\e>0}$ of \eqref{jump} satisfies the moderate deviation principle on $C([0,T];\rr^d)$ with speed $h^{-2}(\e)$ and rate function $I_{\text{jump}}$ defined as 
\bea
I_{\text{jump}}(\eta)\ad = \inf_{(u,\wdh \psi, \wdt \psi, \pi)\in \CV(\eta)} \bigg\{\sum_{i\in \LL} \half \int_0^T \|u_i(s)\|^2 \pi_i(s)ds + \sum_{i\in \LL} \int_{[0,T]\times \Upsilon} \half 
|\wdh \psi_i(s,\xi)|^2 \pi_i(s) \nu(d\xi) ds \\
\aad \qquad\qquad\qquad\qquad \quad  +  \sum_{(i,j)\in \TT} \half \int_{[0,T]\times [0,\zeta]} |\wdt \psi_{ij}(s,z)|^2 \pi_i(s) \la_\zeta(dz) ds \bigg\},
\eea
where $\CV(\eta)$ is the collection of all $\{u=(u_i), \wdh\psi=(\wdh \psi_i), \wdt \psi=(\wdt \psi_{ij}), \pi=(\pi_i)\}$ belonging to 
\bea
\MM([0,T];\rr^d)^{|\LL|} \times (L^2([0,T]\times \Upsilon))^{|\LL|} \times (L^2([0,T]\times [0,\zeta]))^{|\TT|} \times \MM([0,T];\CP(\LL))
\eea
such that for each $i\in \LL$ and each $(i,j)\in \TT$, $\int_0^T \|u_i(s)\|^2 \pi_i(s)ds <\infty$,
\bea\disp
 \int_{[0,T]\times \Upsilon} |\wdh \psi_i(s,\xi)|^2 \pi_i(s) \nu(d\xi)ds <\infty, \quad
 \int_{[0,T]\times [0,\zeta]} |\wdt \psi_{ij}(s,z)|^2 \pi_i(s) \la_\zeta(dz) ds<\infty,
 \eea 
 and for each $i\in \LL$, 
 \bea
 \eta(t)\ad = \int_0^t \nabla \bar b (\bar X(s)) \eta(s)ds + \sum_{i\in \LL} \int_0^t \sg_i(\bar X(s)) u_i(s) \pi_i(s)ds \\
 \aad + \int_{[0,t]\times\Upsilon} \nabla_x \bar \chi(\bar X(s),\xi)\eta(s) \nu(d\xi)ds  + \sum_{i\in \LL}\int_{[0,t]\times \Upsilon}  \bar{\chi}(\bar X(s),\xi)\wdh \psi_i(s,\xi) \nu(d\xi)ds \\
 \aad + \sum_{(i,j)\in \TT} \int_{[0,t]\times [0,\zeta]} [\Phi(\bar X(s), j) - \Phi(\bar X(s),i)] \indi_{E_{ij}(\bar X(s))}(z) \wdt \psi_{ij}(s,z) \la_\zeta(dz)ds \\
 \aad + \int_{\Upsilon} \sum_{(i,j)\in \TT} \int_0^t \int_0^{\zeta} \big[\Psi(\bar X(s), j, \xi) -\Psi(\bar X(s), i, \xi)\big] \indi_{E_{ij}(\bar X(s))}(z)\wdt \psi_{ij}(s,z) \la_\zeta(dz) ds \nu(d\xi)
 \eea
 and for a.e. $s\in [0,T]$ and $j\in \LL$,
 \bea\ad
 \sum_{i\in\LL} \pi_i(s)\Ga_{ij}(\bar X(s))=0. 
\eea
\end{thm}
The rigorous proof of \thmref{thm:mdp-jump} is left to our subsequent work.

\subsection{ Countable state Markovian environment}
This paper establishes a moderate deviation principle for stochastic differential equations evolving in a fast-varying random environment modeled by a state-dependent regime switching process on a finite state space. A natural question is whether we could extend the result to the Markov chain with countable state space $\LL=\{1,2,\dots\}$. Such an  extension is highly non-trivial and would likely require arguments beyond the weak convergence method used in this work. Below we outline the primary obstacles that arise when attempting to generalize our results for the countable-state setting.

The most fundamental difficulty concerns the variational representation of a countable-state Markov chain via Poisson random measures (PRMs). When the Markov chain is countable, though we can still represent the Markov chain as an equation driven by PRMs as in \eqref{re-sde}, we instead have \textit{infinite many/countable} PRMs. To our  knowledge, there is no available variational representation for countable PRMs. The literature such as \cite{BCD13,BDG16,BDG18,BDM11} all considered models driven by finite many PRMs. This countable states make our rate function \eqref{rate} problematic. Specifically, for each $i\in \LL, (i,j)\in \TT$, under the conditions that
\bea
\int_0^T \|u_i(s)\|^2 \pi_i(s)ds<\infty, \quad \int_{[0,T]\times[0,\zeta]} |\psi_{ij}(s,z)|^2 \pi_i(s) \la_\zeta(dz)ds < \infty,
\eea
we cannot obtain the finiteness of the rate function in \eqref{rate} as we take infinite summation over indices $i\in \LL$ and $(i,j)\in \LL$, respectively. Thus, some new arguments are needed. Second, the finiteness of the state space of the Markov chain ensures the tightness of random measures defined in \eqref{Pie}. When the state space is countable,  however, additional conditions are needed to ensure tightness; see \cite[Section 6.10]{BD19} for details on this point. The third difficulty lies in the proof of large deviation upper and lower bound, both of which rely on the finiteness of the state space. For instance, estimates such as \eqref{up-e1}
use this finiteness directly. Moreover, in the proof of large deviation upper bound, the construction of $\pi^\dl$ yields estimates \eqref{diff-pi-dl}, which fails in the countable setting, thus requires a new approximation scheme.

In summary, extending our moderate deviation results from finite-state to countable-state Markov chains presents several challenges, with the lack of a variational representation for infinitely many PRMs being the most crucial obstacle. This problem remains open, and we hope to address it in future work.

\subsection{Beyond  Markovian environment}
Similar result can also be established for settings where the fast-varying random environment is modeled by certain jump diffusions in $\rr^{d_2}$, by adapting  the notion of viable pairs \cite{DS12}. For this direction, we consider the following slow-fast stochastic system 
\beq{stable}\left\{\barray
d X^\e(t)\ad =b(X^\e(t), Y^\e(t))dt+ \sqe \sg(X^\e(t), Y^\e(t)) dW(t), \quad X^\e(0)=x_0\in\rr^{d_1}\\
dY^\e(t)\ad =\frac{1}{\e} [-Y^\e(t)+ f(X_t^\e,Y_t^\e)]dt + \frac{1}{\e^{1/\alpha}} dL_t^1, \quad Y^\e(0)=y_0 \in \rr^{d_2},
\earray\right.\eeq
where $L_t^1$ is an isotropic $\al$-stable process with $\al\in(1,2)$ in $\rr^{d_2}$. The associated Poisson random measure for $L_t^1$ is defined by 
\bea\ad 
N((0,t]\times U):=\sum_{s\in (0,t]}\indi_{U}(L_s^1-L_{s-}^1),\quad U\in \CB(\rr^{d_2}\setminus \{0\}), \quad t>0.
\eea
By Le\'{v}y-It\^{o} decomposition \cite{Sat99}, we can write the equation of $Y^\e$ in \eqref{stable} as
\beq{Y-re}\barray
dY^\e(t)\ad =\frac{1}{\e}[-Y^\e(t)+f(X_t^\e, Y_t^\e)]dt + \frac{1}{\e^{1/\alpha}} \int_{|z|\leq 1} z \wdt N(ds, dz)+\frac{1}{\e^{1/\alpha}} \int_{|z|>1} z N(ds,dz) \\
\aad = \frac{1}{\e}[-Y^\e(t)+f(X_t^\e, Y_t^\e)]dt- \frac{1}{\e^{1/\alpha}}  \int_{|z|\leq 1} z \nu(dz)ds +\frac{1}{\e^{1/\alpha}} \int_{\rr^{d_2}} z N(ds, dz).
\earray\eeq
where $N(ds,dz)$ is a PRM on $[0,\infty)\times \rr^{d_2}$ with intensity $\nu(dz)ds=C_{d_2,\al}/|z|^{d_2+\al}dz ds$ with $C_{d_2,\al}$ being a normalization constant and $\wdt N(ds, dz)= N(ds, dz)-\nu(dz)ds$ is the corresponding compensated PRM. The slow dynamics $X^\e$ in \eqref{stable} can also be a jump diffusion like in \eqref{jump}. We here focus on diffusions for simplicity. Under suitable conditions, as $\e\to 0$, $X^\e$ converges weakly to the following deterministic ODE 
\bea
d\bar X(t)= \bar b(\bar X(t))dt, \quad \bar X(0)=x_0,
\eea
where $\bar b$ is given by $
\bar b(x)=\int_{\rr^{d_2}} b(x,y)\mu^x(dy)$ with $\mu^x(dy)$ being the invariant measure of following SDE: for frozen $x\in \rr^{d_1}$,
\bea
dY^x(t)= [- Y^x(t)+f(x,Y^x(t))]dt + dL_t^1, \quad Y^x(0)=y_0\in \rr^{d_1}.
\eea
The operator associated with $Y^x$ is defined by $
\CL_0(x,y):=\CL_{\text{OU}}+ f(x,y)\cdot \nabla_y$, 
where $\CL_{\text{OU}}$ is the Ornstein-Uhlenbeck operator given by $\CL_{\text{OU}}:=\nabla_y^{\al/2}-y\cdot \nabla_y$.

Define $\eta^\e(t):=[X^\e(t)-\bar X(t)]/\sqe h(\e)$, where $h(\e)$ satisfies the condition \eqref{dev-h}. We aim to study the large deviation principle of $\eta^\e$. From \eqref{stable} and \eqref{Y-re}, the slow-fast system can be viewed as being driven by Brownian motions and Poisson random measures, allowing us to employ the weak convergence approach.

For slow-fast processes in Euclidean space, a powerful tool for large deviation analysis is to use the viable pair in \cite{DS12}. Denote by $L^2(\nu)$ the space of $L^2(\rr^{d_2})$ under measure $\nu(dz)$. For any fixed $x,\eta \in \rr^{d_1}, y\in \rr^{d_2}, u\in \rr^{d_1}$ and $\psi\in L^2(\nu)$, let us define 
\beq{viable}
\rho(x,\eta, y,u,\psi):= \nabla \bar{b}(x)\eta+ \sg(x,y) u \text{  and  } \CL_{u,\psi}^x:= \CL_0(x,y)
\eeq
Recall the definition of a viable pair in \cite{DS12}.
\begin{defn}
A pair $(\xi,P)\in C([0,T];\rr^{d_1}) \times \CP(\rr^{d_1} \times L^2(\nu) \times \rr^{d_2} \times [0,T])$ is called a viable pair with respect to $(\rho,\CL_{u,\psi}^x)$ if the followings are satisfied: the function $\xi_t$ is absolutely continuous, $P$ is square integrable in the sense that
\bea\ad
\int_{\rr^{d_1}\times L^2(\nu)\times \rr^{d_2}\times [0,T]} \|u(s)\|^2 + |\psi(s)|_{L^2(\nu)}^2 P(du\, d\psi\, dy\, ds)<\infty
\eea
and the following hold for all $t\in [0,T]$,
\beq{lim-xi}
\xi_t= x_0 + \int_{\rr^{d_1}\times L^2(\nu) \times \rr^{d_2} \times [0,t]} \rho(\bar X(s),\xi_s, y, u,\psi) P(du\,d\psi\,dy\,ds)
\eeq
and for every $h\in \CD(\CL_{u,\psi}^x)$, the domain of opertor $\CL_{u,\psi}^x$, we have 
\beq{fast-cond-1}
\int_{\rr^{d_1}\times L^2(\nu) \times \rr^{d_2}\times [0,t]} \CL_{u,\psi}^{\bar X(s)} h(y) P(du\,d\psi\,dy\,ds)=0.
\eeq
Moreover, $P(\rr^{d_1}\times L^2(\nu) \times \rr^{d_2} \times [0,t])=t$. In this case, we will write $(\xi,P)\in \CV_{\rho,\CL_{u,\psi}^{\bar X}}$.
\end{defn}
\begin{rem}
In the definition of \eqref{viable}, although the right-hand side of expressions does not depend on the control $\psi$, we still display it to highlight that  $\psi$ is a control variable. As in \thmref{thm:main}, we would expect $\psi(s,z)\in L^2([0,T]\times \rr^{d_2})$ with respect to the measure $\nu(dz)ds$ in our cost functional. We would encode $\psi(s,z)$ as a measure-valued process $\psi(s)$ taking values in the space $L^2(\nu)$.
\end{rem}

Consider the following non-local Poisson equation 
\beq{non-local-Poi}
\CL_0(x,y) \Phi(x,y) = - [b(x,y)-\bar b(x)], \quad x\in \rr^{d_1}
\eeq
where $x$ is a parameter. Motivated by the regularity result of above non-local Poisson equation in \cite{RXX24}, we could obtain the following result.
\begin{thm}
Assume assumption (H1)$'$ hold and $f$ is bounded, Lipschitz continuous in $\rr^{d_1}\times \rr^{d_2}$. Then the process $\{X^\e\}_{\e>0}$ of \eqref{stable} satisfies the moderate deviation principle in $C([0,T];\rr^{d_1})$ with speed $h^{-2}(\e)$ and good rate function $I_{\text{stable}}$ defined by
\bea\ad 
I_{\text{stable}}(\eta)= \inf_{(\eta,P) \in \CV_{\rho,\CL_{u,\psi}^{\bar X}}}\bigg\{\half\int_{\rr^{d_1}\times L^2(\nu)\times \rr^{d_2}\times [0,T]} \|u(s)\|^2 +|\psi(s)|_{L^2(\nu)}^2 P(du\, d\psi\,dy\, ds)\bigg\},
\eea
with the convention that the infimum over the empty set is $\infty$.

\end{thm}
\begin{proof}
The rigorous proof is omitted and will be presented in our subsequent work. Here, we provide a brief justification explaining why $\rho$ and $\CL_{u,\psi}^x$, as defined in \eqref{viable}, are independent of $\psi$ and the solution to the Poisson equation \eqref{non-local-Poi}. This independence extends to their appearance in the limiting dynamics \eqref{lim-xi} and the constraint on the fast process \eqref{fast-cond-1}.

In line with \secref{sec:pre}, we therefore introduce the following controlled dynamics: 
\bea
d\Xe(t)\ad =b(\Xe(t),\Ye(t))dt+ \sqe \sg(X^\e(t), Y^\e(t))dW(t)\\
\aad \quad +\sqe h(\e)\sg(\Xe(t), \Ye(t)) u^\e(t)dt \\ d\Ye(t)\ad = \frac{1}{\e} [-\Ye(t)+f(\Xe(t), \Ye(t))]dt\\
\aad\quad - \frac{1}{\e^{1/\alpha}} \int_{|z|\leq 1} z\nu(dz)ds + \frac{1}{\e^{1/\alpha}} \int_{\rr^{d_2}} z N^{\phi^\e}(ds,dz).
\eea
where $N^{\phi^\e}$ is controlled PRM with intensity $\phi^\e(s,z)\nu(dz)ds$. Denote by $\wdh X^\e=\Xe$ and $\wdh Y^\e=\Ye$. Note that the dynamics $\wdh Y^\e$ can be rewritten as 
\bea
d\wdh Y^\e(t) \ad = \frac{1}{\e}\big[-\wdh Y^\e+f(\wdh X^\e(t),\wdh Y^\e(t))\big] dt + \frac{1}{\e^{1/\alpha}}dL_t^1 \\
\aad \quad +\frac{1}{\e^{1/\alpha}} \Big\{\int_{\rr^{d_2}} z N^{\phi^\e}(ds,dz)- \int_{\rr^{d_2}} z N(ds,dz)\Big\}.
\eea
Applying It\^{o} formula to  the solution of nonlocal Poisson equations \eqref{non-local-Poi}, we have
\beq{Ito-phi-nonlocal}\barray\ad 
\Phi(\wdh X^\e(t), \wdh Y^\e(t))\\
\aad = \Phi(x_0,y_0)+\int_0^t \Big(\frac{1}{\e}\CL_0+\CL_1^\e \Big) \Phi(\wdh X^\e(s),\wdh Y^\e(s))ds + \CM_0^\e(t)+\CM_1^\e(t)+\CM_2^\e(t),
\earray\eeq
where $\CL_1^\e$ is defined as 
\bea\ad 
\CL_1^\e := \Big[b(x,y)+\sqe h(\e)\sg(x,y)u\Big] \cdot \nabla_x + \frac{\e}{2} \text{Tr}\big[\sg\sg^\top(x,y)\nabla_x^2 \big]
\eea
and for $i=0,1,2$, $\CM_i^\e(t)$ are defined by 
\bea
\CM_0^\e(t)\ad :=\int_0^t \int_{\rr^{d_2}} \Phi(\wdh X^\e(s),\wdh Y^\e(s)+\e^{-1/\al}z)-\Phi(\wdh X^\e(s), \wdh Y^\e(s))[\phi^\e(s,z)-1] \nu(dz)ds\\
\CM_1^\e(t)\ad = \sqe \int_0^t \nabla_x \Phi(\wdh X^\e(s), \wdh Y^\e(s)) \sg(\wdh X^\e(s), \wdh Y^\e(s))dW(s) \\
\CM_2^\e(t)\ad = \int_0^t \int_{\rr^{d_2}} \Big[\Phi(\wdh X^\e(s),\wdh Y^\e(s)+\e^{-1/\al}z)-\Phi(\wdh X^\e(s), \wdh Y^\e(s)) \Big] \wdt N^{\phi^\e}(ds,dz).
\eea

We then define $\etae=(\Xe-\bar X)/\sqe h(\e)$. As in the decomposition \eqref{b-eta}, we must rewrite the term $\etae_1$ defined in \eqref{b-eta} using the solution of the Poisson equation, following the same steps as in \eqref{Ito-Phi}, \eqref{eta1-1} and \eqref{eta1}. To do this, we multiply both sides of \eqref{Ito-phi-nonlocal} by $\e$, divide by $\sqe h(\e)$ and rearranging the resulting terms in analogy with \eqref{eta1-2}. The  only term require additional attention is $\CM_0^\e$; all other terms converge to zero in suitable sense. Setting $\psi^\e=(\phi^\e-1)/\sqe h(\e)$, the mean-value theorem yields
\bea\disp 
\frac{\e}{\sqe h(\e)}\CM_0^\e(t) \ad = \e \int_0^t \int_{\rr^{d_2}} \Big[ \Phi(\wdh X^\e(s),\wdh Y^\e(s)+\e^{-1/\al}z)-\Phi(\wdh X^\e(s), \wdh Y^\e(s))\Big] \psi^\e(s,z)\nu(dz)ds \\
\aad = \e \int_0^t \int_{\rr^{d_2}} \nabla_y \Phi(\wdh X^\e(s), \wdh Y^{+,\e}(s))\frac{z}{\e^{1/\al}}\psi^\e(s,z) \nu(dz)ds
\eea
for some $\wdh Y^{+,\e}(s)$ lying between $\wdh Y^\e(s)$ and $\wdh Y^\e(s)+\e^{-1/\al}z$. 
Thanks to $\e^{1-1/\al}\to 0$ for $\al\in (1,2)$ as $\e\to 0$ and the boundedness of $\Phi$, we can obtain
$
\e/(\sqe h(\e)) \CM_0^\e(t) \to 0, \quad \text{as } \e \to 0$. 
This establishes that the limiting dynamics of $\etae$ are independent of $\psi$ and $\Phi$. Analogously, it can be verified that the constraint condition \eqref{fast-cond-1} is also independent of them.
\end{proof}

\medskip

\textbf{Acknowledgments.} 
The author would like to thank two anonymous referees and the editors for their insightful comments and suggestions.

\bibliographystyle{plain}



\end{document}